\documentclass[11pt]{amsart}

\usepackage{amsmath, amssymb, amsthm,mathtools,tikz,tikz-cd,graphicx,color,stmaryrd,hyperref,enumerate,mathdots, mathrsfs,dynkin-diagrams}
\usepackage[new]{old-arrows}
\usepackage{colonequals}
\usepackage[shortlabels]{enumitem}

\usepackage[utf8]{inputenc}

\usepackage[margin=1in]{geometry}
\usepackage{cleveref}

\usepackage[english]{babel}

\hypersetup{
    colorlinks,
    citecolor=blue,
    filecolor=blue,
    linkcolor=blue,
    urlcolor=blue
}

\newcommand{\C}{{\mathbb C}}
\newcommand{\bO}{{\mathbb O}}

\newcommand{\cH}{\mathcal{H}}
\newcommand{\hatW}{\widehat{W}}

\newcommand{\fg}{\mathfrak{g}}
\newcommand{\fh}{\mathfrak{h}}
\newcommand{\fn}{\mathfrak{n}}
\newcommand{\fb}{\mathfrak{b}}
\newcommand{\G}{\mathbb{G}}
\newcommand{\bA}{\mathbf{A}}

\newcommand{\ch}{\operatorname{ch}}
\newcommand{\End}{\operatorname{End}}
\newcommand{\sgn}{\operatorname{sgn}}
\newcommand{\Pic}{\operatorname{Pic}}
\newcommand{\Ind}{\operatorname{Ind}}
\newcommand{\subreg}{\mathrm{subreg}}
\newcommand{\reg}{\mathrm{reg}}
\newcommand{\bounded}{\mathrm{b}}
\newcommand{\cF}{\mathcal{F}}
\newcommand{\std}{\operatorname{std}}
\newcommand{\bP}{\mathbf{P}}
\newcommand{\R}{{\mathbb R}}
\newcommand{\Z}{{\mathbb Z}}

\newcommand{\cV}{{\mathcal V}}
\newcommand{\cO}{{\mathcal O}}
\newcommand{\cD}{{\mathcal D}}

\newcommand{\cN}{{\mathcal N}}

\newcommand{\ft}{{\mathfrak{t}}}

\newcommand{\cE}{\mathcal{E}}
\newcommand{\Sym}{\mathrm{Sym}}

\newcommand{\Irr}{\operatorname{Irr}}
\newcommand{\Coh}{\operatorname{Coh}}
\newcommand{\Rep}{\operatorname{Rep}}

\newcommand{\fp}{\mathfrak{p}}

\newcommand{\SL}{{\mathrm{SL}}}
\newcommand{\fsl}{{\mathfrak{sl}}}

\newcommand{\Sp}{{\mathrm{Sp}}}
\newcommand{\GL}{{\mathrm{GL}}}
\newcommand{\SO}{{\mathrm{SO}}}
\newcommand{\cB}{{\mathscr{B}}}
\newcommand{\cP}{{\mathscr{P}}}

\newcommand{\tO}{{\mathrm{O}}}

\newcommand{\PGL}{{\mathrm{PGL}}}
\newcommand{\Cent}{{\mathrm{Z}}}

\newcommand\iso{\,\vphantom{j^{X^2}}\smash{\overset{\sim}{\vphantom{\rule{0pt}{0.20em}}\smash{\longrightarrow}}}\,}

\DeclareMathOperator{\tr}{tr}

\newtheorem{thm}{Theorem}[section]

\newtheorem{lemma}[thm]{Lemma}
\newtheorem{conjecture}[thm]{Conjecture}
\newtheorem{cor}[thm]{Corollary}
\newtheorem{prop}[thm]{Proposition}

\theoremstyle{definition}
\newtheorem{remark}[thm]{Remark}

\newtheorem{defn}[thm]{Definition}
\newtheorem{example}[thm]{Example}
\newtheorem{prop-def}[thm]{Proposition-Definition}

\title[subregular Kazhdan-Lusztig polynomials]{Affine Kazhdan-Lusztig polynomials on the subregular cell in non simply-laced Lie algebras\\ with an application to character formulae}

\makeatletter
\let\@wraptoccontribs\wraptoccontribs
\makeatother

\author{Vasily Krylov}
\address{Department of Mathematics
Harvard University and CMSA
\newline
1 Oxford Street,
Cambridge, MA 02138,
USA}
\email{vkrylov@math.harvard.edu, vasily@cmsa.fas.harvard.edu, krylovasya@gmail.com} 

\author{Kenta Suzuki}
\address{M.I.T., 77 Massachusetts Avenue,
Cambridge, MA, USA}
\email{kjsuzuki@mit.edu}

\contrib[with an appendix by]{Roman Bezrukavnikov, Vasily Krylov, and Kenta Suzuki}

\begin{document}

\begin{abstract}
We extend the techniques in \cite{BKK} to the non-simply-laced situation, and calculate explicit special values of parabolic affine inverse Kazhdan-Lusztig polynomials for subregular nilpotent orbits. We thus obtain  explicit character formulas for certain irreducible representations of affine Lie algebras. As particular cases, we compute characters of simple vertex algebras $V_{k}(\mathfrak{g})$ for $k=-1,\ldots,-b$, where $b$ is the largest label of the highest short coroot $\theta^\vee$. Conjecturally,  all ordinary modules over $V_{-b}(\mathfrak{g})$ are covered by our computations. As an application, we obtain the explicit formulas for flavoured Schur indices of rank one Argyres–Douglas 4d SCFTs with flavour symmetry $G_2$ and $B_3$. Our results are proved using the geometry of the Springer resolution. We identify the cell quotient of the anti-spherical module over $\widehat{W}$ corresponding to the subregular cell with a certain one-dimensional extension of a module defined by Lusztig in \cite{lusztig-some-examples}. We describe the canonical basis in this module geometrically and present an explicit description of the corresponding objects in the derived category of equivariant coherent sheaves on the Springer resolution. They correspond to irreducible objects in the heart of a certain $t$-structure that we describe using an equivariant version of the derived McKay correspondence.  
\end{abstract}
\maketitle

\tableofcontents

\section{Introduction}

Let $G$ be a simple, simply-connected reductive group with Lie algebra $\fg$, and let $\widehat\fg$ be the associated affine Lie algebra. Let $\widehat W=Q^\vee\rtimes W$ be the affine Weyl group, which acts (denoted by $\circ$) on the dual of the Cartan $\widehat\fh^*$ by the shifted action by the sum of the fundamental weights $\widehat\rho$. For each character $\Lambda\in\widehat\fh^*$, we may define two $\widehat\fg$-modules: $M(\Lambda)$, the Verma module, and $L(\Lambda)$, the unique irreducible quotient of $M(\Lambda)$ (see \S\ref{sec:rep-theory-review} for details). Although by construction the character of $M(\Lambda)$ is easy to describe, the character of $L(\Lambda)$ is much more complicated to describe. For $\kappa>-h^\vee$ where $h^\vee$ is the dual Coxeter number, let $\cO_\kappa$ be the category $\cO$ of $\widehat\fg$-modules of level $\kappa$ with some finiteness conditions, as in \cite{kashiwara-tanisaki2} (see Definition~\ref{defn:category-o}). For $\lambda\in\widehat\fh^*$ regular (i.e., the stabilizer of the $\widehat W$-action is trivial) with integer level $\kappa=\kappa(\lambda)>-h^\vee$ and $v\in\widehat W$, there exists an equality in the $K$-theory of the category $\cO_\kappa$ \cite[Thm~1.1]{kashiwara-tanisaki}:
\[
[L(v^{-1}\circ\lambda)]=\sum_{w\in \widehat W}\epsilon(wv^{-1})\mathbf m_v^w[M(w^{-1}\circ \lambda)],
\]
where $\epsilon\colon\widehat W\to \{\pm1\}$ is the sign, and $\mathbf m_v^w\colonequals\mathbf m_v^w(1)$ are special values of the inverse parabolic Kazhdan-Lusztig polynomials. Thus, calculating the character of $L(\Lambda)$ is essentially equivalent to calculating $\mathbf m_v^w$. As in \cite[Definition~3.7]{BKK}, we define:
\begin{defn}Let $\widehat W^f$ be the set of $v\in \widehat W$ such that $v$ is the shortest element in $vW$. There is a bijection $\widehat W^f\to\widehat W/W\simeq Q^\vee$, and we let $\nu\mapsto w_\nu$ be the inverse bijection.
\end{defn}

In \cite{kazhdan-lusztig}, Kazhdan and Lusztig introduced a notion of a two-sided cell in $\widehat{W}$; Lusztig then proved in \cite{lusztig-cells4} that there exists the bijection between two-sided cells in $\widehat{W}$ and the set of nilpotent $G^\vee$-orbits.

\begin{defn}\label{subregular-cell-definition}
    Let $c_\subreg\subset\widehat W$ be the two-sided cell, which corresponds to the subregular orbit. Let $c^0_\subreg\colonequals\widehat W^f\cap c_\subreg$, which is a left cell in $\widehat W$.
\end{defn}

\subsection{Main results of the paper} We calculate $\mathbf m_v^w$ when $\fg$ is not simply-laced and $v\in c^0_\subreg$, where $c^0_\subreg\subset\widehat W$ is the left cell defined above (the case of simply-laced $\mathfrak{g}$ was treated in \cite{BKK}). This allows us to prove the following analog of \cite[Thm~2.9]{BKK}:

\begin{thm}\label{thm_main_desc_m}
Let $\fg$ be of type $B_n$ ($n \geqslant 3$), $C_n$ ($n \geqslant 2$), or $F_4$, $G_2$. Let $w\in c^0_\subreg$ and let $\lambda\in\widehat{\fh}^*$ be such that $\Lambda=w^{-1}(\lambda)$ is quasi-dominant, where $w$ is the longest element in the coset $\widehat W_\lambda w$. Then
\[
\widehat R\ch L(\Lambda)=\sum_{u\in W}\epsilon(uw)\sum_{\gamma\in Q^\vee}\mathbf m^{w_\gamma}_{w} e^{ut_\gamma w(\Lambda+\widehat\rho)},
\]
where $m^{w_\gamma}_w$ is given in Theorems~\ref{KL-polynomial-computation-B}, \ref{thm_KL_type_C}, \ref{thm_KL_type_F}, and \ref{thm_KL_type_G}.
\end{thm}

The following Lemma describes all possible $\Lambda$'s that can appear in Theorem \ref{thm_main_desc_m} and is the direct generalization of \cite[Lemma 2.11 and Lemma 2.17]{BKK}. Recall that $c_{\mathrm{subreg}} \subset \widehat{W}$ consists of the elements $w \in \widehat{W} \setminus \{1\}$ that have {\emph{unique}} minimal decomposition. For such an element $w$, we denote by $\mu(w) \in \widehat{S}$ the first term of the minimal decomposition of $w$.
\begin{lemma}\label{lem:descr_poss_lambda}
Elements $\lambda \in \widehat{\mathfrak{h}}^*$, $w \in c_{\mathrm{subreg}}^0$, and $\Lambda=w^{-1} \circ \lambda$ satisfying the conditions in Theorem \ref{thm_main_desc_m} are one of the following:
\begin{enumerate}[(a)]
    \item $\lambda+\widehat{\rho}$ is regular dominant integral and $w$ is an arbitrary element of $c^0_{\mathrm{subreg}}$; or
    \item $\lambda=-\Lambda_i+\sum_{k \neq i} m_k\Lambda_k+x\delta$ for $i=\mu(w)$, $x \in \mathbb{C}$, and $m_k \in \mathbb{Z}_{\geqslant 0}$.
\end{enumerate} 
\end{lemma}
\begin{proof}
Same as the proof of \cite[Lemma 2.11]{BKK}.
\end{proof}

\begin{remark}
Note that the level $k:=\Lambda(K)$ of $L(\Lambda)$ appearing in Lemma \ref{lem:descr_poss_lambda} is $\geqslant -b$, where $b$ is the maximal among the coefficients $a_i$ in $\theta^\vee=\sum_{i}a_i\alpha_i^{\vee}$. Furthermore, $\Lambda=k\Lambda_0$ always satisfies the conditions of Lemma  \ref{lem:descr_poss_lambda} for $k=-1,\ldots,-b$.    
\end{remark}

\begin{remark}
For example, $\Lambda=-\Lambda_0$ always satisfies the conditions of Lemma \ref{lem:descr_poss_lambda} (for $\lambda=-\Lambda_0$, and $w=s_0 \in c^0_{\mathrm{subreg}}$). Theorem \ref{thm_main_desc_m} then provides an explicit formula for the character of $L(-\Lambda_0)$. To illustrate, for $\mathfrak{g}=G_2$ we get (see Example \ref{char_G_2_lambda_0}):
\begin{equation}\label{intro_G_2_1}
\widehat{R}\operatorname{ch}L(-\Lambda_0)=-\sum_{\gamma=m\alpha^\vee+n\beta^\vee}\sum_{u \in W}\epsilon(u)\Big(\Big\lfloor\frac{m-1}3\Big\rfloor+\frac{1}{4}\!(m-\mathbf 1_{m\equiv n\bmod2}-\mathbf 1_{n\equiv0\bmod2})\Big)e^{ut_{-\gamma}(-\Lambda_0+\widehat{\rho})},
\end{equation}
where $\alpha^\vee$, $\beta^\vee$ are the simple coroots of $\mathfrak{g}$ (see Example \ref{type-g-highest}).
\end{remark}

\begin{remark}
    When $\fg$ is of type $G_2$ or $B_3$, $\Lambda=-2\Lambda_0$ also satisfies the conditions of Lemma~\ref{lem:descr_poss_lambda}. When $\fg=G_2$,
\begin{equation}\label{intro_G_2_2}
\widehat{R}\operatorname{ch}L(-2\Lambda_0)=-\sum_{\gamma=m\alpha^\vee+n\beta^\vee}\sum_{u \in W}\epsilon(u)\Big(\Big\lfloor\frac{m-1}3\Big\rfloor+\frac{1}{4}\!(m-\mathbf 1_{m\equiv n\bmod2}-\mathbf 1_{n\equiv0\bmod2})\Big)e^{ut_{-\gamma}(-2\Lambda_0+\widehat{\rho})}.
\end{equation}
When $\fg$ is of type $B_3$, we have (see Example~\ref{char_B_3_2lambda_0})
\begin{equation}\label{intro_B_3_2}
\widehat R\ch L(-2\Lambda_0)=-\sum_{\gamma=a_1\alpha_1^\vee+a_2\alpha_2^\vee+b\beta^\vee}\sum_{u\in W}\epsilon(u)\bigg(\frac14(a_1-1)-\frac14(-1)^{a_1}\mathbf 1_{a_2\not\equiv0\bmod2}\bigg)e^{ut_{-\gamma}(-2\Lambda_0+\widehat\rho)},
\end{equation}
where $\alpha^\vee$, $\beta^\vee$ are the simple corrots of $\fg$ (see Example~\ref{type-b-highest}).
\end{remark}

To calculate, we use the fact that the $\mathbf m_v^w$ are realized in the $Q^\vee\subset \widehat W$-action on $K^{G^\vee}\!(\widetilde{\mathcal{N}})$, where $\pi\colon \widetilde\cN\to\cN$ is the Springer resolution of the nilpotent cone in $\fg^\vee$; explicitly, $K^{G^\vee}\!(\widetilde\cN)$ has the standard basis $\{[\cO_{\widetilde\cN}(\gamma)]:\gamma\in Q^\vee\}$ and the canonical basis, i.e., the irreducible objects of the heart of the exotic $t$-structure on $\cD^\bounded(\Coh^{G^\vee}\!(\widetilde\cN))$, see \cite{bez1} and \cite{bez2}. Moreover, letting $U\colonequals\bO_\subreg\cup\bO_\reg\subset\cN$ be an open subset consisting of the regular and subregular nilpotent orbits, on which we have the restricted Springer resolution $\widetilde U\to U$, there is a surjection $K^{G^\vee}\!(\widetilde\cN) \twoheadrightarrow K^{G^\vee}\!(\widetilde U)$. Moreover, $K^{G^\vee}\!(\widetilde U)$ has a canonical basis $\{\overline C_w:w\in c_\subreg^0 \cup \{1\}\}$ consisting of classes of the irreducible objects of the heart of the exotic $t$-structure on $\cD^\bounded(\Coh^{G^\vee}\!(\widetilde U))$ (see Section \ref{sec:explicit-description-of-the-canonical-basis}). In Appendix~\S\ref{appendix} we 
explicitly characterize the irreducible objects by 
proving an equivariant version of the derived McKay correspondence, and observing that the standard $t$-structure on the derived category of double quiver representations transfers to the exotic $t$-structure.  
    We give an explicit description of these irreducible exotic coherent sheaves in Section \ref{sec:explicit-description-of-the-canonical-basis}:
\begin{thm}
    The irreducible exotic coherent sheaves of $\cD^\bounded(\Coh^{G^\vee}\!(\widetilde U))$ is described in Propositions \ref{lem:canonical-basis-b}, \ref{lem:canonical-basis-c}, \ref{lem:canonical-basis-f}, and \ref{lem:canonical-basis-g}.
\end{thm}

When $w=s_0$, we can go further, and describe the irreducible object corresponding to $s_0$ in $\cD^\bounded(\Coh^{G^\vee}\!(\widetilde\cN))$, not just its restriction to $\widetilde U$:

\begin{prop}[Proposition~\ref{prop:Cs0}]
$\cO_{\widetilde{\overline\bO}_\subreg}[-1]\in \cD^\bounded(\Coh^{G^\vee}\!(\widetilde\cN))$ is an irreducible exotic coherent sheaf corresponding to $s_0$.
\end{prop}

Now, $c_\subreg^0$ is a left cell, so by \cite[Proposition~3.8]{lusztig-some-examples} there is a $\hatW$-module $E_\subreg^0$ with basis $(e_w)_{w\in c_\subreg^0}$.
In \S\ref{sec:extending-lusztig} we define an extension $\widetilde E_\subreg^0=E_\subreg^0 \oplus \mathbb{Z} e_1$ of $E_\subreg^0$, which fits into an exact sequence
\[
0\to E_\subreg^0\to\widetilde E_\subreg^0\to \Z\to 0,
\]
where $\Z$ is the trivial representation of $\hatW$. Let $\Z_{\sgn}$ be the sign representation of $\hatW$.

\begin{prop}\label{main-thm}
There is an isomorphism of exact sequences of $\widehat W$-modules
\[ \begin{tikzcd}
0\arrow{r}&E^0_\subreg\otimes\Z_{\sgn}\arrow{r}\arrow{d}{\simeq}&\widetilde E^0_\subreg \otimes\Z_{\sgn}\arrow{r} \arrow{d}{\simeq} &   \Z_{\sgn}\arrow{r}\arrow{d}{\simeq}&0\\
0\arrow{r}&K^{G^\vee}\!(\widetilde\bO_\subreg) \arrow{r}&K^{G^\vee}\!(\widetilde U)\arrow{r}& K^{G^\vee}\!(\widetilde\bO_\reg)\arrow{r}&0,
\end{tikzcd}
\]
where the middle isomorphism sends $e_w$ to $C_w$ for $w\in c_\reg\cup c^0_\subreg$.
\end{prop}
\begin{remark}
    In fact, all of the modules in Proposition~\ref{main-thm} can be extended to $\cH_q(\hatW)$-modules (see also Proposition~\ref{prop:iso-E-and-H}), and the isomorphisms extend. 
\end{remark}
\begin{remark}
    In \S\ref{sec:cartan-of-kac-moody}, we use Proposition~\ref{main-thm} to prove a conjecture of \cite{bezr-kar-kr} and provide a description of $K^{G^\vee}(\widetilde U)$ using the Cartan of Kac-Moody algebras.
\end{remark}

Thus, calculating $\mathbf m_w^{w_\gamma}$ for such elements reduces to calculating the class of $\cO_{\widetilde U}(\gamma)$ in $K^{G^\vee}\!(\widetilde U)$, as an explicit linear combination of the canonical basis. 

\subsubsection{Main result of the Appendix \ref{appendix_C}} Let us finish this section by stating the last main result of this text obtained in Appendix \ref{appendix_C} written jointly with Roman Bezrukavnikov. 

\begin{defn} An element $e\in \mathcal{N}$ is called {\em{distinguished}} if 
it is not contained in a proper Levi subalgebra.
An element $e\in \mathcal{N}$ is called {\em very distinguished} if every element $e'\in \mathcal{N}$ such that $e \in \overline{\mathbb{O}}_{e'}$
is distinguished. 
\end{defn}

For example, subregular nilpotent $e \in \mathfrak{g}^\vee$ is distinguished iff $\mathfrak{g}$ is not of types $A_n$, $C_n$. In all of these cases, using the explicit formulas for ${\bf{m}}^{w_\gamma}_{w}$, $w \in c^0_{\mathrm{subreg}}$ obtained in the present paper (and in \cite{BKK}) we see that the function 
$
\gamma \mapsto {\bf{m}}^{w_\gamma}_{w}
$
is the sum of {\emph{quasi-polynomials}} of degree $2\operatorname{dim}\mathcal{B}_e=2$ and with periods equal to {\emph{orders}} of the elements in the finite group $\Cent_e$. In Appendix \ref{appendix_C} we generalize this result to the case of an {\em{arbitrary}} very distinguished nilpotent element.

Let $e \in \mathcal{N}$ be a very distinguished nilpotent element. 
Set $U:=\bigsqcup_{e \in \overline{\mathbb{O}}_{e'}}\mathbb{O}_{e'} \subset \mathcal{N}$. This is an open $G^\vee$-invariant subset of $\mathcal{N}$. Note that $K^{G^\vee}(\widetilde{U})$ is finite dimensional (use that $e$ is very distinguished).
Pick $\nu \in Q^\vee$ such that $w_\nu$ lies in the two-sided cell. Using the localization theorem in equivariant $K$-theory obtained by Edidin and Graham in \cite{edidin-graham}, we prove the following theorem.
\begin{thm}
[see Theorem \ref{quasi_pol_dist}]
The function $\gamma\mapsto {\bf{m}}_{w_\nu}^{w_\gamma}$ is a quasi-polynomial, and we can estimate its
period and degree.
\end{thm}

\subsection{Applications to representations of vertex algebras and Schur indices of 4D SCFT's}
For $k \in \mathbb{C}$ consider the vacuum $\widehat{\mathfrak{g}}$-module $V^k(\mathfrak{g})$ of level $k$. It is known that $V^k(\mathfrak{g})$ has a natural vertex algebra structure. Let $V_k(\mathfrak{g})$ be the unique nonzero irreducible quotient of $V^k(\mathfrak{g})$. It is known that as $\widehat{\mathfrak{g}}$-module, we have  $V_k(\mathfrak{g}) \simeq L(k\Lambda_0)$. It makes sense to talk about the category $\mathcal{O}$ for $V_k(\mathfrak{g})$. We propose the following conjecture, joint with Victor Kac. Recall that $b$ is the maximal among the coefficients $a_i$ in $\theta^\vee=\sum_{i}a_i\alpha_i^\vee$.

\begin{conjecture}\label{conj_label_irred_O_vertex}
Irreducible objects $L(\Lambda) \in \mathcal{O}(V_{-b}(\mathfrak{g}))$ with integral $\Lambda$ are in bijection with the set $\{w \in c_{\mathrm{subreg}}\,|\,\mu(w)=s_i\,\text{s.t.}\,\Lambda_i(K)=b\}$. The bijection is explicitly given by 
\begin{equation*}
w \mapsto \Lambda=w^{-1} \circ (-\Lambda_i).
\end{equation*}
In particular, for $\mathfrak{g}$ s.t. $Z_{e}$ is finite, the irreducible objects as above are in bijection with a certain {\emph{right}} Lusztig cell in the affine Weyl group. 
\end{conjecture}

We have checked that  Conjecture \ref{conj_label_irred_O_vertex} is correct in all cases when the set $\operatorname{Irr}\mathcal{O}(V_{-b}(\mathfrak{g}))$ is  described explicitly (see \cite{adamovic-perse} for type $A$, \cite{arakawa-moreau} for types $A_2, D_4, E_6, E_7, E_8$ and \cite{arakawa-dai-fasquel-li-moreau} for types $B_3$, $G_2$). As was observed in \cite{Bohan-Xie}, vertex algebras $V_{-2}(\mathfrak{g})$ for $\mathfrak{g}$ of types $B_3$ and $G_2$ correspond to certain rank one Argyres–Douglas theories in four dimension with flavour symmetries $B_3$ and $G_2$ respectively. So, our formulas \eqref{intro_G_2_2} and \eqref{intro_B_3_2} provide an explicit answer for the flavored Schur indices of these theories.

\begin{remark} 
In Conjecture \ref{conj_label_irred_O_vertex} above we only consider simple vertex algebra $V_{k}(\mathfrak{g})$ for $k=-b$. One can consider arbitrary $k \in \{-1,\ldots,-b\}$. We do not know if it is natural to expect that every integral $\Lambda$ such that $L(\Lambda) \in \mathcal{O}(V_{k}(\mathfrak{g}))$ is as in Lemma \ref{lem:descr_poss_lambda} above.
\end{remark}

It is expected that the condition that $Z_{e}$ is finite, corresponds to $V_{-b}(\mathfrak{g})$ being {\emph{quasi-lisse}} (see \cite{arakawa-kawasetsu}), i.e., the associated variety  $\operatorname{Ass}(V_{-b}(\mathfrak{g}))$ should be contained in the nilpotent cone of $\mathfrak{g}$. 

From now on, assume that $Z_e$ is finite. The following conjecture was communicated to us by Peng Shan. 
\begin{conjecture}\label{conj_ass_BVLS}
Let $D\colon \{\mathcal{N}_{\mathfrak{g}}/G\} \rightarrow \{\mathcal{N}_{\mathfrak{g}^\vee}/G^\vee\}$ be the Barbasch-Vogan-Lusztig-Spaltenstein duality. 
Then 
$
\operatorname{Ass}(V_{-b}(\mathfrak{g}))=\overline{D(\mathbb{O}_{\mathrm{subreg}})}.
$ 
\end{conjecture}
It would be very interesting to figure out if there is any precise relation between the two conjectures above. Note that the subregular {\emph{right}} cell appears in Conjecture \ref{conj_label_irred_O_vertex} and parametrizes irreducible objects of the corresponding category $\mathcal{O}$, while the subregular nilpotent orbit appears in Conjecture \ref{conj_ass_BVLS} and controls the associated variety of the corresponding vertex algebra; this should not be a coincidence. A similar relation is known to be true for $W$ (instead of $\widehat{W}$), namely, it is known (see \cite[Remark 1.6]{kazhdan-lusztig}) that the irreducible objects in the regular block of the category $\mathcal{O}$ for $\mathfrak{g}$ have the same annihilator iff their labels in $W$ lie in the same {\emph{left}} Lusztig cell (the replacement of right cell by left cell might have to do with the appearance of the BVLS duality in Conjecture \ref{conj_ass_BVLS}).
Finally, let us mention that the Conjecture \ref{conj_ass_BVLS} is known to be true in some cases (see \cite{arakawa-moreau}, \cite{arakawa-dai-fasquel-li-moreau}).

\subsubsection{Modular invariance vs quasi-polynomials}
Recall that in Appendix \ref{appendix_C} written jointly with Roman Bezrukavnikov, we prove that if $w_\nu \in \widehat{W}$ lies in the two-sided cell corresponding to an arbitrary  {\emph{very distinguished}} nilpotent element, then the coefficients $m^{w_\gamma}_{w_\nu}$ are {\emph{quasi-polynomials}} in $\gamma$ and we can estimate their
period and degree. So, the (normalized) characters of the corresponding irreducible $\widehat{\mathfrak{g}}$-modules might have nice modular invariance properties in the sense that they satisfy certain modular linear differential equations (cf. \cite{arakawa-kawasetsu} and references therein). 
It would be interesting to study them in some examples. For example, for $\mathfrak{g}= G_2$, we obtain explicit formulas for the characters of $L(-\Lambda_0)$, $L(-2\Lambda_0)$ (see \eqref{intro_G_2_1} and \eqref{intro_G_2_2}) and it would be very interesting to understand their modular invariance  properties. 
Note, for example, that the explicit character formulas obtained in \cite{BKK} were applied in \cite{kac-wakimoto2} to obtain modular invariance of (normalized) characters of certain $W$-algebras.

\subsection{Structure of the paper} In \S\ref{sec:prelim}, we cover preliminaries, on canonical bases, the Springer resolution, Slodowy varieties, and representations of affine Lie algebras. In \S\ref{sec:slodowy-slice}, we use the geometry of the subregular Slodowy variety to deduce a description of the irreducible exotic coherent sheaf corresponding to $s_0$. In \S\ref{sec:extending-lusztig} we review the definition of the subregular two-sided cell, and give and review the algebraic description of the subregular affine Hecke algebra given by \cite{xu}. In \S\ref{sec:explicit-description-of-the-canonical-basis} we explicitly characterize the irreducible exotic coherent sheaves in $D^\bounded(\Coh^{G^\vee}(\widetilde U))$ for $G$ non-simply-laced. In \S\ref{sec:proof-of-main-thm} we use the explicit description of the irreducible objects to give a geometric proof of Proposition~\ref{main-thm}, which describes $K^{G^\vee}(\widetilde U)$ combinatorially. In \S\ref{sec:cartan-of-kac-moody} we re-phrase the combinatorial description of $K^{G^\vee}(\widetilde U)$ using the Cartan of Kac-Moody algebras. In \S\ref{sec:KL-computation} we use our geometric description to compute Kazhdan-Lusztig polynomials on $c^0_\subreg$. Finaly in \S\ref{sec:character-computation} we apply our results to compute new character formulas for representations of affine Lie algebras. In Appendix~\ref{appendix} we re-cast our description of the irreducible exotic sheaves using an equivariant version of the derieved McKay correspondence. In Appendix~\ref{appendix_C}, written jointly with Roman Bezrukavnikov, we show a general result that, in particular, explains why the Kazhdan-Lusztig polynomials we compute in types B, D, E, F, and G are quasi-polynomial.

\subsection{Acknowledgements}
This project is supported by the MIT UROP+ program. The authors thank Roman Bezrukavnikov for numerous helpful suggestions. The fisrt author is also grateful to Ivan Karpov for the useful discussions of the material of the Appendix \ref{appendix_C}. Many thanks are also to Victor Kac,  George Lusztig and Peng Shan.
\section{Preliminaries}\label{sec:prelim}
\subsection{Notation}\label{sec:notation}
We fix the following notation:
\begin{itemize}
\item $G$ is a simple simply-connected reductive group with Lie algebra $\fg$. Let
 $G^\vee$ and $\fg^\vee$ be the dual Lie group and Lie algebra, respectively. We fix a pinning: $T\subset B_0\subset G$ and $T^\vee\subset B_0^\vee\subset G^\vee$. Let $ Q^\vee\colonequals\hom(T^\vee,\G_m)=\hom(\G_m,T)$ be the root lattice of $G^\vee$, and let $Q$ be the dual lattice. Any $\gamma\in Q^\vee$ viewed as an element of $\hatW$ is denoted as $t_\gamma$. Let $\Phi\subset Q$ be the root system for $G$, with simple roots $\Delta$, and let $\Phi^\vee$ be the root system for $G^\vee$, with simple roots $\Delta^\vee$. There is a canonical bijection $\alpha\in\Delta\mapsto\alpha^\vee\colonequals\big(\frac{2\alpha}{(\alpha,\alpha)},-\big)\in\Delta^\vee$.
 \item Let $W$ be the Weyl group of $G$ and $G^\vee$ (identified by $s_\alpha\mapsto s_{\alpha^\vee}$), and let $\widehat W\colonequals Q^\vee\rtimes W$ be the affine Weyl group of $G$. Let $\prec$ denote the Bruhat order on $\widehat W$. Let $S$ be the set of simple reflections of $W$, and let $\widehat S\colonequals S\cup\{s_0\}$ be the set of simple reflections of $\widehat W$. For $t_1,t_2\in\widehat S$, let $m_{t_1,t_2}$ denote the order of $t_1t_2$.
\item For each $\alpha\in\Delta$ let $P_{\alpha,0}^\vee\subset G^\vee$ be the parabolic subgroup generated by $s_\alpha$ and $B_0^\vee$ (and denote the Lie algebra by $\fp_{\alpha,0}^\vee$). This provides a bijection between $\Delta$ and conjugacy classes of sub-minimal parabolic subgroups of $G^\vee$.
\item $\cB=\cB_{\fg^\vee}\simeq G^\vee/B_0^\vee$, the \emph{flag variety}, is the set of Borel subalgebras of $\fg^\vee$. For each $\alpha\in\Delta$, $\cP_\alpha\simeq G^\vee/P_{\alpha,0}^\vee$ is the \emph{partial flag variety}, the set of parabolic subalgebras of $\fg^\vee$ conjugate to $\fp_\alpha^\vee$. There is a canonical projection $\pi_i\colon\cB\to\cP_i$, which sends $\fb^\vee$ to the unique parabolic algebra $\fp_\alpha^\vee\supset\fb^\vee$ conjugate to $\fp_{\alpha,0}^\vee$ (equivalently, the projection $G^\vee/B^\vee_0\to G^\vee/P_{\alpha,0}^\vee$). The fibers of $\pi_i$ are isomorphic to $P_{\alpha,0}^\vee/B_0^\vee\simeq\bP^1$.
\item $\cN=\cN_{\fg^\vee}$ is the variety of nilpotent elements in $\fg^\vee$.
\item $\bO_\reg$ is the set of regular nilpotent elements in $\fg^\vee$, and $\bO_\subreg$ is the set of subregular nilpotent elements in $\fg^\vee$. Let $U\colonequals\bO_\subreg\cup\bO_\reg$, an open and dense subset of $\cN$.
\end{itemize}
We also use the following notation for representations of $\fsl_2$, used to explicitly construct sub-regular nilpotent orbits in classical types:
\begin{itemize}
    \item For an integer $N\geqslant0$, let $V(N)$ be the unique irreducible $(N+1)$-dimensional representation of $\fsl_2$.
    \item Let $V(N)_k$ be the $k$-th weight space of $V(N)$, so $V(N)=V(N)_{-N}\oplus V(N)_{2-N}\cdots V(N)_{N-2}\oplus V(N)_N$.
    \item For $N$ odd, let $\omega\colon V(N)\times V(N)\to\C$ be the non-degenerate, $\fsl_2$-equivariant symplectic form, which is unique up to scalar. For $N$ even, similarly let $(-,-)$ be the non-degenerate, $\fsl_2$-equivariant symmetric form, which is unique up to scalar.
\end{itemize}
We will furthermore use the following notation for the subregular two-sided cell, see \S\ref{sec:extending-lusztig}:
\begin{itemize}
    \item $c_\subreg\subset\hatW$ is the subregular two-sided cell and $c^0_\subreg\subset c_\subreg$ is a left cell, as in Definition~\ref{subregular-cell-definition}.
    \item $\mu\colon c^0_\subreg\to\widehat S$ sends $w\in c^0_\subreg$ to the unique $i\in\widehat S$ such that $\ell(s_iw)<\ell(w)$.
\end{itemize}

Let $\pi\colon\widetilde\cN=T^*\cB\to\cN$ be the \emph{Springer resolution}: $T^*\cB\simeq\{(e,\fb^\vee)\in\fg^\vee\times\cB:e\in[\fb^\vee,\fb^\vee]\}$, and $\pi(e,\fb^\vee)\colonequals e$. For each $e\in\cN$, let $\cB_e\colonequals\{\fb\in\cB:e\in[\fb^\vee,\fb^\vee]\}$ be the scheme-theoretic fiber of $\pi$ over $e$, called a \emph{Springer fiber}. Given any locally closed subset $X\subset\cN$, we denote the scheme-theoretic pre-image $\pi^{-1}(X)$ as $\widetilde X$.

Let $\mathcal H_q(\widehat W)$ be the Hecke algebra, with a free $\Z[q^{\pm1}]$-basis $\{H_w\}_{w\in\widehat W}$. Then \cite{kazhdan-lusztig} defines $C_\nu\in\mathcal H_q(\widehat W)$ such that $\overline C_v=q^{-\ell(v)}C_v$ and which admit an expansion
\[
C_\nu=\sum_{w\preccurlyeq  v}P_{w,v}(q)H_w
\]
where $P_{v,v}(q)=1$ and when $w\prec v$, the polynomial $P_{w,v}(q)\in\Z[q]$ has degree $<\frac{\ell(v)-\ell(w)-1}{2}$. The \emph{inverse} Kazhdan-Luszitg polynomials $\mathbf m_v^w(q)$ are then characterized by:
\[
H_w=\sum_{v\preccurlyeq w}\epsilon(wv^{-1})\mathbf m_v^w(q)C_v.
\]
We will be particularly interested in the special value $\mathbf m_v^w\colonequals\mathbf m_v^w(1)$.

\begin{remark}\label{rmk:non-reduced}
Unless otherwise specified, $\cB_e\colonequals\pi^{-1}(e)$ is the \emph{scheme-theoretic} fiber, which in general may be non-reduced. For example, when $G$ is of type D\textsubscript{4} and $e\in\fg$ is subregular nilpotent, then $\cB_e$ is the exceptional divisor of a the resolution of a type D\textsubscript{4} Kleinian singularity $X^2+Y^3+Z^3=0$, which is non-reduced.
\end{remark}

\subsection{$G^\vee$-equivariant line bundles on flag varieties}
\begin{defn}
For $\gamma\in Q^\vee$, let $\cO_\cB(\gamma)\colonequals G^\vee\times^{B_0^\vee}\C_{-\gamma}$, a $G^\vee$-equivariant line bundle on $\cB$. For a locally closed subset $X\subset\cB$, we let $\cO_X(\gamma)\colonequals\cO_\cB(\gamma)|_X$.
\end{defn}
In fact, all $G^\vee$-equivariant line bundles on $\cB$ are obtained this way:
\begin{lemma}
$\gamma\mapsto\cO_\cB(\gamma)$ gives an isomorphism $\Lambda\xrightarrow\sim\mathrm{Pic}^{G^\vee}\!(\cB)$.
\end{lemma}

\begin{example}\label{alphai-example}
For each $i\in I$, there is a $G^\vee$-equivariant line bundle on $\cB$ whose fiber over $\fb\in\cB$ is $\pi_i(\fb)/\fb$. This is $\cO_\cB(\alpha_i^\vee)$.
\end{example}

\begin{lemma}\label{restrict-line-bundle}
For each $i\in I$, the line bundle $\cO_\cB(\gamma)$ restricted to $\pi_i^{-1}(\fp_i)\simeq\bP^1$ is $\cO_{\bP^1}(\langle \gamma,\alpha_i^\vee\rangle)$.
\end{lemma}

\subsection{Equivariant line bundles on $\bP^1$}

$\GL_2$-equivariant line bundles on $\bP^1$ are given by, for $a_1,a_2\in\Z$,
\begin{equation}\label{equiv-line-bundle}
\cO_{\bP^1}(a_1,a_2)\colonequals\GL_2\times^B\C_{(a_1,a_2)},
\end{equation}
where $\C_{(a_1,a_2)}$ is the one-dimensional $B$-representation where $\begin{pmatrix}\alpha_1&*\\&\alpha_2\end{pmatrix}\in B$ acts as $\alpha_1^{a_1}\alpha_2^{a_2}$ and $\GL_2\times^B\C_{(a_1,a_2)}$ is $\GL_2\times\C_{(a_1,a_2)}$ modulo the relation $(gb,x)\sim(g,b\cdot x)$, for $g\in\GL_2$, $b\in B$, and $x\in\C_{(a_1,a_2)}$. The underlying line bundle of $\cO_{\bP^1}(a_1,a_2)$ is $\cO_{\bP^1}(a_2-a_1)$.

\begin{remark}
    Viewing $\bP^1$ as the space of lines $\ell\subset\C^2$, there are two natural $\GL_2(\C)$-equivariant line bundles on it:
    \begin{itemize}
        \item the stalk at $\ell$ is $\ell$: this is isomorphic to $\cO_{\bP^1}(1,0)$.
        \item the stalk at $\ell$ is $\C^2/\ell$: this is isomorphic to $\cO_{\bP^1}(0,1)$.
    \end{itemize}
    Thus, $\cO_{\bP^1}(a_1,a_2)$ is isomorphic to the line bundle whose stalk at $\ell\in\bP^1$ is $\ell^{\otimes a_1}\otimes(\C^2/\ell)^{\otimes a_2}$.
\end{remark}

The global sections of $\cO_{\bP^1}(a_1,a_2)$ can be explicitly computed as a $\GL_2$-representation:
\begin{prop}\label{prop:gl2-equivariant-global-sections}
    Let $a_1,a_2\in\Z$ be such that $a_1\leqslant a_2$. Then letting $V$ be the tautological $2$-dimensional representation of $\GL_2$,
    \[
    H^0(\bP^1,\cO_{\bP^1}(a_1,a_2))\simeq \det(V)^{\otimes -a_2}\otimes\Sym^{a_2-a_1}(V).
    \]
\end{prop}
\begin{proof}
Global sections of $\cO_{\bP^1}(a_1,a_2)$ are functions $\GL_2\to \GL_2\times^B\C_{(a_1,a_2)}:g\mapsto (g,f(g))$, which satisfies $(gb,f(gb))=(g,f(g))$ for all $g\in \GL_2$ and $b\in B$. In other words, it is the space of functions $f\colon\GL_2\to\C_{(a_1,a_2)}$ such that $f(gb)=b^{-1}\cdot f(g)$ for all $g\in\GL_2$ and $b\in B$. A basis for $H^0(\bP^1,\cO_{\bP^1}(a_1, a_2))$ is given by, for $g=\begin{pmatrix}\alpha&\beta\\\gamma&\delta\end{pmatrix}$,
\[
f(g)\colonequals\det(g)^{-a_2}\alpha^i\gamma^j,
\]
where $i$ and $j$ are non-negative such that $i+j=a_2-a_1$.
\end{proof}

\begin{example}
The canonical bundle on $\bP^1$ is $\cO_{\bP^1}(1,-1)$: the underlying bundle being $\cO(-2)$ and the $\GL_2$-equivariance descending to $\PGL_2$-equivariance uniquely characterizes it.
\end{example}

\begin{lemma}\label{sym^n-calc}
    Let $n\geqslant0$, and let $V$ be a $2$-dimensional vector space.
    Then, in $K^{\GL(V)}(\bP_V)$,
    \[
    [\Sym^{n}(V)\otimes\cO_{\bP_V}]\simeq\sum_{i=0}^n[\cO_{\bP_V}(i,n-i)]
    \]
\end{lemma}
\begin{proof}
    After taking the image of $[\Sym^n(V)\otimes\cO_{\bP_V}]$ under the isomorphism $K^{\GL(V)}(\bP_V)\simeq K^T(*)$ for a maximal torus $T\subset\GL(V)$, this is simply the weight space decomposition.
\end{proof}
\begin{remark}\label{sym-o2-rep}
Consider the orthogonal group $\tO_2 \subset \GL_2$, which we identify with the semidirect product $\G_m\rtimes\Z/2$.
    We will eventually restrict the equality in Lemma~\ref{sym^n-calc} to $K^{\G_m\rtimes\Z/2}(\bP^1)$. Then, it will be useful to observe that as $\G_m\rtimes\Z/2$-representations:
    \begin{align*}
        \Sym^{2n}(V)&\simeq1+\sum_{i=1}^{n}\xi_{2i}\\
        \Sym^{2n-1}(V)&\simeq\sum_{i=1}^n\xi_{2i-1}.
    \end{align*}
\end{remark}

As a corollary of Proposition~\ref{prop:gl2-equivariant-global-sections}, we obtain the following classification of $\tO_2$-equivariant line bundles on $\bP^1$:

\begin{cor}\label{cor:o2-equiv}
The $\G_m\rtimes\Z/2$-equivariant line bundles on $\bP^1$ are of the form $\cO_{\bP^1}(-m,0)$ and $\cO_{\bP^1}(-m+1,1)$. In particular, there is an isomorphism
\begin{align*}
    \Z\times\Z/2&\simeq\Pic^{\G_m\rtimes\Z/2}(\bP^1)\\
    (m,0)&\mapsto\cO_{\bP^1}(-m,0)\\
    (m,1)&\mapsto\cO_{\bP^1}(-m+1,1).
\end{align*}
Moreover, the global sections are given by, for $m\geqslant0$:
\begin{align*}
    H^0(\bP^1,\cO_{\bP^1}(-m,0))&\simeq\begin{cases}
        \C+\sum_{i=1}^m\xi_{2i}&m\equiv 0\pmod2\\
        \sum_{i=0}^{(m-1)/2}\xi_{2i+1}&m\equiv 1\pmod 2
    \end{cases}\\
    H^0(\bP^1,\cO_{\bP^1}(-m+1,1))&\simeq\begin{cases}
        \sgn+\sum_{i=1}^m\xi_{2i}&m\equiv 0\pmod2\\
        \sum_{i=0}^{(m-1)/2}\xi_{2i+1}&m\equiv 1\pmod 2
    \end{cases}
\end{align*}
where $\xi_m\colonequals\Ind_{\G_m}^{\tO_2}\C\langle m\rangle$ are irreducible representations of $\tO_2$.
\end{cor}
\begin{proof}
    The first half is clear. For the second half, it suffices to compute $\Sym^m(\xi_1)$. Note that the character of $\xi_i$ is 
   \begin{align*}
   \tr\xi_1(z,1)&=z^i+z^{-i}\\
   \tr\xi_1(z,-1)&=0.
   \end{align*}
    Now, the character of $\Sym^m(\xi_1)$ is
    \begin{align*}
   \tr\xi_1(z,1)&=z^m+z^{m-2}+\cdots+z^{-m}\\
   \tr\xi_1(z,-1)&=1+(-1)^m,
   \end{align*}
   which matches the character of the expressions in the Corollary.
\end{proof}

\subsection{Explicit description of non-simply laced root systems}
Throughout the paper, we will rely on the following explicit description of root systems in types B, C, F\textsubscript{4}, and G\textsubscript{2}:

\begin{example}\label{type-b-highest} When $G$ is of type $\text{B}_n$ and $G^\vee$ is of type $\text{C}_n$, the dual root system is:
\begin{align*}
V^\vee&\colonequals\R\{\epsilon_1,\dots,\epsilon_{n}\}\\
\Phi^\vee&\colonequals\{\pm\epsilon_i\pm\epsilon_j:1\leqslant i\ne j\leqslant n\}\cup\{\pm2\epsilon_i:1\leqslant i\leqslant n\}\\
\Delta^\vee&\colonequals\{\epsilon_1-\epsilon_2,\dots,\epsilon_{n-1}-\epsilon_n\}\cup\{2\epsilon_n\}.
\end{align*}
Let $\alpha_i^\vee\colonequals\epsilon_i-\epsilon_{i+1}$ for $i\leqslant n-1$ and $\beta^\vee\colonequals2\epsilon_n$. The highest short coroot is
\[\theta^\vee=\alpha_1^\vee+2\alpha_2^\vee+\cdots+2\alpha_{n-1}^\vee+\beta^\vee=\epsilon_1+\epsilon_2.\]
\end{example}

\begin{example}\label{type-c-highest} When $G$ is of type $\text{C}_n$ and $G^\vee$ is of type $\text{B}_n$, the dual root system is:
\begin{align*}
V^\vee&\colonequals\R\{\epsilon_1,\dots,\epsilon_{n}\}\\
\Phi^\vee&\colonequals\{\pm\epsilon_i\pm\epsilon_j:1\leqslant i\ne j\leqslant n\}\cup\{\pm\epsilon_i:1\leqslant i\leqslant n\}\\
\Delta^\vee&\colonequals\{\epsilon_1-\epsilon_2,\dots,\epsilon_{n-1}-\epsilon_n\}\cup\{\epsilon_n\}.
\end{align*}
Let $\alpha_i^\vee\colonequals\epsilon_i-\epsilon_{i+1}$ for $i\leqslant n-1$ and $\beta^\vee\colonequals\epsilon_n$. The highest short root is
\[\theta^\vee=\alpha_1^\vee+\cdots+\alpha_{n-1}^\vee+\beta^\vee=\epsilon_1.\]   
\end{example}

\begin{example}\label{type-f-highest}
When $G$ and $G^\vee$ are of type F\textsubscript{4}, 
\begin{align*}
    V^\vee&\colonequals\R\{\epsilon_1,\epsilon_2,\epsilon_3,\epsilon_4\}\\
    \Phi^\vee&\colonequals\{\pm\epsilon_i,\pm\epsilon_i\pm\epsilon_j:i\ne j\}\cup\Big\{\frac12(\pm\epsilon_1\pm\epsilon_2\pm\epsilon_3\pm\epsilon_4)\Big\}\\
    \Delta^\vee&\colonequals\{\alpha_1^\vee,\alpha_2^\vee,\alpha_3^\vee,\alpha_4^\vee\},
\end{align*}
where $\alpha_1^\vee=\epsilon_1-\epsilon_2–\epsilon_3-\epsilon_4$, $\alpha_2^\vee=2\epsilon_4$, $\alpha_3^\vee=\epsilon_3-\epsilon_4$, $\alpha_4^\vee=\epsilon_2-\epsilon_3$. The highest short root is
\[
\theta^\vee=\alpha_1^\vee+3\alpha_2^\vee+2\alpha_3^\vee+\alpha_4^\vee=\epsilon_1+\epsilon_2.
\]
\end{example}

\begin{example}\label{type-g-highest}
    When $G$ and $G^\vee$ are of type G\textsubscript{2},
    \begin{align*}
        V^\vee&\colonequals\R\{\alpha^\vee,\beta^\vee\}\\
        \Phi^\vee&\colonequals\{\pm\alpha^\vee,\pm\beta^\vee,\pm(\alpha^\vee+\beta^\vee),\pm(2\alpha^\vee+\beta^\vee),\pm(3\alpha^\vee+\beta^\vee),\pm(3\alpha+2\beta^\vee)\}\\
        \Delta^\vee&\colonequals\{\alpha^\vee,\beta^\vee\},
    \end{align*}
    where the inner product on $V^\vee$ is given by $(\alpha^\vee,\alpha^\vee)=2,(\alpha^\vee,\beta^\vee)=-3,(\beta^\vee,\beta^\vee)=6$. The highest short root is 
    \(
    \theta^\vee=2\alpha^\vee+\beta^\vee
    \).
\end{example}

\subsection{Folding Dynkin diagrams}\label{sec:unfolding}
Let $\widetilde\Delta^\vee\subset \widetilde\Phi^\vee$ be a simple, simply-laced root system (i.e., of type A, D, or E) in a Euclidean vector space $\widetilde V^\vee$, and let $\Gamma$ be a finite group of automorphisms of the root system (equivalently, graph-theoretic automorphisms of the Dynkin diagram). We assume without loss that $(\widetilde\alpha^\vee,\widetilde\alpha^\vee)=2$ for all $\widetilde\alpha^\vee\in\widetilde\Delta^\vee$. Assume further that for any $\widetilde\alpha^\vee\in\widetilde\Delta^\vee$ and $\gamma\in\Gamma$ such that $\widetilde\alpha^\vee\ne\gamma(\widetilde\alpha^\vee)$, the two coroots $\widetilde\alpha^\vee$ and $\gamma(\widetilde\alpha^\vee)$ are orthogonal to each other.

Then we define the \emph{folding} of the root system $\widetilde\Delta^\vee$ as follows: for each 
$\Gamma$-orbit $\zeta\in\Gamma\backslash\widetilde\Delta^\vee$, let
\[\alpha_\zeta^\vee\colonequals\sum_{\widetilde\alpha^\vee\in\zeta}\widetilde\alpha^\vee,\]
which has length $2|\zeta|$, by our orthogonality assumption. Now let $\Delta^\vee\colonequals\{\alpha_\zeta^\vee:\zeta\in\Gamma\backslash\widetilde\Delta^\vee\}$, which form a basis of the root system $\Phi^\vee\colonequals(\widetilde\Phi^\vee)^\Gamma=\{\widetilde\alpha^\vee\in \widetilde\Phi^\vee:\gamma(\widetilde\alpha^\vee)=\widetilde\alpha^\vee\text{ for all }\gamma\in\Gamma\}$, in the Euclidean space $V^\vee\colonequals(\widetilde V^\vee)^\Gamma$.

Moreover, the highest short coroot $\theta^\vee\in\Phi^\vee$ equals the highest coroot $\theta^\vee\in\widetilde \Phi^\vee$.

Any simple Lie algebra can be realized as the folding of a simply-laced Lie algebra: $\text{C}_n$ is the folding of $\text{A}_{2n-1}$ and G\textsubscript{2} is the folding of D\textsubscript{4}. Thus, we realize our root system of $G^\vee$ as the folding of a simply-laced adjoint reductive group $\widetilde G^\vee$ by some automorphism group $\Gamma$.

\begin{example}\label{typeb-unfolding}
Let $\widetilde\Delta^\vee\subset \widetilde\Phi^\vee$ be of type ${\text{D}}_{n+1}$, where $n\geqslant3$. Explicitly:
\begin{align*}
\widetilde V^\vee&\colonequals\R\{\epsilon_1,\dots,\epsilon_{n+1}\}\\
\widetilde\Phi^\vee&\colonequals\{\pm\epsilon_i\pm\epsilon_j:1\leqslant i\ne j\leqslant n+1\}\\
\widetilde\Delta^\vee&\colonequals\{\epsilon_1-\epsilon_2,\dots,\epsilon_{n-1}-\epsilon_n\}\cup\{\epsilon_n\pm\epsilon_{n+1}\}.
\end{align*}
Then $\Gamma=\{1,\sigma\}$ where $\sigma(\epsilon_i)\colonequals\epsilon_i$ for $i\leqslant n$ and $\sigma(\epsilon_{n+1})\colonequals-\epsilon_{n+1}$ satisfies the conditions above. The folding then gives a root system $\Delta^\vee\subset\Phi^\vee$ of type $\text{C}_n$ is as in Example~\ref{type-b-highest}.
\end{example}

\begin{example}
    Let $\widetilde\Delta^\vee\subset\widetilde\Phi^\vee$ be of type $A_{2n-1}$, so:
    \begin{align*}
        \widetilde V^\vee&=\Big\{\sum_{i=1}^{2n}a_i\delta_i:\sum_{i=1}^{2n}a_i=0\Big\}\\
        \widetilde\Delta^\vee&=\{\delta_i-\delta_{i+1}:0< i<2n\}.
    \end{align*}
    Then $\Gamma=\{1,\sigma\}$ where $\sigma(\delta_i)=-\delta_{2n+1-i}$. Then the folding of the root system is given by:
    \begin{align*}
        V^\vee&=\Big\{\sum_{i=1}^{n}a_i(\delta_i-\delta_{2n+1-i})\Big\}\\
        \Delta^\vee&=\{\delta_i-\delta_{i+1}+\delta_{2n-i}-\delta_{2n+1-i}:0< i<n\}\cup\{\delta_n-\delta_{n+1}\},
    \end{align*}
    which is identified with the root system of Example~\ref{type-c-highest} of type $\text{B}_n$, by $\epsilon_i=\delta_i-\delta_{2n+1-i}$ for $1\leqslant i\leqslant n$.
\end{example}

\begin{example}\label{ex:f4-root-lattice}
Recall that explicitly, the E\textsubscript{6} lattice is
\[
\widetilde Q^\vee\colonequals\Big\{\sum_{i=1}^8a_i\delta_i:\text{either all }a_i\in\Z\text{ or all }a_i\in\Z+\frac12,\text{ and }\sum_{i=1}^6a_i=a_7+a_8=0\Big\},
\]
with root system
\begin{align*}
\widetilde \Phi^\vee\colonequals&\Big\{\delta_i-\delta_j:i\ne j,\text{ either }i,j\notin\{7,8\}\text{ or }i,j\in\{7,8\}\Big\}\\&\cup\Big\{\frac12(\pm\delta_1\cdots\pm\delta_8):\text{3 minus signs among }\delta_1,\dots,\delta_6\text{ and 1 minus sign among }\delta_7,\delta_8\Big\},
\end{align*}
with simple roots
\[
\widetilde\Delta^\vee\colonequals\big\{\widetilde\alpha_i^\vee\colonequals\delta_i-\delta_{i+1}:2\leqslant i\leqslant 5\big\}\cup\big\{\widetilde\beta^\vee\colonequals\delta_7-\delta_8\big\}\cup\big\{\widetilde\gamma^\vee\colonequals\frac12(\delta_1-\delta_2-\delta_3-\delta_4+\delta_5+\delta_6-\delta_7+\delta_8)\big\},
\]
with Dynkin diagram:
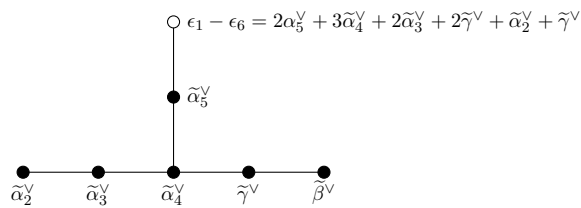
\[
\dynkin[extended,root
radius=.08cm,edge length=1cm,labels={\epsilon_1-\epsilon_6=2\alpha_5^\vee+3\widetilde\alpha_4^\vee+2\widetilde\alpha_3^\vee+2\widetilde\gamma^\vee+\widetilde \alpha_2^\vee+\widetilde\gamma^\vee,\widetilde\alpha_2^\vee,\widetilde\alpha_5^\vee,\widetilde\alpha_3^\vee,\widetilde\alpha_4^\vee,\widetilde\gamma^\vee,\widetilde\beta^\vee}]E6
\]
Thus, the folded root system $\Phi^\vee$ has simple roots:
\begin{align*}
    \alpha_1^\vee&=\widetilde\alpha_2^\vee+\widetilde\beta^\vee=\delta_2-\delta_3+\delta_7-\delta_8\\
\alpha_2^\vee&=\widetilde\alpha_3^\vee+\widetilde\gamma^\vee=\frac12(\delta_1-\delta_2+\delta_3-3\delta_4+\delta_5+\delta_6-\delta_7+\delta_8)\\
\alpha_3^\vee&=\widetilde\alpha_4^\vee=\delta_4-\delta_5\\
\alpha_4^\vee&=\widetilde\alpha_5^\vee=\delta_5-\delta_6.
\end{align*}
The relation with Example~\ref{type-f-highest} is given by the change of basis:
\begin{align*}
    \epsilon_1&=\frac14(3,1,-1,-1,-1,-1,1,-1)\\
    \epsilon_2&=\frac14(1,-1,1,1,1,-3,-1,1)\\
    \epsilon_3&=\frac14(1,-1,1,1,-3,1,-1,1)\\
    \epsilon_4&=\frac14(1,-1,1,-3,1,1,-1,1).
\end{align*}
\end{example}

\begin{example}
    Let $\widetilde\Delta^\vee\subset\widetilde\Phi^\vee$ be of type D\textsubscript{4}, so:
    \begin{align*}
        \widetilde V^\vee&=\R\{\epsilon_1,\epsilon_2,\epsilon_3,\epsilon_4\}\\
        \widetilde\Delta^\vee&=\{\epsilon_1-\epsilon_2,\epsilon_2-\epsilon_3,\epsilon_3-\epsilon_4,\epsilon_3+\epsilon_4\}.
    \end{align*}
    The root system has an automorphism that permutes the three coroots $\epsilon_1-\epsilon_2$, $\epsilon_3-\epsilon_4$, and $\epsilon_3+\epsilon_4$. The folded root system $\Phi^\vee$ is of type G\textsubscript{2} and has simple roots, as in Example~\ref{type-g-highest}:
    \begin{align*}
        \alpha^\vee&=\epsilon_2-\epsilon_3\\
        \beta^\vee&=(\epsilon_1-\epsilon_2)+(\epsilon_3-\epsilon_4)+(\epsilon_3+\epsilon_4)=\epsilon_1-\epsilon_2+2\epsilon_3.
    \end{align*}
\end{example}

\subsection{The Springer resolution}

Since the projection to the second factor $\widetilde\cN\to\cB$ is $G^\vee$-equivariant, the $G^\vee$-equivariant line bundles $\cO_\cB(\gamma)$ pullback to $G^\vee$-equivariant line bundles $\cO_{\widetilde\cN}(\gamma)$.

For each $\alpha\in\Delta$, as in \cite[\S4.1]{BKK} let $\iota_{\alpha}$ be the closed immersion (an inclusion of vector bundles over $\cB$)
\begin{equation}\label{Ni-definition}
\widetilde \cN_\alpha\colonequals T^*\cP_\alpha\times_{\cP_\alpha}\cB=\{(e,\fb^\vee)\in\fg^\vee\times\cB:e\in\pi_\alpha(\fb^\vee)^{\mathrm{nil}}\}\hookrightarrow \widetilde\cN=T^*\cB.\end{equation}
\begin{lemma}[{\cite[Equation~(13)]{bez1}} or {\cite[Lem~5.3]{achar-perverse}}]\label{ses-lemma}
    For each $\alpha\in\Delta$, there is a $G^\vee$-equivariant short exact sequence
    \[
    0\to \cO_{\widetilde\cN}(\alpha^\vee)\to \cO_{\widetilde\cN}\to i_{\alpha *}\cO_{\widetilde\cN_\alpha}\to 0.
    \]
\end{lemma}

\subsection{Kleinian singularities and Slodowy slices}\label{sec:slodowy-slice}
Let $\{e,h,f\}$ be a $\fsl_2$-triple in $\fg^\vee$ such that $e\in\bO_\subreg$. Then let $S\colonequals(e+\Cent_{\fg^\vee}(f))\cap\cN$, which is a surface with a Kleinian singularity and has a resolution of singularities $\pi_S\colon\widetilde S\to S$. Here $S$ carries the action of $\Cent_e\colonequals\Cent_{G^\vee}(e,f)$. Since $\pi_S$ is given by a sequence of blow-ups, the exceptional divisor $\cB_e$ is the union of $\bP^1$'s. The $\Cent_e$-orbits of irreducible components of $\cB_e$ can be explicitly described:
\begin{lemma}
For $e\in\bO_\subreg$, the Springer fiber $\cB_e\colonequals\{\fb\in\cB:e\in\fb\}$ has a (set-theoretic) decomposition
\[
\cB_e=\bigcup_{\alpha\in\Delta}\big(\cB_e\cap\widetilde\cN_\alpha\big),
\]
where $\widetilde\cN_\alpha$ is defined in \eqref{Ni-definition}. Here 
\[
\cB_e\cap\widetilde\cN_\alpha=\pi_\alpha^{-1}(S_\alpha)
\]
for $S_\alpha\colonequals\{\fp_\alpha^\vee\in\cP_\alpha:e\in(\fp_\alpha^\vee)^{\mathrm{nil}}\}\subset\cP_\alpha$ is a finite set with a transitive $\Cent_e$-action. In other words, $\cB_e$ is a finite union of $\bP^1$'s.
\end{lemma}
Let $\bP_\alpha\colonequals\pi_\alpha^{-1}(S_\alpha)$, a disjoint union of $\bP^1$'s, so set-theoretically,\footnote{As noted in Remark~\ref{rmk:non-reduced}, $\cB_e$ may be non-reduced.} $\cB_e=\bigcup_{\alpha\in\Delta}\bP_\alpha$. Now, the intersection multiplicity is given as:
\begin{lemma}[{\cite[\S6.2]{slodowy}}]
Let $\fg$ be a simple Lie algebra and let $\alpha,\beta\in\Delta$. Then $\bP_\alpha\cdot\bP_\beta=-(\alpha^\vee,\beta^\vee)$. 
\end{lemma}
 The sheaf of ideals defining $\cB_e\subset\widetilde S$ has the following explicit description:
\begin{prop}\label{slodowy-exact-seq}
Let $\theta^\vee\in\Phi^\vee$ be the highest short coroot. Then there is a short exact sequence of $\cO_{\widetilde S}$-modules
\[
0\to\cO_{\widetilde S}(\theta^\vee) \to \cO_{\widetilde S}\to\cO_{\cB_e}\to 0
\]
\end{prop}
\begin{proof}
By \cite[Theorem~4]{artin} for $n=1$, the sheaf of ideals defining $\cB_e\subset\widetilde S$ is $\cO_{\widetilde S}(-Z_{\mathrm{num}})$, where $Z_{\mathrm{num}}$, the \emph{numerical cycle}, is the unique minimal positive exceptional divisor supported on $\cB_e$ such that $Z_{\mathrm{num}}\cdot\bP_\alpha\leqslant0$ for all $\alpha\in\Delta$. Using the formula $\bP_\alpha\cdot\bP_\beta=-(\alpha^\vee,\beta^\vee)$ and Example~\ref{alphai-example} gives the desired result.
\end{proof}

\begin{remark}
$\theta^\vee$ is equivalently the dual of the highest root $\theta\in\Phi$.
\end{remark}

In fact, the exact sequence of Proposition~\ref{slodowy-exact-seq} extends to $\widetilde\cN$:
\begin{cor}\label{cor:subreg-desc}
Let $\theta^\vee\in\Phi^\vee$ be the highest short root. Then there is a $G^\vee$-equivariant short exact sequence of coherent sheaves on $\widetilde\cN$
\[
0\to\cO_{\widetilde\cN}(\theta^\vee)\to\cO_{\widetilde\cN}\to\cO_{\widetilde{\overline\bO}_\subreg}\!\to 0.
\]
\end{cor}
\begin{proof}
Since $\widetilde{\overline\bO}_\subreg\subset\widetilde\cN$ is a divisor it is defined by some line bundle $\mathcal I\subset \mathcal O_{\widetilde\cN}$. All $G^\vee$-equivariant line bundles on $\widetilde\cN$ must be of the form $\cO_{\widetilde\cN}(\gamma)$ but by restricting to $\widetilde S$, and by comparing with Proposition~\ref{slodowy-exact-seq}, we see that for all $\alpha\in\Delta$, $(\alpha^\vee,\gamma)=(\alpha^\vee,\theta^\vee)$,
i.e., $\gamma=\theta^\vee$.
\end{proof}

\begin{remark}\label{Cs0-calc}
Since $s_0=t_{\theta^\vee}\cdot s_{\theta^\vee}$, the above shows that in $K^{G^\vee}\!(\widetilde U)$,
\[
s_0\cdot[\cO_{\widetilde U}]=t_{\theta^\vee}{s_{\theta^\vee}}[\cO_{\widetilde U}]=-[\cO_{\widetilde U}(\theta^\vee)]=[\cO_{\widetilde\bO_\subreg}]-[\cO_{\widetilde U}].
\]
Another way to see this is to use Lemma \ref{alpha-action} below.
\end{remark}
As a consequence of Corollary~\ref{cor:subreg-desc}, we obtain the following:
\begin{prop}\label{prop:Cs0}
    Using notation from \cite{arkhipov-bez}, consider Arkhipov and Bezrukavnikov's equivalence 
    \[
    {^f}\Phi\colon D^\bounded\Coh^{G^\vee}\!\big(\widetilde\cN\big)\simeq D^\bounded\big({^f}\mathcal P_I\big),
    \]
    and let $L_{s_0}=j_{s_0!*}(\underline{\overline{\mathbb Q}}_\ell[1])$. Then
    \[{^f}\Phi\big(\cO_{\widetilde{\overline\bO}_\subreg}[-1]\big)\simeq L_{s_0}.\] In particular, $\cO_{\widetilde{\overline\bO}_\subreg}[-1]$ is an irreducible exotic coherent sheaf.
\end{prop}
\begin{proof}
    Since the stratum $\mathcal F\ell_{s_0}\subset\mathcal F\ell$ has closure isomorphic to $\bP^1$, there is an exact sequence
    \[    0\to L_{s_0}\to j_{s_0!}\to \delta_e\to 0,
    \]
    and there is a one-dimensional space of homomorphisms $j_{s_0!}\to \delta_e$. Now under Arkhipov and Bezrukavnikov's equivalence, $j_{s_0!}=j_{t_{\theta^\vee}!}$ is sent to $\cO_{\widetilde\cN}(\theta^\vee)$ and $\delta_e$ is sent to $\cO_{\widetilde\cN}$, so we obtain an exact sequence
    \begin{equation}\label{eq:perverse-exact-seq}
        0\to {^f}\Phi^{-1}(L_{s_0})\to \cO_{\widetilde\cN}(\theta^\vee)\to \cO_{\widetilde\cN}\to 0.
    \end{equation}
    Comparing this with the non-split exact triangle
    \[
    \cO_{\widetilde{\overline\bO}_\subreg}[-1]\to \cO_{\widetilde \cN}(\theta^\vee)\to\cO_{\widetilde\cN}\xrightarrow{+1},
    \]
    from Corollary~\ref{cor:subreg-desc}, we see that ${^f}\Phi^{-1}(L_{s_0})\simeq\cO_{\widetilde{\overline\bO}_\subreg}[-1]$.
\end{proof}

\subsection{$K$-theory}

There is a standard short exact sequence of $\Z\hatW\simeq K^{G^\vee}\!(\widetilde\cN\times_\cN\widetilde\cN)$-modules
\[
0\to K^{G^\vee}\!(\widetilde\bO_\subreg)\to K^{G^\vee}\!(\widetilde U)\to K^{G^\vee}\!(\widetilde\bO_\reg)\to0.
\]
Note that $\pi\colon\widetilde\cN\to\cN$ is an isomorphism on the dense open subset $\bO_\reg\subset\cN$, so $K^{G^\vee}\!(\widetilde\bO_\reg)\simeq K^{G^\vee}\!(\bO_\reg)\simeq\Z$. Moreover, $K^{G^\vee}\!(\widetilde\bO_\subreg)\simeq K^{\Cent_e}(\cB_e)$, where $e\in\bO_\subreg$, and $\Cent_e$ is the reductive part of the centralizer of $e$ in $G^\vee$ (i.e., the centralizer of a $\fsl_2$-triple completing $e$). 

\begin{lemma}[{\cite[Prop~2.9]{lusztig-affine-hecke-algebras}}]\label{lem:lambda-action}
The $\widehat W$-action on $K^{\Cent_e}(\cB_e)$ is such that $\gamma\in\Lambda\subset\widehat W$ acts as $-\otimes\cO_{\cB_e}(\gamma)$.
\end{lemma}

Lemma~\ref{ses-lemma} allows us to inductively calculate the $ Q^\vee$-action on $\cO_{\widetilde U}$:

\begin{lemma}\label{alpha-action}
For all $\alpha^\vee\in\Delta^\vee$ and $\gamma\in Q^\vee$,
\(
[\cO_{\widetilde U}(\gamma+\alpha^\vee)]=[\cO_{\widetilde U}(\gamma)]-[\cO_{\bP_\alpha}(\gamma)]
\) in $K^{G^\vee}(\widetilde U)$.
\end{lemma}

\subsection{Representations of the affine Lie algebra $\widehat\fg$}\label{sec:rep-theory-review}
Let $\widehat\fg\colonequals\fg[t^{\pm1}]\oplus\C K\oplus \C d$ be the affine Lie algebra associated to $\fg$, with root system $\widehat\Phi$. Then $\widehat\fg$ admits a standard direct sum decomposition $\widehat\fg=\widehat\fn_+\oplus\widehat\fh\oplus \widehat\fn_-$, where $\widehat\fh\colonequals\fh\oplus \C K\oplus\C d$, and $\widehat \fn_+\colonequals\fn_+\oplus\bigoplus_{n>0}\fg t^n$, and $\widehat \fn_-\colonequals\fn_-\oplus\bigoplus_{n<0}\fg t^n$. Let $\widehat\fb\colonequals\widehat\fh\oplus\widehat\fn_+$. For any character $\Lambda\in\widehat\fh^*$, extending trivially to $\widehat \fn_+$ gives a character $\widehat\fb\to\C$. Now, the associated \emph{Verma module} is
\[
M(\Lambda)\colonequals U(\widehat\fg)\otimes_{U(\widehat\fb)}\C_\Lambda,
\]
which has a unique irreducible quotient $L(\Lambda)$.
\begin{defn}
For $\Lambda\in\widehat\fh^*$, let $\kappa\colonequals\Lambda(K)$, the scalar by which the (central) element $K\in\widehat\fg$ acts on $L(\Lambda)$.
\end{defn}
Let $\widehat{\Delta}$ be the set of simple roots of $\widehat{\mathfrak{g}}$. 
Set $\widehat{Q}^+\colonequals\sum_{i \in \widehat{\Delta}}\Z_{\geqslant 0}\alpha_i$. The above representations naturally live in the following category:
\begin{defn}\label{defn:category-o}
Fix a $\kappa>-h^\vee$, the dual Coxeter number. The category $\cO_\kappa$ is the full subcategory of the category of $\widehat\fg$-modules of level $\kappa$, consisting of $\widehat{\fg}$-modules $M$ such that
\begin{enumerate}
\item $M=\bigoplus_{\mu\in\widehat\fh^*}M_\mu$, where $M_\mu$ is a generalized $\mu$-weight space for $\widehat\fh$;
\item each $M_\mu$ are finite-dimensional; and
\item for any $\mu\in\widehat\fh^*$ there exists only finitely many $\beta\in\mu+\widehat{Q}^+$ such that $M_\beta\ne0$.
\end{enumerate}
\end{defn}
For an object in the category $\cO_\kappa$, we can define the character:
\begin{defn}
For an object $M\in \cO_\kappa$, let the \emph{character} of $M$ be
\[
\ch M\colonequals\sum_{\mu\in\widehat\fh^*}(\dim M_\mu)\cdot e^\mu,
\]
which is an element of the group of formal power series $\C[\![\widehat\fh^*]\!]\colonequals\{\sum_{\mu\in\widehat\fh^*}a_\mu e^\mu:a_\mu\in\C\}$.
\end{defn}
Moreover, for convenience, let $\widehat R\colonequals e^{\widehat\rho}\prod_{\alpha\in\widehat\Phi_+}(1-e^{-\alpha})^{\mathrm{mult}(\alpha)}$, where $\widehat\Phi_+\subset\widehat\Phi$ is the set of positive roots of $\widehat\fg$. Moreover, let $\widehat\rho\colonequals\sum_{i=0}^r\Lambda_i$, where $\Lambda_i\in\widehat\fh^*$ are the fundamental weights. Then
\[
\widehat R\ch M(\Lambda)=e^{\Lambda+\widehat\rho}.
\]

\section{The subregular two-sided cell}\label{sec:extending-lusztig}
The subregular two-sided cell $c_\subreg$ and the left cell $c^0_\subreg$ have the following description, by \cite[Prop~3.8]{lusztig-some-examples}:
\begin{align*}
c_\subreg&=\{w\in\widehat W \setminus \{1\}:w\text{ has a unique reduced decomposition}\}\\
c_\subreg^0&=c_\subreg\cap\hatW^f=\{w\in c_\subreg:\ell(ws_0)<\ell(w)\}.
\end{align*}

We first make some easy preliminary observations (see \cite[Section 3.7]{lusztig-some-examples}):
\begin{lemma}\label{graph-theory}
In a Coxeter group $\langle s_1,\dots,s_k|(s_is_j)^{m_{ij}}=1\rangle$ with $m_{ii}=1$ and $m_{ij}\geqslant2$ for $i\ne j$, if $w=s_{i_1}\cdots s_{i_n}$ is a unique reduced decomposition, then:
\begin{enumerate}
\item\label{graph1} For each $k\leqslant n-1$, $m_{i_ki_{k+1}}\geqslant3$; and
\item\label{graph2} No pattern of the form $s_is_js_i$ appears for $i,j$ such that $m_{ij}=3$, no pattern of the form $s_is_js_is_j$ appears for $i,j$ such that $m_{ij}=4$, etc.
\end{enumerate}
In other words, $w$ defines a path $(i_1,\dots,i_n)$ on the Coxeter-Dynkin diagram such that an edge $ij$ appears $<m_{ij}-1$ times consecutively.
\end{lemma}
\begin{proof}
For \eqref{graph1}, if $m_{i_ki_{k+1}}=2$ then $s_{i_k}$ and $s_{i_{k+1}}$ commute with each other, so swapping $s_{i_k}$ and $s_{i_{k+1}}$ gives another presentation of $w$.

Similarly for \eqref{graph2}, if $m_{ij}=3$ then $s_{i}s_js_i=s_js_is_j$, and $m_{ij}=4$ then $s_is_js_is_j=s_js_is_js_i$, so the reduced decomposition is not unique.
\end{proof}

Now each element $w\in c_\subreg^0$ has a unique reduced expression of the form $s_{i_1}\cdots s_{i_k}s_0$. Following \cite[\S3.7]{lusztig-some-examples} Let $\Gamma_0$ be the graph with underlying set $c_\subreg^0$, and an edge $y\sim w$ for $y,w\in c_\subreg^0$ such that there exists a simple reflection $s\in \widehat S$ with $y=sw$.

There is a map
\begin{align*}
\mu\colon c^0_\subreg&\to \widehat S=S\cup\{s_0\}\\
    s_{i_1}\cdots s_{i_k}s_0&\mapsto s_{i_1},
\end{align*}
characterized by $\mu(w)=s_i$ for $w\in c_\subreg^0$ if and only if $\ell(s_iw)<\ell(w)$. Now \cite[Prop~3.8]{lusztig-some-examples} defines the $\widehat W$-module $E^0_\subreg$ with $\Z$-basis $\{e_w:w\in c_\subreg^0\}$: for $y\in\widehat S$,
\begin{equation}\label{e-defn}
    t(e_w)=\begin{cases}
        -e_w&\text{if }\mu(w)=t\\
        \displaystyle e_w+\sum_{\substack{y\in c_\subreg^0,\ \mu(y)=t\\
        y,w\text{ edge of }\Gamma_0}}e_y&\text{if }\mu(w)\ne t.
    \end{cases}
\end{equation}
Now, we define the following module, which will be used in our description of the $\widehat W$-module structure on $K^{G^\vee}\!(\widetilde U)$:
\begin{prop-def}\label{prop-def-E}
Let $\widetilde E^0_\subreg$ be the $\widehat W$-module with $\Z$-basis $\{e_w:w\in c_\reg\cup c_\subreg^0\}=\{e_1\}\cup\{e_w:w\in c^0_\subreg\}$ such that $t(e_w)$ is defined as in \eqref{e-defn} for $w\in c_\subreg$ and
\begin{equation}\label{e-defn2}
t(e_1)=\begin{cases}
    e_1+e_{s_0}&t=s_0\\
    -e_1&t\in S.
\end{cases}
\end{equation}

The module $\widetilde E^0_\subreg$ is well-defined, and fits into a short exact sequence
\[
0\to E^0_\subreg\to\widetilde E^0_\subreg\to\Z\to 0.
\]
\end{prop-def}
\begin{proof}
It suffices to check, for each $t\in  S$, the braid relation of $t$ and $s_0$ acting on $e_1$, i.e., that
\begin{equation}\label{braid}
\underbrace{\cdots s_0ts_0}_{m_{s_0,t}}(e_1)=\underbrace{\cdots ts_0t}_{m_{s_0,t}}(e_1).
\end{equation}
An easy induction shows that for $n<m_{s_0,t}$, 
\begin{equation}\label{braid2}
\underbrace{\cdots s_0ts_0}_{n}(e_1)=e_1+e_{s_0}+\cdots +e_{\underbrace{\cdots s_0ts_0}_{n}}.
\end{equation}
When $n=m_{s_0,t}-1$, the above is exactly the right-hand side of \eqref{braid}. For  $n=m_{s_0,t}$, we apply $s_0$ or $t$ to \eqref{braid2} (depending on the parity of $m_{s_0,t}$), to obtain
\[
\underbrace{\cdots s_0ts_0}_{m_{s_0,t}}(e_1)=e_1+e_{s_0}+\cdots +e_{\underbrace{\cdots s_0ts_0}_{m_{s_0,t}-1}},
\]
using the observation that $w\colonequals\underbrace{\cdots s_0ts_0}_{m_{s_0,t}}$  is not in $c_\subreg$, since the braid relation (\ref{braid}) gives two reduced expressions for $w$. 
\end{proof}

Lusztig actually defines a representation of the affine Hecke algebra $\cH_q(\hatW)$, whose specialization is $E^0_\subreg$. The extension $\widetilde E^0_\subreg$ can also be defined as a $\cH_q(\hatW)$-module. More generally, one can define the following bimodule (whose well-definedness follows from the same argument as in Proposition-Definition~\ref{prop-def-E}):
\begin{defn}
    Let $\widetilde E_\subreg$ be the $\cH_q(\hatW)$-bimodule with $\Z[q^{\pm1/2}]$-basis $\{e_w:w\ge_{LR}c_\subreg\}=\{e_1\}\cup\{e_w:w\in c_\subreg\}$ such that for $w\in c_\subreg$ and $t\in \widehat S$,
    \begin{align*}
    T_t(e_w)&=\begin{cases}
        -e_w&\text{if }\ell(tw)<\ell(w)\\
        \displaystyle qe_w+q^{1/2}\sum_{\substack{y\in c_\subreg,ty<y,\\
        y=sw\text{ for some }s\in\widehat S}}e_y&\text{if }\ell(tw)>\ell(w).
    \end{cases}\\
    (e_w)T_t&=\begin{cases}
        -e_w&\text{if }\ell(wt)<\ell(w)\\
        \displaystyle qe_w+q^{1/2}\sum_{\substack{y\in c_\subreg,yt<y,\\
        y=ws\text{ for some }s\in\widehat S}}e_y&\text{if }\ell(wt)>\ell(w),
    \end{cases}
\end{align*}
and
\begin{equation*}
    T_t \cdot e_1=e_1 \cdot T_t=qe_1+q^{1/2}e_t.
    \end{equation*}
Let $E_\subreg$ be the sub-bimodule spanned by $\{e_w:w\in c_\subreg\}$.
\end{defn}

The $\cH_q(\hatW)$-bimodule $\widetilde E_\subreg$ contains $\widetilde E^0_\subreg$ as a left $\cH_q(\hatW)$-submodule. Then we see:
\begin{prop}\label{prop:iso-E-and-H}
    There is an isomorphism of $\cH_q(\hatW)$-bimodules between $\widetilde E_\subreg$ and 
    \begin{equation*}
\cH_q(\hatW)_{\ge_{LR}\mathrm{subreg}}\colonequals\cH_q(\hatW)/\operatorname{Span}_{\Z[q^{\pm 1/2}]}(C_v\,|\, v \notin c_{\mathrm{subreg}} \cup \{1\})
\end{equation*}
and between $E_\subreg$ and 
\begin{equation*}
\cH_q(\hatW)_{\mathrm{subreg}}\colonequals\operatorname{Span}_{\Z[q^{\pm 1/2}]}(C_v\,|\, v \neq 1)/\operatorname{Span}_{\Z[q^{\pm 1/2}]}(C_v\,|\, v \notin c_{\mathrm{subreg}} \cup \{1\}),
\end{equation*}
sending $e_w$ to $C_w$.
\end{prop}
\begin{proof}
    Follows from \cite[Lemma~4.1]{xu}.
\end{proof}
\begin{remark}
    Although Proposition~\ref{prop:iso-E-and-H} provides an explicit description for the $\cH$-bimodule $\cH_{\ge_{LR}\subreg}$ just using algebraic methods. However, in this paper we characterize the subregular anti-spherical module \emph{geometrically}. Such a characterization has independent interest since it allows to {\emph{categorify}} the canonical basis elements; we explicitly describe the corresponding irreducible perverse exotic sheaves in Section \ref{sec:explicit-description-of-the-canonical-basis}. Moreover, the geometric approach also helps in computing Kazhdan-Lusztig polynomials $\mathbf m_v^w$. Indeed, to compute them, we need to express $w(e_1)\in\widetilde E^0_\subreg$ in terms of the canonical basis. To do so, we can write a reduced expression $s_{i_1}\cdots s_{i_k}$ for $w$ and compute $s_{i_1}\cdots s_{i_k}(e_1)$. However, doing so naively does not give a nice general formula. Only by using our geometric characterization can we give general formulas. The second author plans to return to the problem of using Proposition~\ref{prop:iso-E-and-H} directly to compute $\mathbf m_v^w$ for all $w\in c_\subreg$, where geometric methods are not available.
\end{remark}

\subsection{Explicit description of the subregular cell}

Recall that \cite[\S3.13]{lusztig-some-examples} explicitly computes $c_\subreg^0$, which we explicitly write as $w_\nu$ for $\nu\in Q^\vee$. This is done by a case-work on the types:

\subsection{When $G$ is of type B}

Let $G$ be of type $\text{B}_n$ where $n\geqslant3$, so $\widehat W= Q^\vee\rtimes W$ is the affine Weyl group of type B. Explicitly, $W$ is viewed as the group of signed permutations, i.e., permutations on $\{1,\dots,n,\overline 1,\dots,\overline n\}$ fixing the subsets $\{i,\overline i\}$ for each $1\leqslant i\leqslant n$. Recall that the Coxeter-Dynkin diagram of $\widehat W= Q^\vee\rtimes W$ is
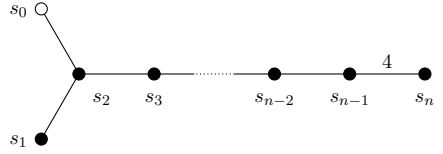
\[
\dynkin[extended,Coxeter,root
radius=.08cm,edge length=1cm,labels={s_0,s_1,s_2,s_3,s_{n-2},s_{n-1},s_n}]B{}.
\]
Explicitly, the simple reflections in $ Q^\vee\rtimes W$ are:
\begin{align*}
    s_0&\colonequals(\epsilon_1+\epsilon_2).(1\mapsto\overline 1,2\mapsto\overline 2)\\
    s_i&\colonequals(i,i+1)\text{ for }1\leqslant i\leqslant n-1\\
    s_n&\colonequals(n\mapsto\overline n).
\end{align*}

\begin{prop}\label{typec-subregular-cell}
Let $\widehat W= Q^\vee\rtimes W$ be the affine Weyl group of type $\text{B}_n$ where $n\geqslant3$. Then the elements of $c_\subreg^0$ are one of:
\begin{align}
s_0&=w_{\epsilon_1+\epsilon_2}\label{cell-case1}\\
s_1s_2s_0&=w_{\epsilon_2+\epsilon_3}\label{cell-case2}\\
s_is_{i-1}\cdots s_2s_0&=w_{\epsilon_1+\epsilon_{i+1}}\text{ for }2\leqslant i<n\label{cell-case3}\\
s_i\cdots s_{n-1}s_ns_{n-1}\cdots s_2s_0&=\begin{cases}w_{\epsilon_1-\epsilon_i}\text{ for }2\leqslant i\leqslant n\\w_{\epsilon_2-\epsilon_1}\text{ for }i=1\end{cases}\label{cell-case4}\\
s_0s_2\cdots s_{n-1}s_ns_{n-1}\cdots s_2s_0&=w_{2\epsilon_2}.\label{cell-case5}
\end{align}
The graph $\Gamma_0$ is:
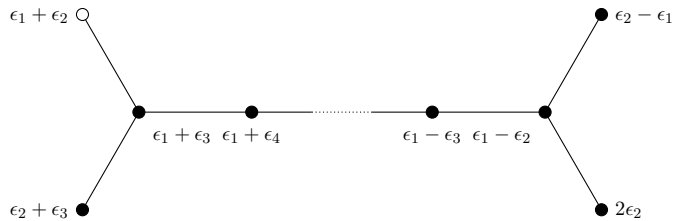
\[
\dynkin[extended,Coxeter,root
radius=.08cm,edge length=1.5cm,labels={\epsilon_1+\epsilon_2,\epsilon_2+\epsilon_3,\epsilon_1+\epsilon_3,\epsilon_1+\epsilon_4,\epsilon_1-\epsilon_3,\epsilon_1-\epsilon_2,\epsilon_2-\epsilon_1,2\epsilon_2}]D{}
\]
\end{prop}

\subsection{When $G$ is of type C}

When $\fg$ is of type C, with affine Coxeter diagram labelled as follows:
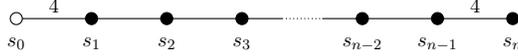
\[
\dynkin[extended,Coxeter,root
radius=.08cm,edge length=1cm,labels={s_0,s_1,s_2,s_3,s_{n-2},s_{n-1},s_n}]C{}
\]
Explicitly, write $Q^\vee=\{\sum_{i=1}^na_i\epsilon_i:a_i\in\Z\}$, with $W=(\Z/2)^n\rtimes S_n$ the group of signed permutations, so:
\begin{align*}
    s_0&=(\epsilon_1).(1\mapsto\overline 1)\\
    s_i&=(i,i+1)\ \text{ for }\ 1\leqslant i\leqslant n-1\\
    s_n&=(n\mapsto\overline n).
\end{align*}

Similar to type A \cite[Cor~5.2, Lem~6.1]{BKK}, the subregular cell is infinite:
\begin{prop}\label{typeb-subregular-cell}
Let $\widehat W=Q^\vee\rtimes W$ be the affine Weyl group of type C. Then the elements of $c_\subreg\cap\widehat W^f$ are either:
\begin{itemize}
    \item $w_{\epsilon_1+\epsilon_2}$; or
    \item of the form $w_{k\epsilon_i}$ where $k\ne0\in\Z$ and $1\leqslant i\leqslant n$.
\end{itemize}
The Bruhat-ordering is given by $w_{k\epsilon_i}\prec w_{\ell\epsilon_j}$ if and only if $N_{k,i}<N_{\ell,j}$, where
\[
N_{k,i}\colonequals\begin{cases}
2kn+1-i&k\equiv0\mod2,k>0\\
-2(k+1)n+i&k\equiv0\mod2,k<0\\
2(k-1)n+i&k\equiv1\mod2,k>0\\
-2kn+1-i&k\equiv1\mod2,k<0.
\end{cases}
\]
\end{prop}
\begin{proof} Direct calculation shows:
\begin{align*}
s_is_{i-1}\cdots s_0&=(\epsilon_{i+1}).(i+1,\overline i,\dots,\overline 1)\\
s_is_{i+1}\cdots s_ns_{n-1}\cdots s_0&=(-\epsilon_i).(i,i-1,\dots,1).
\end{align*}
In particular, note that $s_1s_2\cdots s_{n-1}s_ns_{n-1}\cdots s_0=-\epsilon_1$. From Lemma~\ref{graph-theory}, all elements of $c_\subreg\cap\widehat W$ are of the form (for some $N\geqslant0$):
\begin{align*}
s_0s_1s_0&=(\epsilon_1+\epsilon_2)(\overline 1,2)\\
s_is_{i-1}\cdots s_0(s_1s_2\cdots s_n\cdots s_0)^N&=(\epsilon_{i+1}).(i+1,\overline i,\dots,\overline 1).(-N\epsilon_1)\\
&=(N+1)\epsilon_{i+1}.(i+1,\overline i,\dots,\overline 1)\\
s_is_{i-1}\cdots s_n\cdots s_0(s_1s_2\cdots s_n\cdots s_0)^N&=(-\epsilon_{i}).(i,i-1,\dots,1).(-N\epsilon_1)\\
&=(-(N+1)\epsilon_i).(i,i-1,\dots,1).
\end{align*}
Now $N_{k,i}$ counts the number of simple reflections in the reduced expression of $w_{k\epsilon_i}$, as can be checked by case-work.
\end{proof}

\subsection{When $G$ is of type F\textsubscript{4}} 
Analogously, in type F\textsubscript{4}, where the Coxeter diagram of the affine Weyl group is labelled as follows:
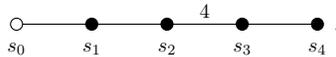
\[
\dynkin[extended,Coxeter,root
radius=.08cm,edge length=1cm,labels={s_0,s_1,s_2,s_3,s_4}]F4.
\]
We have:
\begin{prop}\label{typef-subregular-cell}
Let $\widehat W= Q^\vee\rtimes W$ be the affine Weyl group of type F\textsubscript{4}. Then the elements of $c_\subreg^0$ are one of:
\begin{align}
s_0&=w_{\epsilon_1+\epsilon_2}\label{cell-case1f}\\
s_1s_0&=w_{\epsilon_1+\epsilon_3}\label{cell-case2f}\\
s_2s_1s_0&=w_{\epsilon_1+\epsilon_4}\label{cell-case3f}\\
s_3s_2s_1s_0&=w_{\epsilon_1-\epsilon_4}\label{cell-case4f}\\
s_4s_3s_2s_1s_0&=w_{\epsilon_2+\epsilon_3}\label{cell-case5f}\\
s_2s_3s_2s_1s_0&=w_{\epsilon_1-\epsilon_3}\label{cell-case6f}\\
s_1s_2s_3s_2s_1s_0&=w_{\epsilon_1-\epsilon_2}\label{cell-case7f}\\
s_0s_1s_2s_3s_2s_1s_0&=w_{2\epsilon_1},\label{cell-case8f}
\end{align}
with graph $\Gamma_0$:
\[
\dynkin[extended,Coxeter,root
radius=.08cm,edge length=1.6cm,labels={(1,1,0,0),(1,0,1,0),(0,1,1,0),(1,0,0,1),(1,0,0,-1),(1,0,-1,0),(1,-1,0,0),(2,0,0,0)}]E7.
\]
\end{prop}

\subsection{When $G$ is of type G\textsubscript{2}}

\begin{prop}\label{typeg-subregular-cell}
Let $\widehat W= Q^\vee\rtimes W$ be the affine Weyl group of type G\textsubscript{2}. Then the elements of $c_\subreg^0$ are one of:
\begin{align}
s_0&=w_{2\alpha^\vee+\beta^\vee}\\
s_1s_0&=w_{\alpha^\vee+\beta^\vee}\\
s_2s_1s_0&=w_{\alpha^\vee}\\
s_1s_2s_1s_0&=w_{-\alpha^\vee}\\
s_0s_1s_2s_1s_0&=w_{3\alpha^\vee+\beta^\vee}\\
s_2s_1s_2s_1s_0&=w_{-\alpha^\vee-\beta^\vee}\\
s_1s_2s_1s_2s_1s_0&=w_{-2\alpha^\vee-\beta^\vee}\\
s_0s_1s_2s_1s_2s_1s_0&=w_{4\alpha^\vee+2\beta^\vee},
\end{align}
with graph $\Gamma_0$:
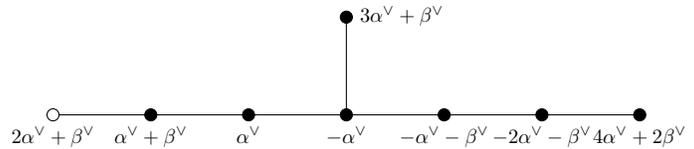
\[
\dynkin[extended,Coxeter,root
radius=.08cm,edge length=1.3cm,labels={2\alpha^\vee+\beta^\vee,\alpha^\vee+\beta^\vee,3\alpha^\vee+\beta^\vee,\alpha^\vee,-\alpha^\vee,-\alpha^\vee-\beta^\vee,-2\alpha^\vee-\beta^\vee,4\alpha^\vee+2\beta^\vee}]E7.
\]
\end{prop}

\section{Explicit description of the canonical basis}\label{sec:explicit-description-of-the-canonical-basis}

\subsection{Irreducible objects in the exotic $t$-structure in $D^{\bounded}(\operatorname{Coh}^{G^\vee}\!(\widetilde{U}))$}

We start with the definition of the exotic $t$-structures we are dealing with. For a group $H$ and an $H$-equivariant quasi-coherent sheaf of algebras $\mathcal{A}$ on some variety $X$ we will denote by $\mathcal{A}-\operatorname{mod}^H$ the category of finitely generated $H$-equivariant quasi-coherent sheaves of $\mathcal{A}$-modules on $X$. 

\begin{thm}(see \cite[Theorem 4.5]{BKK})\label{thm_descr_irred}
There exists a $G^\vee$-equivariant vector bundle $\mathcal{E}$ on $\widetilde{\mathcal{N}}$ with the following properties.
\begin{enumerate}
    \item The structure sheaf $\cO_{\widetilde{\cN}}$ is a direct summand in $\mathcal{E}$. Also, $\cE^*$ is globally generated. 
    \item Let $A\colonequals\operatorname{End}(\cE)$. Then the functor $\cF \mapsto \operatorname{RHom}(\cE^*,\cF)$ provides equivalence
\begin{equation*}
D^{\bounded}(\operatorname{Coh}^{G^\vee}\!(\widetilde{\cN})) \simeq D^{\bounded}(A\mathrm{-mod}^{G^\vee}).
\end{equation*}
\item\label{thm4.5-item3} Irreducible objects in the heart of the exotic $t$-structure are in bijection with pairs $({\mathbb{O}},M)$ where ${\mathbb{O}}$ is a $G^\vee$-orbit in $\cN$ and $M$ is an irreducible $G^\vee$-equivariant module for $A|_{{\mathbb{O}}}$. The corresponding irreducible object is supported on $\overline{\mathbb{O}}$.
\end{enumerate}
\end{thm}
By \cite[\S6.2]{bezrukavnikov-mirkovic}, the perverse exotic $t$-structure in $D^{\bounded}(\operatorname{Coh}^{G^\vee}(\widetilde{\cN}))$ is the image of the perverse coherent $t$-structure in $D^{\bounded}(A\mathrm{-mod}^{G^\vee})$ of middle perversity. The same description holds upon restricting to $U\subset\cN$. Abusing notations, we denote the natural morphism $\widetilde{U} \rightarrow U$ by $\pi$.
Set $\cE_{\widetilde{U}}\colonequals\cE|_{\widetilde{U}}$ and $A_U\colonequals\pi_*End(\cE_{\widetilde{U}})$ (considered as a sheaf of algebras on $U$). Then the functor $\cF \mapsto \pi_*(RHom(\cE_U^*,\cF))$ provides the derived equivalence 
\begin{equation}\label{equiv_U_t_str}
D^{\bounded}\big(\!\operatorname{Coh}^{G^\vee}\!(\widetilde{U})\big) \simeq D^{\bounded}(A_U\mathrm{-mod}^{G^\vee}),
\end{equation}
which sends the exotic $t$-structure on the left to the perverse coherent $t$-structure on the right.

Let $D^{\bounded}(A_{\mathbb{O}_{\mathrm{subreg}}}\mathrm{-mod}^{G^\vee})$ be the subcategory of $D^{\bounded}(A_{U}\mathrm{-mod}^{G^\vee})$ consisting of complexes of finitely generated $A_U$-modules on $U$ that are set-theoretically supported on $\mathbb{O}_{\mathrm{subreg}}$.
The equivalence (\ref{equiv_U_t_str}) induces the equivalence 
\begin{equation*}
D^{\bounded}\big(\!\operatorname{Coh}^{G^\vee}_{\widetilde{\mathbb{O}}_{\mathrm{subreg}}}\!(\widetilde{U})\big) \simeq D^{\bounded}(A_{\mathbb{O}_{\mathrm{subreg}}}\mathrm{-mod}^{G^\vee}), 
\end{equation*}
where $D^{\bounded}(\operatorname{Coh}^{G^\vee}_{\widetilde{\mathbb{O}}_{\mathrm{subreg}}}(\widetilde{U}))$ is the subcategory of $D^{\bounded}(\operatorname{Coh}^{G^\vee}\!(\widetilde{U}))$ consisting of complexes that are set-theoretically supported on $\widetilde{\mathbb{O}}_{\mathrm{subreg}}$.

Recall that we have a short exact sequence of categories
\begin{equation*}
0 \rightarrow D^{\bounded}\big(\!\operatorname{Coh}^{G^\vee}_{\widetilde{\mathbb{O}}_{\subreg}}(\widetilde{U})\big) \rightarrow D^{\bounded}\big(\!\operatorname{Coh}^{G^\vee}\!(\widetilde{U})\big) \rightarrow D^{\bounded}\big(\!\operatorname{Coh}^{G^\vee}\!(\widetilde{\mathbb{O}}_{\mathrm{reg}})\big) \rightarrow 0
\end{equation*}
compatible with the exotic $t$-structures (see \cite[Section 6.2.2]{bezrukavnikov-mirkovic}). It is easy to see that the object $\cO_{\widetilde{U}} \in D^{\bounded}(\operatorname{Coh}^{G^\vee}\!(\widetilde{U}))$ is an irreducible object of the exotic $t$-structure. All other irreducible objects lie in $D^{\bounded}(\operatorname{Coh}^{G^\vee}_{\widetilde{\mathbb{O}}_{\mathrm{subreg}}}(\widetilde{U}))$.

\begin{lemma}\label{lem:irred_supp_on_preimage}
Every irreducible object of the perverse exotic $t$-structure in $D^{\bounded}(\operatorname{Coh}^{G^\vee}_{\widetilde{\mathbb{O}}_{\mathrm{subreg}}}(\widetilde{U}))$ is scheme-theoretically supported on $\widetilde{\mathbb{O}}_{\mathrm{subreg}}$ i.e., it lies in $D^{\bounded}(\operatorname{Coh}^{G^\vee}\!(\widetilde{\mathbb{O}}_{\mathrm{subreg}}))$.
\end{lemma}
\begin{proof}
Follows from Theorem~\ref{thm_descr_irred}~\eqref{thm4.5-item3} together with the fact that $\mathbb{O}_{\mathrm{subreg}} \subset U$ is closed.
\end{proof}

Let $\widehat{\cB_e}$, $\widehat{\widetilde{\mathbb{O}}_e}$ be the formal neighbourhoods of $\cB_e$, $\widetilde{\mathbb{O}}_e$ as in \cite[Section 0.1]{bezrukavnikov-mirkovic}. Set $\cE_{\widehat{\cB_e}}\colonequals\cE|_{\widehat{\cB_e}}$, $\widehat{A_e}\colonequals\operatorname{End}(\cE_{\widehat{\cB_e}})$.
By \cite[Section 5.2.1]{bezrukavnikov-mirkovic}, we have the equivalence 
\begin{equation*}
D^{\bounded}(\operatorname{Coh}^{\Cent(e)}(\widehat{\cB_e})) \simeq D^{\bounded}(\widehat{A_e} \mathrm{-mod}^{\Cent(e)}).
\end{equation*}
This equivalence defines exotic $t$-structure in $D^{\bounded}(\operatorname{Coh}^{\Cent(e)}(\widehat{\cB_e}))$.

\begin{lemma}\label{lem:ident_via_ind}
There exists an equivalence 
\begin{equation*}
D^{\bounded}(\operatorname{Coh}^{\Cent(e)}(\widehat{\cB_e})) \iso D^{\bounded}(\operatorname{Coh}^{G^\vee}\!(\widehat{\widetilde{\mathbb{O}}_e})),\, L \mapsto \operatorname{Ind}_{\Cent(e)}^{G^\vee}L\colonequals G^{\vee} \times^{\Cent(e)} L.
\end{equation*}
This equivalence is compatible with exotic $t$-structures.
\end{lemma}
\begin{proof}
The first claim is standard, compatibility with $t$-structures follows from the definitions.
\end{proof}

\begin{lemma}
\label{lem:forg_funct_bij_irred} The forgetful functor $D^{\bounded}(\operatorname{Coh}^{\Cent(e)}(\widehat{\cB_e})) \rightarrow D^{\bounded}(\operatorname{Coh}^{\Cent_e}(\widehat{\cB_e}))$ is bijective on irreducible objects (of the corresponding exotic $t$-structures).
\end{lemma}
\begin{proof}
Follows from \cite[Section 6.3.2]{bezrukavnikov-mirkovic}.
\end{proof}

So, our goal is to describe irreducible objects  in $D^{\bounded}(\operatorname{Coh}^{\Cent_e}(\cB_e)) \hookrightarrow D^{\bounded}(\operatorname{Coh}^{\Cent_e}(\widehat{\cB_e}))$.
 It follows from Lemma \ref{lem:forg_funct_bij_irred} that the corresponding irreducible objects in $D^{\bounded}(\operatorname{Coh}^{\Cent(e)}(\cB_e))$ will be the same complexes with additional equivariance (that we explicitly describe case-by-case in \S\ref{subsec:type-B-canonical-basis}, \S\ref{subsec:type-C-canonical-basis}, \S\ref{subsec:type-F-canonical-basis}, and \S\ref{subsec:type-G-canonical-basis}). The corresponding irreducibles in $D^{\bounded}(\operatorname{Coh}^{G^\vee}\!(\widetilde{\mathbb{O}}_e))$ are $\operatorname{Ind}^{G^\vee}_{\Cent(e)} L$ (see Lemma \ref{lem:ident_via_ind}).

Recall that $\widetilde{S}$ is the Slodowy variety. Let $D^{\bounded}(\operatorname{Coh}^{\Cent_e}_{\cB_e}(\widetilde{S}))$ be the category of coherent $\Cent_e$-equivariant sheaves on $\widetilde{S}$ that are set-theoretically supported on $\cB_e$, equipped with the exotic $t$-structure (as in \cite[Section 5.2.1]{bezrukavnikov-mirkovic}). We have an embedding  $D^{\bounded}(\operatorname{Coh}^{\Cent_e}(\cB_e))  \hookrightarrow D^{\bounded}(\operatorname{Coh}^{\Cent_e}_{\cB_e}(\widetilde{S}))$ inducing bijection on irreducible objects. So, it remains to describe irreducible objects in $D^{\bounded}(\operatorname{Coh}^{\Cent_e}_{\cB_e}(\widetilde{S}))$.

\begin{prop}\label{prop:descr_irred_t_str_B_e}
Let $\cP$ be a $\Cent_e$-equivariant vector bundle on $\widetilde{S}$ such that the structure sheaf $\cO_{\widetilde{S}}$ is a direct summand in $\cP$, $\cP^*$ is globally generated, and $\cP$ is a tilting generator. Then, the irreducible objects of the induced $t$-structure in $D^{\bounded}(\operatorname{Coh}^{\Cent_e}_{\cB_e}(\widetilde{S}))$ are as in Propositions \ref{lem:canonical-basis-b}, \ref{lem:canonical-basis-c}, \ref{lem:canonical-basis-f}, \ref{lem:canonical-basis-g}. 
\end{prop}
\begin{proof}
The argument is similar to the one in \cite[Lemma 4.7]{BKK}. 

Let $\mathcal{A}_e \subset D^{\bounded}(\operatorname{Coh}^{\Cent_e}_{\cB_e}(\widetilde{S}))$ be the heart of the exotic $t$-structure. Let $\mathcal{A}^0_e \subset \mathcal{A}_e$ be the subcategory consisting of objects $\cF$ such that $R\Gamma(\cF)=0$. This is indeed a subcategory because $\cO_{\widetilde{S}}$ is the direct summand of $\cP$.

Note now that all our candidates for simple objects indeed lie in $\mathcal{A}_e$ as after forgetting the equivariance, the corresponding objects lie in heart (use \cite[Lemma 4.7]{BKK}) and the functor providing equivalence is the same. To see that the candidates that we have lying in $\mathcal{A}_e$ are indeed simple it is enough to note that every object of $\mathcal{A}_e^0$ is of the form $\cF[1]$, where $\cF$ is an extension of $\cO_{\bP_\alpha}$ (compare with the proof of \cite[Lemma 4.7]{BKK}).

Now it remains to describe irreducible objects of $\mathcal{A}_e$ that do not lie in $\mathcal{A}^{0}_e$. Note that as in the proof of \cite[Lemma 4.7]{BKK} we can count the number of them (up to a twist by $\Cent_e^0$ that appears only in type $C$) and the answer will be equal precisely to the number of $\Cent_e$-equivariant sheaves of the form $\cO_{\cB_e} \otimes ?$ that we have in Propositions \ref{lem:canonical-basis-b}, \ref{lem:canonical-basis-c}, \ref{lem:canonical-basis-f}, \ref{lem:canonical-basis-g}.

It remains to check that all of the objects above are indeed irreducible. Note that the functor $\operatorname{Forg}\colon D^{\bounded}(\operatorname{Coh}^{\Cent_e}_{\cB_e}(\widetilde{S})) \rightarrow D^{\bounded}(\operatorname{Coh}_{\cB_e}(\widetilde{S}))$ forgetting the equivariance is $t$-exact, so if object $L$ is not irreducible, then $\operatorname{Forg}(L)$ is also not irreducible. The only problem may appear when we deal with $L=\cO_{\cB_e} \otimes \xi$, where $\xi$ is a two-dimensional irreducible representation of $\Cent_e$. In this case, $\operatorname{Forg}(L)=\cO_{\cB_e}^{\oplus 2}$ so the only proper submodule of $\operatorname{Forg}(L)$ is $\cO_{\cB_e}$. Let $F$ be $\cO_{\cB_e}$ with any equivariant structure. It remains to check that every $\Cent_e$-equivariant map $F \rightarrow \cO_{\cB_e} \otimes \xi$ is equal to zero. First of all this map is determined by the map it induces on global sections (we use that $\operatorname{Hom}(\cO_{\cB_e},\cO_{\cB_e})=\mathbb{C}$). On global sections, we obtain a map from some one-dimensional representation of $\Cent_e$ to the irreducible two-dimensional representation $\xi$. Any such map should be zero.
To see that $\operatorname{Hom}(\cO_{\cB_e},\cO_{\cB_e})=\mathbb{C}$ let us rewrite it as $\operatorname{Hom}(\pi^*\mathbb{C}_e,\cO_{\cB_e})=\operatorname{Hom}(\mathbb{C}_e,\pi_*\cO_{\cB_e})=\operatorname{Hom}(\mathbb{C},\mathbb{C})$ as global sections of $\cO_{\cB_e}$ is $\mathbb{C}$ (see the end of the proof of \cite[Lemma 4.7]{BKK}). 
\end{proof}

\subsection{When $G$ is of type B}\label{subsec:type-B-canonical-basis}
Let $G$ be of type $\text{B}_n$ where $n\geqslant3$, so $G^\vee$ is of type $\text{C}_n$. Let $G^\vee=\Sp(V)/\{\pm1\}$ where $V=V(2n-3)\oplus V(1)$, and $\omega$ is a non-degenerate $\fsl_2$-equivariant symplectic form on $V$, which is unique up to scaling on each factor (for the notation, see \S\ref{sec:notation}). Then $e=\begin{pmatrix}0&1\\0&0\end{pmatrix}\in\fsl_2$ maps to a subregular nilpotent element $e\in\fg^\vee$.

The flag variety $\mathscr B$ of $G^\vee$ is identified with:
\begin{align*}
\mathscr B\colonequals&\ \{0\subsetneq V_1\cdots\subsetneq V_n\subset V:\dim(V_i)=i,\text{ and }\omega(V_n,V_n)=0\}\\
=&\ \{0\subsetneq V_1\cdots\subsetneq V_{2n-1}\subsetneq V:V_i^\perp=V_{2n-i}\},
\end{align*}
where $V_i^\perp\colonequals\{v\in V:\omega(v,V_i)=0\}$ is the complement with respect to $\omega$. For convenience, also let $V_0\colonequals0$ and $V_{2n}\colonequals V$. The Springer fiber $\mathscr B_e$ is:
\begin{align*}
\mathscr B_e\colonequals&\ \{(V_i)\in\mathscr B:eV_i\subset V_{i-1}\text{ for }i=1,\dots,n\}\\
=&\ \{(V_i)\in\mathscr B:eV_i\subset V_{i-1}\text{ for }i=1,\dots,2n\}.
\end{align*}

Moreover, for each integer $j\leqslant n$, let $\cP_j$ be the partial flag variety of type $j$, i.e.,
\[
\cP_j\colonequals\{0\subsetneq V_i\subsetneq V\text{ for }i\ne j,2n-j:V_i^\perp=V_{2n-i},\dim(V_i)=i,\text{ and for all }i_1<i_2,V_{i_1}\subsetneq V_{i_2}\}.
\]
The natural projection $\cB\to\cP_j$ simply forgets $V_j$ and $V_{2n-j}$.

\begin{prop}\label{explicit}
Let $e\in\mathfrak{sp}_{2n}$ be subregular nilpotent. Then $\mathscr B_e$ is a union of $n+1$ copies of $\bP^1$, and the dual graph of the components is:
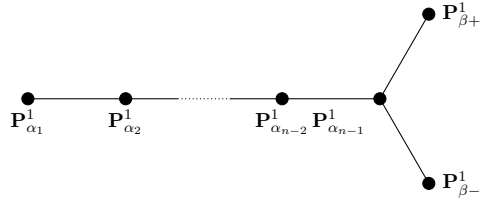
\[
\dynkin[root
radius=.08cm,edge length=1.3cm,labels={\bP^1_{\alpha_1},\bP_{\alpha_2}^1,\bP^1_{\alpha_{n-2}},\bP^1_{\alpha_{n-1}},\bP^1_{\beta+},\bP^1_{\beta-}}]D{}
\]
\end{prop}
\begin{proof}
The $n+1$ copies of $\bP^1$ are given explicitly as follows:
\begin{itemize}
\item For each $j<n$, consider the partial flag:
\[
V_i^{(j)}\colonequals\begin{cases}\bigoplus_{k=1}^i V(2n-3)_{2n-2k-1}&1\leqslant i<j\\
V(1)_1\oplus \bigoplus_{k=0}^{i-1}V(2n-3)_{2n-2k-1}&j<i\leqslant n.
\end{cases}
\]
Then let $\bP_{\alpha_j}^1\colonequals\pi_j^{-1}((V_i^{(j)})_{1\leqslant i\leqslant n,i\ne j})$.
\item For $j=n$, there are nonzero $v\in V(1)_1$ and $w\in V(2n-3)_{1}$ and for $1\leqslant i<n$, we may let
\[
V^{\pm}_i\colonequals\begin{cases}
\bigoplus_{k=1}^i V(2n-3)_{2n-2k-1}&i<n-1\\
\C(v\pm w)\oplus \bigoplus_{k=1}^{n-2} V(2n-3)_{2n-2k-1} &i=n-1.
\end{cases}
\]
Then let $\bP^1_{\beta\pm}\colonequals\pi_n^{-1}((V_i^{\pm})_{1\leqslant i<n})$.
\end{itemize}
Now, $\cB_e$ is the union $\bP_{\alpha_1}^1\cup\cdots\cup\bP_{\alpha_{n-1}}^1\cup\bP^1_{\beta+}\cup\bP^1_{\beta-}$.
\end{proof}

Note that $\Cent_e\simeq\Z/2$ is generated by $\sigma\in\Sp(V)$ which acts as $-1$ on $V(1)$ and as $1$ on $V(2n-3)$. 

\begin{lemma}\label{lem:fixed-pts}
The fixed points of the $\sigma$-action on $\cB_e$ are the points $x_j=(W_i^{(j)})_{1\leqslant j\leqslant n}$ for $1\leqslant j\leqslant n$ corresponding to the flag
\[
W_i^{(j)}\colonequals\begin{cases}
\bigoplus_{k=1}^iV(2n-3)_{2n-2k-1}&i<j\\
V(1)_1\oplus\bigoplus_{k=1}^{i-1}V(2n-3)_{2n-2k-1}&i\geqslant j.
\end{cases}
\]
When $1<j<n$, we have $x_j=\bP^1_{\alpha_{j-1}}\cap\bP^1_{\alpha_j}$.
\end{lemma}

Now, we may explicitly describe the $\Z/2$-action:
\begin{itemize}
\item for $1\leqslant j<n$, there is an isomorphism $\bP^1_{\alpha_j}\simeq\bP(V(1)_1\oplus V(2n-3)_{2n-2j-1})$, so $\sigma$ acts by an order $2$ involution, fixing $x_{j-1}$ and $x_j$.
\item there are isomorphisms $\bP^1_{\beta\pm}\simeq\bP(\C(v\pm w)\oplus V(2n-3)_{-1})$, and $\sigma$ swaps the two copies. 
\end{itemize}
Moreover, for convenience, let $y_\pm\colonequals\bP^1_{\alpha_{n-1}}\cap\bP^1_{\beta\pm}$, so $y\colonequals\{y_+,y_-\}$ is a $\sigma$-invariant divisor of $\cB_e$.

We have the following explicit description of the canonical basis.

\begin{lemma}
    When $G$ is of type $\text{B}_n$ where $n\geqslant3$, the canonical basis of $D^{\bounded}(\operatorname{Coh}^{\Cent_e}(\cB_e))$ consists of:
\[
\cO_{\cB_e},\sgn\otimes\cO_{\cB_e},\text{ and }\cO_{\bP_{\alpha_i}^1}(-x_i)[1],\cO_{\bP_{\alpha_i}^1}(-x_{i+1})[1]\text{ for }1\leqslant i\leqslant n-1,\text{ and }\cO_{\bP_{\beta}}(-y)[1].
\]
\end{lemma}

\begin{proof}[Proof of Lemma~\ref{lem:forg_funct_bij_irred} when $G$ is of type B]
The sheaves $\cO_{\cB_e}$ and $\sgn\otimes\cO_{\cB_e}$ clearly have $\Cent(e)$-equivariance.   The $\Cent(e)$-equivariance is also clear for $\cO_{\bP_{\alpha_i}^1}(-x_i)[1]$ for $2\leqslant i\leqslant n-1$, for $\cO_{\bP_{\alpha_i}^1}(-x_{i+1})[1]$ for $1\leqslant i\leqslant n-2$, and for $\cO_{\bP_\beta}(-y)[1]$ since the divisors $x_i$ for $1<i<n$ and $y$ are all $\Cent(e)$-equivariant (as they are realized as intersections of components of $\cB_e$).

The only objects left are $\cO_{\bP^1_{\alpha_1}}\!(-x_1)[1]$ and $\cO_{\bP^1_{\alpha_{n-1}}}\!(-x_n)[1]$. But they have a $\Cent(e)$-equivariance arising from the isomorphisms:
\[
\cO_{\bP^1_{\alpha_1}}\!(-x_1)[1]\simeq \sgn\otimes\cO_{\bP^1_{\alpha_1}}\!(-x_2)[1],\hspace{0.3cm}\cO_{\bP^1_{\alpha_{n-1}}}\!(-x_n)[1]\simeq\sgn\otimes\cO_{\bP^1_{\alpha_{n-1}}}\!(-x_{n-1})[1].\qedhere
\]
\end{proof}
Thus, we have:
\begin{prop}\label{lem:canonical-basis-b}
When $G$ is of type $\text{B}_n$ where $n\geqslant3$, the irreducible exotic coherent sheaves of $D^{\bounded}(\operatorname{Coh}^{G^\vee}\!(\widetilde{U}))$ consists of:
\begin{align*}
&\cO_{\widetilde U},\ \cO_{\widetilde{\mathbb{O}}_e}[-1],\ \sgn\otimes\cO_{\widetilde{\mathbb{O}}_e}[-1],\ \operatorname{Ind}_{\Cent(e)}^{G^\vee}\cO_{\bP_{\alpha_i}^1}(-x_i),\\&\operatorname{Ind}_{\Cent(e)}^{G^\vee}\cO_{\bP_{\alpha_i}^1}(-x_{i+1})\text{ for }1\leqslant i\leqslant n-1,\text{ and }\operatorname{Ind}_{\Cent(e)}^{G^\vee}\cO_{\bP_{\beta}}(-y).
\end{align*}
\end{prop}

\subsection{When $G$ is of type C}\label{subsec:type-C-canonical-basis}

Let $G^\vee=\SO(V)$ where $V$ is a $2n+1$-dimensional vector space with a non-degenerate symmetric form $(-,-)$. Then the flag variety $\mathscr B$ is identified with:
\begin{align*}
\mathscr B\colonequals&\ \{0\subsetneq V_1\cdots\subsetneq V_n\subset V:\dim(V_i)=i,\text{ and }(V_n,V_n)=0\}\\
=&\ \{0\subsetneq V_1\cdots\subsetneq V_{2n}\subsetneq V:V_i^\perp=V_{2n+1-i}\}.
\end{align*}
For convenience, also let $V_0=0$ and $V_{2n}=V$. Given a nilpotent element $e\in\mathfrak{so}(V)$, the Springer fiber $\mathscr B_e$ is:
\[
\mathscr B_e\colonequals \{(V_i)\in\mathscr B:eV_i\subset V_{i-1}\text{ for }i=1,\dots,n\}=\{(V_i)\in\mathscr B:eV_i\subset V_{i-1}\text{ for }i=1,\dots,2n+1\}.
\]
Here, the two definitions agree, since $(ev,w)+(v,ew)=0$ for all $v,w\in V$ implies $eV_i\subset V_{i-1}$ and $eV_{i-1}^\perp\subset V_i^\perp$ are equivalent.

\begin{example}\label{ex:subregular-c}
The subregular nilpotent orbit corresponds to the partition $[2n-1,1^2]$, and is realized on $V=V(2n-2)\oplus V(0)^{\oplus2}$, with the symmetric form as above. The centralizer of the $\fsl_2$-triple is $\tO_2(\C)\simeq\C^\times\rtimes\Z/2$,\footnote{Here $\Z/2$ acts on $\C^\times$ by $z\mapsto z^{-1}$.} which acts tautologically on $V(0)^{\oplus 2}$, and via the determinant on $V(2n-2)$. Explicitly, let $V(0)^{\oplus 2}=\C e_+\oplus\C e_-$, with the invariant form $(e_+,e_+)=(e_-,e_-)=0$ and $(e_+,e_-)=1$. Then $(z,\pm1)\in\C^\times\rtimes\Z/2$ acts as $e_+\mapsto z^{\pm1}e_\pm$ and $e_-\mapsto z^{\mp1}e_\mp$.
\end{example}

Moreover, for each integer $j\leqslant n$, let $\cP_j$ be the partial flag variety of type $j$,\footnote{In the notation of \S\ref{sec:notation}, $\cP_j=\cP_{\alpha_j}$ for $j<n$ and $\cP_n=\cP_{\beta}$.} i.e.,
\[
\cP_j\colonequals\{0\subsetneq V_i\subsetneq V\text{ for }i\ne j,2n+1-j:V_i^\perp=V_{2n+1-i},\dim(V_i)=i,\text{ and for all }i_1<i_2,V_{i_1}\subsetneq V_{i_2}\}.
\]
The natural projection $\cB\to\cP_j$ simply forgets $V_j$ and $V_{2n+1-j}$.\footnote{One should \emph{a priori} expect $\cP_n$ to behave differently than $\cP_j$ for $j<n$, since the two forgotten vector spaces are consecutive.}

\begin{prop}\label{explicit-type-c}
Let $e\in\mathfrak{so}_{2n+1}$ be subregular nilpotent. Then $\mathscr B_e$ is a union of $2n-1$ copies of $\bP^1$.
\end{prop}
\begin{proof}
Identify $V$ with $\C e_+\oplus\C e_-\oplus V(2n-2)$, where $(e_+,e_+)=(e_-,e_-)=0$ and $(e_+,e_-)=1$. Then the image of $\begin{pmatrix}0&1\\0&0\end{pmatrix}\in\fsl_2$ in $\fg^\vee$ is subregular nilpotent. The $2n-1$ copies of $\bP^1$ are given as follows:
\begin{itemize}
\item For each $j<n$ and a sign $\pm$, consider the partial flags:
\[
V_i^{(j,\pm)}\colonequals\begin{cases}\bigoplus_{k=1}^i V(2n-2)_{2n-2k}&1\leqslant i<j\\
\C e_\pm\oplus \bigoplus_{k=1}^{i-1}V(2n-2)_{2n-2k}&j<i\leqslant n.
\end{cases}
\]
Then we let $\bP_{\alpha_j\pm}^1\colonequals\pi_j^{-1}((V_i^{(j,\pm)})_{1\leqslant i\leqslant n,i\ne j})$.
\item For $j=n$, for $i<n$, let
\[
V_i\colonequals\bigoplus_{k=1}^i V(2n-2)_{2n-2k}.
\]
Then let $\bP^1_{\beta}\colonequals\pi_n^{-1}((V_i)_{1\leqslant i<n})$. Here $\bP_\beta^1$ is identified with the flag variety of $\SO_3\simeq\PGL_2$, which is isomorphic to $\bP^1$.
\end{itemize}
Now, $\cB_e$ is the union $\bigcup_{j=1}^{n-1}(\bP_{\alpha_j+}^1\cup\bP_{\alpha_j-}^1)\cup\bP_\beta^1$.
\end{proof}

\begin{remark}
In the notation of \S\ref{sec:notation}, we have $\bP_{\alpha_j}=\bP^1_{\alpha_j+}\cup\bP^1_{\alpha_j-}$ and $\bP_\beta=\bP_\beta^1$.
\end{remark}

As in Example~\ref{ex:subregular-c}, note that $(z,\pm1)\in\Cent_e\simeq\tO_2=\C^\times\rtimes\Z/2$ acts on $V(2n-2)$ as $\pm1$, on $e_+$ as $e_+\mapsto z^{\pm1}e_\pm$, and on $e_-$ as $e_-\mapsto z^{\mp1}e_\mp$. Thus explicitly, $\Z/2$ acts on $\cB_e$ by swapping $\bP^1_{\alpha_j+}$ and $\bP^1_{\alpha_j-}$, and acts as an involution on $\bP^1_\beta$. The $\C^\times$-action is also expressed explicitly:
\begin{lemma}\label{lem:fixed-point-c}
The $\C^\times$-action on $\cB_e$ acts with weight one on all copies of $\bP^1$ in $\cB_e$, with fixed points $x_j^{\pm}$ (where $1\leqslant j\leqslant n$) corresponding to the following flag:
\[
W_i^{(j,\pm)}\colonequals\begin{cases}
\bigoplus_{k=1}^iV(2n-2)_{2n-2k}&i<j\\
\C e_\pm\oplus\bigoplus_{k=1}^{i-1}V(2n-2)_{2n-2k}&j\leqslant i\leqslant n.
\end{cases}
\]
The restriction of the $\C^\times$-action on $\bP_{\alpha_j}$ is attracting at $x_j^-$ and $x_{j+1}^+$, and repelling at $x_j^+$ and $x_{j+1}^-$, in the sense of \cite[\S2.10]{achar-book}, and the restriction of the $\C^\times$-action on $\bP_\beta$ is attracting at $x_n^-$ and repelling at $x_n^+$.
\end{lemma}
For convenience, we define:
\begin{defn}\label{defn:typec,xj}
Let $x_j\colonequals\{x_j^+,x_j^-\}$, which a $\Cent_e$-equivariant divisor in $\bP_{\alpha_j}=\bP_{\alpha_j+}^1\cup\bP^1_{\alpha_j-}$.
\end{defn}

The $\tO_2$-action on $\bP_\beta$ is isomorphic to $(a,\pm1)\in\tO_2=\G_m\rtimes\Z/2$ acting on $\bP^1$ via the M\"obius transformation $z\mapsto az^{\pm1}$.
Let $\mathrm{Pin}_2=\G_m\rtimes\Z/2\to \tO_2$ be the double cover, defined by $(z,\pm1)\mapsto (z^2,\pm1)$. It lies in a pullback square
\[ \begin{tikzcd}
\mathrm{Pin}_2\arrow[hook]{r} \arrow[swap]{d}{2:1} &\SL_2 \arrow{d} \\
\tO_2\arrow[hook]{r}&\PGL_2.
\end{tikzcd}
\]
For $k>0$ write $\xi_{k/2}$ for the two-dimensional representation of $\mathrm{Pin}_2$ obtained by inducing the representation $\C\langle k\rangle$ of $\G_m$ to $\mathrm{Pin}_2$. Note that the representation $\xi_{k/2}$ of $\mathrm{Pin}_2$ factors through $\tO_2$ exactly when $k$ is even.

\begin{remark}
    For any $k\geqslant0$, the object $\xi_{k+1/2}\otimes \cO_{\bP^1}(1,0)\in\Coh^{\mathrm{Pin}_2}(\bP^1)$ lives in the sub-category $\Coh^{\tO_2}(\bP^1)$. Indeed, $\ker(\mathrm{Pin}_2\to\tO_2)=\{\pm I_2\}\subseteq \GL_2$ acts as $-1$ on both $\cO_{\bP^1}(1,0)$ and $\xi_{k+1/2}$.
\end{remark}

\begin{lemma}
    Let $\tO_2=\G_m\rtimes\Z/2$ act on $\bP^1$ by $z\mapsto az^{\pm1}$. Then for any $\cF\in\Coh^{\tO_2}(\bP^1)$ whose underlying coherent sheaf is a direct sum of $\cO_{\bP^1}(-1)$, there are integers $k_1,\dots,k_n\geqslant0$ such that
    \[
    \cF\simeq(\xi_{k_1+1/2}+\cdots+\xi_{k_n+1/2})\otimes\cO_{\bP^1}(1,0).
    \]
\end{lemma}
\begin{proof}
By viewing $\cF$ as an object of $\Coh^{\mathrm{Pin}_2}(\bP^1)$, the tensor product $\cF\otimes\cO_{\bP^1}(-1,0)$ has underlying coherent sheaf isomorphic to a direct sum of $\cO$. Thus (by taking coherent cohomology) there is some representation $V$ of $\mathrm{Pin}_2$ such that $\cF\otimes\cO_{\bP^1}(-1,0)\simeq \cO_{\bP^1}\otimes V$. Now,
\(
\cF\simeq\cO_{\bP^1}(1,0)\otimes V
\)
living in the subcategory $\Coh^{\tO_2}(\bP^1)$ is equivalent to asking $-I_2\in\tO_2$ to act as $-1$ on $V$, which means $V$ must be a direct sum of $\xi_{k+1/2}$ for $k\geqslant0$.
\end{proof}
We have the following description of the canonical basis.  
\begin{lemma}
    When $G$ is of type $\text{C}_n$ where $n\geqslant2$, the canonical basis of $D^{\bounded}(\operatorname{Coh}^{\Cent_e}(\cB_e))$ consists of 
\begin{align*}
&\cO_{\cB_e},\ \sgn\otimes\cO_{\cB_e},\ \xi_k\otimes\cO_{\cB_e}\text{ for $k\in\Z_{>0}$},\\&\cO_{\bP_{\alpha_i}}(kx_i-(k+1)x_{i+1})[1]\text{ for $k\in\Z$ and $1\leqslant i\leqslant n-1$, and }\xi_{k+1/2}\otimes\cO_{\bP_\beta}(1,0)[1]\text{ for $k\in\Z_{\geqslant0}$}.
\end{align*}
\end{lemma}

\begin{proof}[Proof of Lemma~\ref{lem:forg_funct_bij_irred} when $G$ is of type C]
The $\Cent_e$-equivariance on $\cO_{\cB_e}$, $\sgn\otimes\cO_{\cB_e}$, and $\xi_k\otimes\cO_{\cB_e}$ is clear. Since $x_i$ for $i>1$ are characterized as intersection of components of $\cB_e$, they are $\Cent_e$-equivariant, and hence $\cO_{\bP_{\alpha_i}}(kx_i-(k+1)x_{i+1})[1]$ has $\Cent(e)$-equivariance for $i>1$.

To define $\Cent(e)$-equivariant structure on $\cO_{\bP_{\alpha_1}}(kx_1-(k+1)x_2)[1]$, it suffices to define $\Cent(e)^\circ$-equivariant structure on $\cO_{\bP_{\alpha_1+}}(kx_1^+-(k+1)x_2^+)[1]$. But then
\[\cO_{\bP_{\alpha_1+}}(kx_1^+-(k+1)x_2^+)[1]\simeq\C\langle-k\rangle\otimes\cO_{\bP_{\alpha_1+}}\!(-x_2^+)[1],\]
where $\C\langle-k\rangle$ is a representation of $\Cent(e)^\circ\twoheadrightarrow\G_m$ and $\cO_{\bP_{\alpha_1+}}\!(-x_2^+)$ has $\Cent(e)^\circ$-equivariance by the discussion above.

Finally, note that the unipotent radical of $\Cent(e)$ acts on $\bP_\beta^1$ with two fixed points, i.e., is given by a homomorphism to $\G_m$. But unipotent groups do not admit non-trivial homomorphisms to $\G_m$, so the action of $\Cent(e)$ on $\bP_\beta^1$ factors through $\Cent_e$. Thus any $\Cent_e$-equivariant sheaf on $\cB_e$ such as $\xi_{k+1/2}\otimes\cO_{\bP_\beta}(1,0)[1]$ admit a $\Cent(e)$-equivariant structure.
\end{proof}

Thus, we have:
\begin{prop}\label{lem:canonical-basis-c}
When $G$ is of type $\text{C}_n$ where $n\geqslant2$, the irreducible exotic coherent sheaves of $D^{\bounded}(\operatorname{Coh}^{G^\vee}\!(\widetilde{U}))$ consists of
\begin{align*}
\cO_{\widetilde U},\ \cO_{\widetilde{\mathbb{O}}_\subreg}[-1],\ \sgn\otimes\cO_{\widetilde{\mathbb{O}}_\subreg}[-1],\ \xi_k\otimes\cO_{\widetilde{\mathbb{O}}_\subreg}[-1]\text{ for $k\in\Z_{>0}$},\\\operatorname{Ind}_{\Cent(e)}^{G^\vee}\cO_{\bP_{\alpha_i}}(kx_i-(k+1)x_{i+1})\text{ for $k\in\Z$ and $1\leqslant i\leqslant n-1$,}\\ \operatorname{Ind}_{\Cent(e)}^{G^\vee}\big(\xi_{k+1/2}\otimes\cO_{\bP_\beta}(1,0)\big)\text{ for $k\in\Z_{\geqslant0}$}.
\end{align*}
\end{prop}

\subsection{When $G$ is of type F\textsubscript{4}}\label{subsec:type-F-canonical-basis}

Let $G$ (and hence also $G^\vee$) be of type F\textsubscript{4}. Recall that by \S\ref{sec:slodowy-slice}, the Springer fiber $\cB_e$ has $\Cent_e$-equivariant components $\bP_{\alpha_1},\dots,\bP_{\alpha_4}$ (using notation from Example~\ref{ex:f4-root-lattice}). Here, $\bP_{\alpha_1}$ and $\bP_{\alpha_2}$ are connected and isomorphic to $\bP^1$, while $\bP_{\alpha_3}$ and $\bP_{\alpha_4}$ have two components isomorphic to $\bP^1$.

\begin{prop}
Let $G$ be of type F\textsubscript{4}, and let $e\in\fg^\vee$ be subregular nilpotent. Then $\Cent_e$ acts on $\cB_e$ as an involution on $\bP_{\alpha_1}$ and $\bP_{\alpha_2}$, and swaps the two components of $\bP_{\alpha_3}$ and $\bP_{\alpha_4}$.
\end{prop}
\begin{proof}
By \cite{slodowy}, using notation from \S\ref{sec:slodowy-slice}, the resolution $\pi_S\colon \widetilde S\to S$ is isomorphic to the resolution of the type E\textsubscript{6} Kleinian singularity $X^4+Y^3+Z^2=0$. The group $\Cent_e$ acts as $(X,Y,Z)\mapsto(-X,Y,-Z)$ on $S$, which acts as a graph automorphism of the Dynkin diagram of E\textsubscript{6}, as in Example~\ref{ex:f4-root-lattice}. Thus, everything is clear except that the $\Cent_e$-action on $\bP_{\alpha_1}$ is nontrivial.

However, since $Z_{\mathrm{num}}$ (defined in \S\ref{sec:slodowy-slice}) is the highest weight of the E\textsubscript{6} lattice, and $Z_{\mathrm{num}}\cdot \bP_{\alpha_1}<0$, i.e. $\alpha_1^\vee$ is connected to the affine vertex in the E\textsubscript{6} Dynkin diagram, the line $\bP_{\alpha_1}$ appears in the blow-up $\widetilde X$ at a point of $X:x^4+y^3+z^2=0$. Let $\widetilde\bA^3\colonequals\{((X,Y,Z),(S:T:U))\in\bA^3\times\bP^2:XT=YS,XU=ZS,YU=ZT\}$, which has an affine open $\bA^3\simeq\widetilde\bA^3_{S\ne0}:(X,Y_0,Z_0)\mapsto(X,XY_0,XZ_0)$, the intersection $\widetilde X\cap\widetilde\bA^3_{S\ne0}$ is identitified with $\{(X,Y_0,Z_0)\in\bA^3:X^2+XY_0^3+Z_0^2=0\}$. Moreover, the involution $(X,Y,Z)\mapsto (-X,Y,-Z)$ induces the involution $(X,Y_0,Z_0)\mapsto (-X,-Y_0,Z_0)$. The pre-image of the exceptional divisor is then $\{(0,Y_0,Z_0)\in\bA^3:Z_0^2=0\}$, which is a (double) line, on which the involution acts non-trivially.
\end{proof}
We introduce the following notation, for convenience, when $G$ is of type F\textsubscript{4}:
\begin{itemize}
\item $x_i\colonequals\bP_{\alpha_i}\cap\bP_{\alpha_{i+1}}$ for $i=1,2,3$. Here $x_1$ is a point, while $x_2$ and $x_3$ are two points.
\item For $i=1,2$, let $y_i$ be other fixed point of the involution on $\bP_{\alpha_i}$.\footnote{$y_i$ is uniquely determined, since all involutions of $\bP^1$ have two fixed points.}
\end{itemize}

\begin{lemma}
    When $G$ is of type F\textsubscript{4}, the canonical basis of  $D^{\bounded}(\operatorname{Coh}^{\Cent_e}(\cB_e))$ consists of 
\begin{align*}
&\cO_{\cB_e},\ \sgn\otimes\cO_{\cB_e},\ \cO_{\bP_{\alpha_1}}\!(-x_1)[1],\ \cO_{\bP_{\alpha_1}}(-y_1)[1],\\&\cO_{\bP_{\alpha_2}}\!(-x_1)[1],\ \cO_{\bP_{\alpha_2}}(-y_2)[1],\ \cO_{\bP_{\alpha_3}}\!(-x_2)[1], \text{ and }\cO_{\bP_{\alpha_4}}\!(-x_3)[1].
\end{align*}
\end{lemma}

\begin{proof}[Proof of Lemma~\ref{lem:forg_funct_bij_irred} when $G$ is of type F\textsubscript{4}] The same proof as in type B works.
\end{proof}

We have the following explicit description of the canonical basis:
\begin{prop}\label{lem:canonical-basis-f}
When $G$ is of type F\textsubscript{4}, the canonical basis of $D^{\bounded}(\operatorname{Coh}^{G^\vee}\!(\widetilde{U}))$ consists of 
\begin{align*}
&\cO_{\widetilde U},\ \cO_{\widetilde{\mathbb{O}}_\subreg}[-1],\ \sgn\otimes\cO_{\widetilde{\mathbb{O}}_\subreg}[-1],\ \operatorname{Ind}_{\Cent(e)}^{G^\vee}\cO_{\bP_{\alpha_1}}\!(-x_1),\ \operatorname{Ind}_{\Cent(e)}^{G^\vee}\cO_{\bP_{\alpha_1}}(-y_1),\\&\operatorname{Ind}_{\Cent(e)}^{G^\vee}\cO_{\bP_{\alpha_2}}\!(-x_1),\ \operatorname{Ind}_{\Cent(e)}^{G^\vee}\cO_{\bP_{\alpha_2}}(-y_2),\ \operatorname{Ind}_{\Cent(e)}^{G^\vee}\cO_{\bP_{\alpha_3}}\!(-x_2), \text{ and }\operatorname{Ind}_{\Cent(e)}^{G^\vee}\cO_{\bP_{\alpha_4}}\!(-x_3).
\end{align*}
\end{prop}

\subsection{When $G$ is of type G\textsubscript{2}}\label{subsec:type-G-canonical-basis}

Let $G$ (and hence also $G^\vee$) be of type G\textsubscript{2}. Recall that by \S\ref{sec:slodowy-slice}, the Springer fiber $\cB_e$ has $\Cent_e$-equivariant components $\bP_{\alpha}$ and $\bP_{\beta}$ (using notation from Example~\ref{ex:f4-root-lattice}). Here, $\bP_{\alpha_1}$ and $\bP_{\alpha_2}$ are connected and isomorphic to $\bP^1$, while $\bP_{\alpha_3}$ and $\bP_{\alpha_4}$ have two components isomorphic to $\bP^1$.

\begin{prop}
Let $G$ be of type G\textsubscript{2}, and let $e\in\fg^\vee$ be subregular nilpotent. Then $\Cent_e$ acts on $\cB_e$ as an involution on $\bP_{\alpha}$ and permutes the three components of $\bP_{\beta}$.
\end{prop}
\begin{proof}
By \cite{slodowy}, using notation from \S\ref{sec:slodowy-slice}, the resolution $\pi_S\colon \widetilde S\to S$ is isomorphic to the resolution of the type D\textsubscript{4} Kleinian singularity $x^3+y^3+z^2=0$. The group $\Cent_e$ acts as the full automorphism group of $\pi_S$, which acts as a graph automorphism of the Dynkin diagram of E\textsubscript{6}, as in Example~\ref{ex:f4-root-lattice}. Thus, everything is clear except that the $\Cent_e$-action on $\bP_{\alpha_1}$ is nontrivial.

However, since $Z_{\mathrm{num}}$ (defined in \S\ref{sec:slodowy-slice}) is the highest weight of the D\textsubscript{4} lattice, and $Z_{\mathrm{num}}\cdot \bP_{\alpha_1}<0$, i.e. $\alpha^\vee$ is connected to the affine vertex in the D\textsubscript{4} Dynkin diagram, the line $\bP_{\alpha}$ appears in the blow-up $\widetilde S$ at $(0,0,0)$ of $S:X^3+Y^3+Z^2=0$. Let $\widetilde\bA^3\colonequals\{((X,Y,Z),(S:T:U))\in\bA^3\times\bP^2:XT=YS,XU=ZS,YU=ZT\}$, which has an affine open $\bA^3\simeq\widetilde\bA^3_{S\ne0}:(X,Y_0,Z_0)\mapsto(X,XY_0,XZ_0)$, the intersection $\widetilde X\cap\widetilde\bA^3_{S\ne0}$ is identified with $\{(X,Y_0,Z_0)\in\bA^3:X^2+XY_0^3+Z_0^2=0\}$. Moreover, the involution $(X,Y,Z)\mapsto (-X,Y,-Z)$ induces the involution $(X,Y_0,Z_0)\mapsto (-X,-Y_0,Z_0)$. The pre-image of the exceptional divisor is then $\{(0,X_0,Z_0)\in\bA^3:Z_0^2=0\}$, which is a (double) line, on which the involution acts non-trivially.
\end{proof}
For convenience, let $x\colonequals\bP_\alpha\cap\bP_\beta$, which consists of three points; one in each component of $\bP_\beta$.

\begin{lemma}
    Let $G$ be of type G\textsubscript{2}. The canonical basis of $D^{\bounded}(\operatorname{Coh}^{\Cent_e}(\cB_e))$ consists of 
\begin{align*}
&\cO_{\cB_e},\ \sgn\otimes\cO_{\cB_e},\ \std\otimes\cO_{\cB_e},\ \cO_{\bP_\beta}\!(-x)[1],\ \cO_{\bP_\beta}(-y)[1],\\& \cO_{\bP_\alpha}(1,0)[1],\ \cO_{\bP_\alpha}(0,-1)[1],\text{ and } \std\otimes\cO_{\bP_\alpha}(1,0)[1].
\end{align*}
\end{lemma}

\begin{proof}[Proof of Lemma~\ref{lem:forg_funct_bij_irred} when $G$ is of type G\textsubscript{2}]
Clearly $\cO_{\cB_e}$, $\sgn\otimes\cO_{\cB_e}$, and $\std\otimes\cO_{\cB_e}$ have $\Cent(e)$-equivariance. Moreover $x=\bP_\alpha\cap\bP_\beta$ is $\Cent(e)$-equivariant, so $\cO_{\bP_\beta}\!(-x)[1]$ and $\cO_{\bP_\beta}(-y)[1]\simeq \sgn\otimes\cO_{\bP_\beta}\!(-x)[1]$ have $\Cent(e)$-equivariance.

Moreover, the unipotent radical of $\Cent(e)$ acts on $\bP_\alpha\simeq\bP^1$ with three fixed points, i.e., is trivial. Thus $\cO_{\bP_\alpha}(1,0)[1]$, $\cO_{\bP_\alpha}(0,-1)[1]$, and $\std\otimes\cO_{\bP_\alpha}(1,0)[1]$ are $\Cent_e$-equivariant sheaves, and hence automatically also $\Cent(e)$-equivariant sheaves.
\end{proof}

Thus, we have:
\begin{prop}\label{lem:canonical-basis-g}
Let $G$ be of type G\textsubscript{2}. The irreducible exotic coherent sheaves of $D^{\bounded}(\operatorname{Coh}^{G^\vee}\!(\widetilde{U}))$ consists of 
\begin{align*}
\cO_{\widetilde U},\ \cO_{\widetilde{\mathbb{O}}_\subreg}[-1],\ \sgn\otimes\cO_{\widetilde{\mathbb{O}}_\subreg}[-1],\ \std\otimes\cO_{\widetilde{\mathbb{O}}_\subreg}[-1],\ \operatorname{Ind}_{\Cent(e)}^{G^\vee}\cO_{\bP_\beta}\!(-x),\ \operatorname{Ind}_{\Cent(e)}^{G^\vee}\cO_{\bP_\beta}(-y),\\ \operatorname{Ind}_{\Cent(e)}^{G^\vee}\cO_{\bP_\alpha}(1,0),\ \operatorname{Ind}_{\Cent(e)}^{G^\vee}\cO_{\bP_\alpha}(0,-1),\ \operatorname{Ind}_{\Cent(e)}^{G^\vee}\big(\std\otimes\cO_{\bP_\alpha}(1,0)\big).
\end{align*}
\end{prop}

\section{Proof of Proposition~\ref{main-thm}}\label{sec:proof-of-main-thm}

First, we make some preliminary observations:
\begin{lemma}\label{s0-canonical}
    Let $G$ be arbitrary. Then $C_1=[\cO_{\widetilde U}]$ and $C_{s_0}=-[\cO_{\cB_e}]$.    Moreover, the basis elements satisfy a twist by sign of \eqref{e-defn2}:
    \begin{align*}
        T_tC_1&=\begin{cases}
        -C_1-C_{s_0}&t=s_0\\
        -C_1&t\ne s_0.
        \end{cases}
    \end{align*}
\end{lemma}
\begin{proof}
    The fact that $C_1=[\cO_{\widetilde U}]$ is general (since $\cO_{\widetilde\cN}$ is in the heart of the exotic $t$-structure of $\cD^\bounded(\Coh(\widetilde\cN))$). Moreover, by Remark~\ref{Cs0-calc}, $C_{s_0}=-[\cO_{\cB_e}]$. Any $t\in \widehat S\backslash\{s_0\}=S$ lies in $W$, so acts on $[\cO_{\widetilde U}]$ by $\sgn(t)=-1$.
\end{proof}

We proceed by a case-work on the type of $G$.

\subsection{When $G$ is simply-laced} In this case, \cite{BKK} provides an explicit description of the $\widehat W$-module $K^{G^\vee}\!(\widetilde U)$:
\begin{lemma}[{\cite[Prop~5.16]{BKK}}]
    Let $G$ be of type D or E, of rank $n$. There is an isomorphism of $\hatW$-modules
    \(
    K^{G^\vee}\!(\widetilde U)\simeq \widehat{\mathfrak h}_\Z\otimes\Z_{\sgn}
    \)
    given by
    \[
    C_1=[\cO_{\widetilde U}]\mapsto d\otimes1,C_{w_0}=[\cO_{\cB_e}]\mapsto\alpha_0^\vee\otimes1,C_{w_i}=-[\cO_{\bP_{\alpha_i}}(-1)]\mapsto\alpha_i^\vee,
    \]
    where $\widehat{\mathfrak h}\colonequals\mathfrak{h}_\Z\oplus\Z K\oplus\Z d$, which is a $\hatW$-module by the action $s_i(x)\colonequals x-\langle\alpha_i,x\rangle\alpha_i^\vee$, for $x\in\mathfrak{h}$ and $i=0,1,\dots,n$.
\end{lemma}

\begin{lemma}[{\cite[Prop~6.17]{BKK}}]
    Let $G$ be of type A, of rank $n$. There is an isomorphism of $\hatW$-modules 
    \(
    K^{G^\vee}\!(\widetilde U)\simeq\widehat{\mathfrak h}_{\infty,\Z}\otimes\Z_{\sgn}
    \)
    given by
    \(
    C_{\nu_i}\mapsto-\alpha_i^\vee\otimes1\) and \({C}_0\mapsto d\otimes1,
    \)
    where $\widehat{\mathfrak h}_{\infty,\Z}$ is defined in \cite[\S6.4.2]{BKK}.
\end{lemma}

Now, Proposition~\ref{main-thm} reduces to the following statements about representations of $\hatW$:
\begin{prop}
    Let $G$ be of type D or E. Then there is an isomorphism
    \(
    \widetilde E_\subreg^0\simeq\widehat{\mathfrak h}_\Z
    \)
    sending $e_{w_i}$ to $-\alpha_i^\vee$ for $0\leqslant i\leqslant n$ and sending $e_1$ to $d$ where $w_i\in c_\subreg^0$ is defined in \cite[Cor~5.2]{BKK}.
\end{prop}
\begin{proof}
    It suffices to check \eqref{e-defn} and \eqref{e-defn2}. Indeed, we have:
    \begin{align*}
    s_j(d)&=\begin{cases}
        d-\alpha_0^\vee&j=0\\
        d&j\ne0.
    \end{cases}\\
    s_j(-\alpha_i^\vee)=\langle\alpha_i^\vee,\alpha_j\rangle\alpha_j^\vee-\alpha_i^\vee&=\begin{cases}
        \alpha_i^\vee&i=j\\
        -\alpha_i^\vee&\text{$i$ and $j$ are not adjacent in the Dynkin graph}\\
        -\alpha_i^\vee-\alpha_j^\vee&\text{$i$ and $j$ are adjacent in the Dynkin graph}.
    \end{cases}
    \end{align*}
\end{proof}

\begin{prop}
    Let $G$ be of type A. Then there is an isomorphism
    \(
    \widetilde E_\subreg^0\simeq\widehat{\mathfrak h}_{\infty,\Z}
    \)
    sending $e_{s_i\cdots s_0}$ to $\alpha_i^\vee$ for $0\leqslant i\leqslant n$, and sending $e_1$ to $d$.
\end{prop}
\begin{proof}
    Again, it suffices to check \eqref{e-defn} and \eqref{e-defn2}:
    \begin{align*}
        s_j(d)&=\begin{cases}
        d-\alpha_0^\vee&j=0\\
        d&j\ne0,
        \end{cases}\\
        s_j(\alpha_i^\vee)=s_j(\epsilon_i-\epsilon_{i+1})&=\begin{cases}
            \epsilon_{i+1}-\epsilon_i&j=i\\
            \epsilon_i-\epsilon_{i+2}=\alpha_i^\vee+\alpha_{i+1}^\vee&j=i+1\\
            \epsilon_{i-1}-\epsilon_{i+1}=\alpha_{i}^\vee+\alpha_{i-1}^\vee&j=i-1\\
            \epsilon_i-\epsilon_{i+1}&\text{otherwise}.
        \end{cases}
    \end{align*} 
\end{proof}

\subsection{When $G$ is of type B}\label{sec:type-b-canonical} Let $G$ be of type $\text{B}_n$ where $n\geqslant3$. The canonical basis of $K^{G^\vee}\!(\widetilde U)$ was described in Lemma~\ref{lem:canonical-basis-b}, so our first task is to compute which element of $c^0_\subreg$, described in Lemma~\ref{typec-subregular-cell} corresponds to each basis element of $K^{G^\vee}\!(\widetilde U)$ (see Proposition~\ref{prop:canonical-basis-b}). First, we compute some equivariant line bundles:

\begin{lemma}\label{equivariant-structure-b}
Let $G$ be of type $\text{B}_n$ where $n\geqslant3$, and let $e\in\fg^\vee$ be subregular nilpotent. For $\gamma=a_1\epsilon_1+\cdots+a_n\epsilon_n\in Q^\vee$ there are isomorphisms of $\Cent_e\simeq \Z/2$-equivariant line bundles:
\begin{align*}\cO_{\bP_{\alpha_i}}(\gamma)&\simeq\cO_{\bP_{\alpha_i}^1}(a_ix_i-a_{i+1}x_{i+1})\text{ for $1\leqslant i\leqslant n-1$}\\
\cO_{\bP_\beta}(\gamma)&\simeq\cO_{\bP_{\beta}}(a_ny).
\end{align*}
\end{lemma}
\begin{proof}
For $i=1,\dots,n-1$, the underlying line bundle of $\cO_{\bP_{\alpha_i}}(\gamma)$ has fiber over $V_{i-1}\subsetneq V\subsetneq V_{i+1}$ isomorphic to $(V/V_{i-1})^{\otimes -a_i}\otimes (V_{i+1}/V)^{\otimes -a_{i+1}}$, i.e., it is $\cO_{\bP^1_{\alpha_i}}\!(a_i-a_{i+1})$.\footnote{Alternatively, use Lemma~\ref{restrict-line-bundle}.} Moreover, the underlying line bundle of $\cO_{\bP_\beta}(\gamma)$ has fiber over $V_{n-1}\subsetneq V\subsetneq V_{n+1}$ isomorphic to $(V/V_{n-1})^{\otimes -a_n}$, i.e., it is $\cO_{\bP_{\beta}}\!(a_n)$.

Moreover, the fiber of $\cO_\cB(\gamma)$ over $x_j$ (as in Lemma~\ref{lem:fixed-pts}) is
\[
\bigotimes_{i=1}^{j-1}V(2n-3)_{2n-2i-1}^{\otimes -a_i}\otimes V(1)_1^{\otimes -a_{j}}\otimes\bigotimes_{i=j+1}^{n}V(2n-3)_{2n-2i+1}^{\otimes -a_i},
\]
so $\sigma\ne1\in\Cent_G(e)$ acts as $(-1)^{a_{j}}$ on the fiber of $\cO_\cB(\gamma)$ over $x_j$.
\end{proof}

As an immediate corollary, we obtain an explicit description of the $ Q^\vee$-action on $K^{\Cent_e}(\cB_e)$ in type B:
\begin{cor}
$\gamma=\sum_{i=1}^na_i\epsilon_i\in Q^\vee$ acts on $K^{G^\vee}\!(\widetilde\bO_\subreg)\simeq K^{\Z/2}(\cB_e)$ as:
\begin{align*}
t_\gamma\cdot[\C_{x_1}]&=\sgn^{a_1}\cdot[\C_{x_1}]\\
t_\gamma\cdot[\cO_{\bP^1_{\alpha_i}}]&=[\cO_{\bP^1_{\alpha_i}}\!(-x_i)]+(1+\cdots+\sgn^{a_i})[\C_{x_i}]-(1+\cdots+\sgn^{a_{i+1}-1})[\C_{x_{i+1}}]\\
t_\gamma\cdot[\cO_{\bP_{\beta}}]&=[\cO_{\bP_\beta}(-y)]-(a_n-1)[\C_{y}].
\end{align*}
\end{cor}
\begin{proof}
Follows from induction on the standard exact sequence
\[
0\to\cO_{\bP^1_{\alpha_j}}\!\big((a_j-1)x_j-a_{j+1}x_{j+1}\big)\to\cO_{\bP^1_{\alpha_j}}\!(a_jx_j-a_{j+1}x_{j+1})\to \sgn^{a_j}\otimes\C_{x_j}\to 0.\qedhere
\]
\end{proof}

These canonical basis elements, described in Lemma~\ref{lem:canonical-basis-b}, must correspond to elements of the subregular two-sided cell of $\widehat W$, as enumerated in Lemma~\ref{typec-subregular-cell}. To match the elements, recall that for $\gamma\in Q^\vee$ there is an equality \cite[\S6.7]{BKK}
\[
t_\gamma\cdot 1=\sum_{w_\nu\leqslant w_\gamma}\mathbf m_{w_\nu}^{w_\gamma}C_\nu.
\]
We calculate the expression for all of the $\gamma\in Q^\vee$ such that $w_\gamma\in c_\subreg\cap\widehat W^f$.
Using Lemma~\ref{alpha-action} and Lemma~\ref{equivariant-structure-b}, in case~\eqref{cell-case3}, for $2\leqslant i<n$,
\begin{equation}\label{eq:1+(i+1)}
[\cO_{\widetilde U}(\epsilon_1+\epsilon_{i+1})]=[\cO_{\widetilde U}(\epsilon_1+\epsilon_2)]+\sum_{j=2}^i[\cO_{\bP_{\alpha_j}}\!(-x_{j+1})]=[\cO_{\widetilde U}]-[\cO_{\cB_e}]+\sum_{j=2}^i[\cO_{\bP_{\alpha_j}}\!(-x_{j+1})],
\end{equation}so\[
C_{s_is_{i-1}\cdots s_0}=C_{w_{\epsilon_1+\epsilon_{i+1}}}=[\cO_{\bP_{\alpha_i}}\!(-x_{i+1})].\] In case~\eqref{cell-case2}, by \eqref{eq:1+(i+1)} for $i=2$,
\begin{equation}\label{eq:2+3}
    [\cO_{\widetilde U}(\epsilon_2+\epsilon_3)]=[\cO_{\widetilde U}(\epsilon_1+\epsilon_3)]+[\cO_{\bP_{\alpha_1}}\!(-x_2)]=[\cO_{\widetilde U}]-[\cO_{\cB_e}]+[\cO_{\bP_{\alpha_1}}\!(-x_2)]+[\cO_{\bP_{\alpha_2}}\!(-x_3)],
\end{equation}
so $C_{w_{\epsilon_2+\epsilon_3}}=[\cO_{\bP_{\alpha_1}}\!(-x_2)]$.

Moreover, \eqref{cell-case4} corresponds to (for $2\leqslant i\leqslant n$):
\begin{align}
[\cO_{\widetilde U}(\epsilon_1-\epsilon_i)]&=[\cO_{\widetilde U}(\epsilon_1+\epsilon_n)]+[\cO_{\bP_\beta}(-y)]+\sum_{j=i}^{n-1}[\cO_{\bP_{\alpha_j}}\!(-x_j)]\nonumber\\
&=[\cO_{\widetilde U}]-[\cO_{\cB_e}]+\sum_{j=2}^{n-1}[\cO_{\bP_{\alpha_j}}\!(-x_{j+1})]+[\cO_{\bP_\beta}(-y)]+\sum_{j=i}^{n-1}[\cO_{\bP_{\alpha_j}}\!(-x_j)],\label{1-i-formula}
\end{align}
so $C_{w_{\epsilon_1-\epsilon_n}}=[\cO_{\bP_\beta}(-y)]$ and $C_{\epsilon_1-\epsilon_i}=[\cO_{\bP_{\alpha_i}}\!(-x_i)]$.

Note that the above calculation actually gives a change-of-basis formula. For convenience, let
\begin{equation}\label{dj-defn}
d_j=\begin{cases}
1&\text{if }j=1\\
2&\text{if }2\leqslant j\leqslant n-1.
\end{cases}
\end{equation}
\begin{lemma}\label{lem:change-of-basis}
For any $2\leqslant i\leqslant n-1$,
\begin{align}
[\C_{x_1}]&=[\cO_{\cB_e}]-[\cO_{\bP_{\alpha_1}}\!(-x_1)]-\sum_{j=2}^{n-1}(1+\sgn)[\cO_{\bP_{\alpha_j}}\!(-x_j)]-[\cO_{\bP_\beta}(-y)]\label{eq:change-of-basis}\\
[\C_{x_i}]&=[\cO_{\cB_e}]-\sum_{j=1}^{i-1}d_j[\cO_{\bP_{\alpha_j}}\!(-x_{j+1})]-\sum_{j=i}^{n-1}(1+\sgn)[\cO_{\bP_{\alpha_j}}\!(-x_j)]-[\cO_{\bP_\beta}(-y)]\\
[\C_y]&=(1+\sgn)[\cO_{\cB_e}]-(1+\sgn)\sum_{j=1}^{n-1}d_j[\cO_{\bP_{\alpha_j}}\!(-x_j)]-2[\cO_{\bP_\beta}(-y)].
\end{align}
\end{lemma}
\begin{proof}
When $i=2$, \eqref{1-i-formula}, together with Lemma~\ref{alpha-action}, gives:
\[
[\cO_{\widetilde U}]-[\cO_{\bP_{\alpha_1}}]=[\cO_{\widetilde U}]-[\cO_{\cB_e}]+\sum_{j=2}^{n-1}\big([\cO_{\bP_{\alpha_j}}\!(-x_{j+1})]+[\cO_{\bP_{\alpha_j}}\!(-x_j)]\big)+[\cO_{\bP_\beta}(-y)],
\]
which, since $[\cO_{\bP_{\alpha_1}}]=[\C_{x_1}]+[\cO_{\bP_{\alpha_1}}\!(-x_1)]$, gives \eqref{eq:change-of-basis}. Now the rest follows from observing that from the exact sequences
\begin{align*}
0\to\cO_{\bP_{\alpha_i}}\!(-x_i)\to&\ \cO_{\bP_{\alpha_i}}\to\C_{x_i}\to 0\\
0\to\cO_{\bP_{\alpha_i}}\!(-x_{i+1})\to&\ \cO_{\bP_{\alpha_i}}\to\C_{x_{i+1}}\to 0\nonumber
\end{align*}
which gives $[\C_{x_{i+1}}]=[\C_{x_i}]+[\cO_{\bP_{\alpha_i}}\!(-x_i)]-[\cO_{\bP_{\alpha_i}}\!(-x_{i+1})]$.
\end{proof}
\begin{remark}
Note that the multiplicities in \eqref{dj-defn} exactly match $\alpha_0^\vee$, in Example~\ref{type-b-highest}.
\end{remark}
By \eqref{eq:change-of-basis}, case~\eqref{cell-case4} of Proposition~\ref{typec-subregular-cell} when $i=1$ gives:
\begin{align}
[\cO_{\widetilde U}(\epsilon_2-\epsilon_1)]&=[\cO_{\widetilde U}]+[\cO_{\bP_{\alpha_1}}\!(-x_1-x_2)]\nonumber\\
&=[\cO_{\widetilde U}]+[\cO_{\bP_{\alpha_1}}\!(-x_2)]-[\C_{x_1}]\label{eq:2-1-formula}\\
&=[\cO_{\widetilde U}]-[\cO_{\cB_e}]+\sum_{j=1}^{n-1}\big([\cO_{\bP_{\alpha_j}}\!(-x_j)]+[\cO_{\bP_{\alpha_j}}\!(-x_{j+1})]\big)+[\cO_{\bP_\beta}(-y)],\nonumber
\end{align}
which gives $C_{w_{\epsilon_2-\epsilon_1}}=[\cO_{\bP_{\alpha_1}}\!(-x_1)]$.

Finally, case~\eqref{cell-case5} is
\begin{align}
    [\cO_{\widetilde U}(2\epsilon_2)]=&\ [\cO_{\widetilde U}(\epsilon_1+\epsilon_2)]+[\cO_{\bP_{\alpha_1}}(-2x_2)]\nonumber\\
    =&\ [\cO_{\widetilde U}]-[\cO_{\cB_e}]+[\cO_{\bP_{\alpha_1}}\!(-x_2)]-[\sgn\otimes\C_{x_2}]\nonumber\\
    \label{eq:2+2-formula}=&\ [\cO_{\widetilde U}]-[\cO_{\cB_e}]-[\sgn\otimes\cO_{\cB_e}]+[\cO_{\bP_{\alpha_1}}\!(-x_1)]\\&+\sum_{j=2}^{n-1}\big([\cO_{\bP_{\alpha_j}}\!(-x_j)]+[\cO_{\bP_{\alpha_j}}\!(-x_{j+1})]\big)+[\cO_{\bP_\beta}(-y)],\nonumber
\end{align}
so $C_{w_{2\epsilon_2}}=-[\sgn\otimes\cO_{\cB_e}]$.

Thus, we conclude:
\begin{prop}\label{prop:canonical-basis-b}
Let $G$ be of type $\text{B}_n$ where $n\geqslant3$. The canonical basis of $K^{G^\vee}\!(\widetilde U)$ are:
\begin{align*}
C_1&=[\cO_{\widetilde U}]\\
C_{s_0}&=-[\cO_{\cB_e}]\\
C_{s_1s_2s_0}&=[\cO_{\bP_{\alpha_1}}\!(-x_2)]\\
C_{s_i\cdots s_0}&=\begin{cases}
[\cO_{\bP_{\alpha_i}}\!(-x_{i+1})]&\text{if }2\leqslant i\leqslant n-1\\
[\cO_{\bP_\beta}(-y)]&\text{if }i=n
\end{cases}\\
C_{s_i\cdots s_{n-1}s_ns_{n-1}\cdots s_2s_0}&=[\cO_{\bP_{\alpha_i}}\!(-x_i)]\\
C_{s_0s_2\cdots s_{n-1}s_ns_{n-1}\cdots s_2s_0}&=-[\sgn\otimes\cO_{\cB_e}].
\end{align*}
\end{prop}

Finally, we may prove:
\begin{proof}[Proof of Proposition~\ref{main-thm} when $G$ is of type B] It suffices to check the basis in Proposition~\ref{prop:canonical-basis-b} satisfies relations \eqref{e-defn} (note that \eqref{e-defn2} was checked in Lemma~\ref{s0-canonical}).

First, note that for each simple coroot $\alpha^\vee$ of $\fg$ and $s_\alpha\cdots s_0\in c^0_\subreg$ a canonical basis element $C_{s_\alpha\cdots s_0}$ is supported on $\bP_\alpha$. Thus for $\alpha'\in\Delta$ with $\langle\alpha,\alpha'\rangle=0$ (equivalently, $\cN_\alpha$ and $\cN_{\alpha'}$ are disjoint),
\[
s_{\alpha'}C_{s_\alpha\cdots s_0}=C_{s_\alpha\cdots s_0},
\]
as follows, e.g., from the description of the Hecke algebra action in \cite{lusztig-affine-hecke-algebras}.

Now \eqref{eq:1+(i+1)} gives, for $i\geqslant 2$,
\begin{equation}\label{eq:s_icdotss_2s_0C_1}
(-1)^is_i\cdots s_2s_0C_1=C_1+C_{s_0}+\sum_{j=2}^iC_{s_j\cdots s_2s_0}.
\end{equation}
When $i=2$ the above is $s_2s_0C_1=C_1+C_{s_0}+C_{s_2s_0}$, which gives
\begin{align}s_2C_{s_0}&=-s_2s_0C_1-s_2C_1\nonumber\\
&=-C_1-C_{s_0}-C_{s_2s_0}+C_1\label{eq:s_2C_{s_0}}\\
&=-C_{s_0}-C_{s_2s_0}\nonumber.
\end{align}For $i>2$ we have
\begin{equation}\label{eq:s_iC_{s_0}}
s_iC_{s_0}=s_i(-C_1-s_1C_1)=C_1+s_1C_1=-C_{s_0},
\end{equation}
and
\begin{equation}\label{eq:s_0C_{s_0}}s_0C_{s_0}=s_0(-C_1-s_0C_1)=-s_0C_1-C_1=C_{s_0}.\end{equation}
Equations \eqref{eq:s_2C_{s_0}}, \eqref{eq:s_iC_{s_0}}, and \eqref{eq:s_0C_{s_0}} which together is exactly \eqref{e-defn} twisted by sign for $C_{s_0}$.

When $i\leqslant n-2$ by applying $-s_{i+1}$ on both sides of \eqref{eq:s_icdotss_2s_0C_1},
\[
C_1+\sum_{j=2}^{i+1}C_{s_j\cdots s_2s_0}=C_1+\sum_{j=2}^{i-1}C_{s_j\cdots s_0}-s_{i+1}C_{s_i\cdots s_2s_0},
\]
i.e.,
\begin{equation}\label{eq:s_{i+1}C_{s_icdotss_2s_0}}
s_{i+1}C_{s_i\cdots s_2s_0}=-C_{s_i\cdots s_0}-C_{s_{i+1}\cdots s_0}.
\end{equation}
Similarly, applying $-s_{i}$ to both sides of \eqref{eq:s_icdotss_2s_0C_1} gives
\begin{align*}
C_1+C_{s_0}+\sum_{j=2}^{i-1}C_{s_j\cdots s_0}&=C_1+C_{s_0}+\sum_{j=2}^{i-2}C_{s_j\cdots s_0}-s_{i}C_{s_{i-1}\cdots s_0}-s_iC_{s_i\cdots s_0}\\
&=C_1+C_{s_0}+\sum_{j=2}^{i}C_{s_j\cdots s_0}-s_iC_{s_i\cdots s_0}
\end{align*}
i.e.,
\begin{equation}\label{eq:s_iC_{s_icdotss_0}}
s_iC_{s_i\cdots s_0}=C_{s_i\cdots s_0}.
\end{equation}
Finally, applying $-s_{i-1}$ to both sides of \eqref{eq:s_icdotss_2s_0C_1} gives (using the braid relation $s_{i-1}s_is_{i-1}=s_is_{i-1}s_i$)
\begin{align*}
C_1+C_{s_0}+\sum_{j=2}^iC_{s_j\cdots s_2s_0}&=C_1+C_{s_0}+\sum_{j=2}^{i-3}C_{s_j\cdots s_2s_0}-s_{i-1}C_{s_{i-2}\cdots s_0}-s_{i-1}C_{s_{i-1}\cdots s_0}-s_{i-1}C_{s_i\cdots s_0}\\
&=C_1+C_{s_0}+\sum_{j=2}^{i-1}C_{s_j\cdots s_2s_0}-C_{s_{i-1}\cdots s_0}-s_{i-1}C_{s_i\cdots s_0},
\end{align*}
i.e.,
\begin{equation}\label{eq:s_{i-1}C_{s_icdotss_2s_0}}
s_{i-1}C_{s_i\cdots s_2s_0}=-C_{s_i\cdots s_2s_0}-C_{s_{i-1}\cdots s_0}.
\end{equation}
Equations \eqref{eq:s_{i+1}C_{s_icdotss_2s_0}}, \eqref{eq:s_iC_{s_icdotss_0}}, and \eqref{eq:s_{i-1}C_{s_icdotss_2s_0}} together is exactly \eqref{e-defn} twisted by sign for $C_{s_i\cdots s_2s_0}$.

Similarly, \eqref{1-i-formula}, \eqref{eq:2-1-formula}, and \eqref{eq:2+2-formula} gives, for $0\leqslant i\leqslant n$,
\[
(-1)^{i-1}s_i\cdots s_{n-1}s_ns_{n-1}\cdots s_2s_0C_1=C_1+C_{s_0}+\sum_{j=2}^{n-1}C_{s_j\cdots s_2s_0}+\sum_{j=i}^{n}C_{s_i\cdots s_{n-1}s_ns_{n-1}\cdots s_2s_0},\]
so by the same argument as above, we can show \eqref{e-defn} for $C_{s_i\cdots s_ns_{n-1}\cdots s_0}$.\end{proof}

\subsection{When $G$ is of type C}

\begin{lemma}\label{equivariant-structure-c}
Let $\fg$ be of type $\text{C}_n$ where $n\geqslant2$, and let $\gamma=\sum_{i=1}^{n}a_i\epsilon_i\in Q^\vee$. Then (using notation from Definition~\ref{defn:typec,xj} and \eqref{equiv-line-bundle}) in $\Coh^{\Cent_e}(\cB_e)$,
\begin{align*}
\cO_{\bP_{\alpha_i}}(\gamma)&\simeq\cO_{\bP_{\alpha_i}}\!(a_ix_i-a_{i+1}x_{i+1})\\
\cO_{\bP_\beta}(\gamma)&\simeq\cO_{\bP_\beta}\!(a_1+\cdots -a_n,a_1+\cdots+a_{n-1}+a_n)=\sgn^{a_1+\cdots a_{n-1}}\otimes\cO_{\bP_\beta}(-a_n,a_n),
\end{align*}
where as in Definition~\ref{defn:typec,xj}, $x_i\colonequals\{x_i^+,x_i^-\}$ is a divisor on $\bP_{\alpha_i}$. 
\end{lemma}
\begin{proof}
By Lemma~\ref{restrict-line-bundle}, the underlying line bundles of $\cO_{\bP_{\alpha_i}}(\gamma)$ is $\cO_{\bP^1\sqcup\bP^1}(a_i-a_{i+1})$, and the underlying bundle of $\cO_{\bP_\beta}$ is $\cO_{\bP^1}(2a_n)$.

Moreover, using the notation in Lemma~\ref{lem:fixed-point-c}, $z\in\C^\times$ acts on the stalk of $\cO_{\cB_e}(\gamma)$ at $x_j^\pm$ as $z^{\mp a_j}$. Thus, by \cite[Lemma~6.2]{BKK}, we know
\[
\cO_{\bP_{\alpha_i}}(\gamma)\simeq\cO_{\bP_{\alpha_i}}\!(a_ix_i-a_{i+1}x_{i+1}).
\]
Similarly, we conclude from Corollary~\ref{cor:o2-equiv} that $\cO_{\bP_\beta}(\gamma)\simeq\cO_{\bP_\beta}\!(a_n,-a_n)$.
\end{proof}

\begin{prop}\label{prop:canonical-basis-c}
    Let $\fg$ be of type $\text{C}_n$ where $n\geqslant2$. Then for $k>0$,
\begin{align*}
C_{w_{\epsilon_1+\epsilon_2}}&=-[\sgn\otimes\cO_{\cB_e}]\\
    C_{w_{k\epsilon_i}}&=\begin{cases}
        -[\cO_{\cB_e}]&k=1,i=1\\
        -[\xi_{k-1}\otimes\cO_{\cB_e}]&k>1,i=1\\
        [\cO_{\bP_{\alpha_{i-1}}}((k-1)x_{i-1}-kx_i)]&1<i\leqslant n,
    \end{cases}\\
    C_{w_{-k\epsilon_i}}&=\begin{cases}
        [\cO_{\bP_{\alpha_i}}(-kx_i+(k-1)x_{i+1})]&1\leqslant i<n\\
        [\xi_{k-1/2}\otimes\cO_{\bP_\beta}(1,0)]&i=n.
    \end{cases}
\end{align*}
\end{prop}
\begin{proof}
 We proceed by induction on the Weyl group ordering, as described in Lemma~\ref{typeb-subregular-cell}. By Corollary~\ref{cor:subreg-desc}, we have $[\cO_{\widetilde U}(\epsilon_1)]=[\cO_{\widetilde U}]-[\cO_{\cB_e}]$, so $C_{w_{\epsilon_1}}=-[\cO_{\cB_e}]$.  Note that
 \[
 [\cO_{\widetilde U}(\epsilon_2)]=[\cO_{\widetilde U}(\epsilon_1)]+[\cO_{\bP_{\alpha_1}}(\epsilon_2)]=[\cO_{\widetilde U}(\epsilon_1)]+[\cO_{\bP_{\alpha_1}}\!(-x_2)],
 \]
 so $C_{w_{\epsilon_2}}=[\cO_{\bP_{\alpha_1}}\!(-x_2)]$. Moreover, by inductively applying Lemma~\ref{alpha-action} and Lemma~\ref{equivariant-structure-c}, we obtain:
 \begin{align*}
     [\cO_{\widetilde U}]&=[\cO_{\widetilde U}(\epsilon_n)]+[\cO_{\bP_\beta}]\\
                &=[\cO_{\widetilde U}(\epsilon_{n-1})]+[\cO_{\bP_\beta}]+[\cO_{\bP_{\alpha_{n-1}}}\!(-x_n)]\\
                &=\cdots\\
                &=[\cO_{\widetilde U}(\epsilon_1)]+[\cO_{\bP_\beta}]+\sum_{i=1}^{n-1}[\cO_{\bP_{\alpha_i}}\!(-x_{i+1})]\\
                &=[\cO_{\widetilde U}]-[\cO_{\cB_e}]+[\cO_{\bP_\beta}]+\sum_{i=1}^{n-1}[\cO_{\bP_{\alpha_i}}\!(-x_{i+1})],
 \end{align*}
 so
 \begin{equation}\label{beta-exp}
 [\cO_{\cB_e}]=[\cO_{\bP_\beta}]+\sum_{i=1}^{n-1}[\cO_{\bP_{\alpha_i}}\!(-x_{i+1})].
 \end{equation}
Thus,
\begin{align*}
    [\cO_{\widetilde U}(\epsilon_1+\epsilon_2)]&=[\cO_{\widetilde U}(\epsilon_1+\epsilon_3)]-[\cO_{\bP_{\alpha_2}}\!(-x_3)]\\
    &=\cdots\\
    &=[\cO_{\widetilde U}(\epsilon_1+\epsilon_n)]-\sum_{i=2}^{n-1}[\cO_{\bP_{\alpha_i}}\!(-x_{i+1})]\\
    &=[\cO_{\widetilde U}(\epsilon_1)]-[\sgn\otimes\cO_{\bP_\beta}]-\sum_{i=2}^{n-1}[\cO_{\bP_{\alpha_i}}\!(-x_{i+1})]\\
    &=[\cO_{\widetilde U}]-[\cO_{\cB_e}]-[\sgn\otimes\cO_{\cB_e}]+[\cO_{\bP_{\alpha_1}}\!(-x_2)],
\end{align*}
where for the last equality we used \eqref{beta-exp} tensored with $[\sgn]\in K^{\G_m\rtimes\Z/2}(*)$. Thus $C_{w_{\epsilon_1+\epsilon_2}}=[\sgn\otimes\cO_{\cB_e}]$.

For $k>0$ and $i>1$, again by repeated application of Lemma~\ref{alpha-action} and Lemma~\ref{equivariant-structure-c},
\begin{align*}
    [\cO_{\widetilde U}(k\epsilon_i)]&=[\cO_{\widetilde U}(\epsilon_{i-1}+(k-1)\epsilon_i)]+[\cO_{\bP_{\alpha_{i-1}}}(-kx_{i})]\\
    &=\cdots\\
    &=[\cO_{\widetilde U}(k\epsilon_{i-1})]-\sum_{j=0}^{k-1}[\cO_{\bP_{\alpha_{i-1}}}(jx_{i-1}-(k-j)x_{i})].
\end{align*}
Now note the trick:
\begin{align}
    \sum_{j=0}^{k-1}[\cO_{\bP_{\alpha_{i-1}}}(jx_{i-1}-(k-j)x_{i})]&=\frac12\sum_{j=0}^{k-1}\big([\cO_{\bP_{\alpha_{i-1}}}(jx_{i-1}-(k-j)x_{i})]+[\cO_{\bP_{\alpha_{i-1}}}((k-j-1)x_{i-1}-(j+1)x_{i})]\big)\nonumber\\
    &=\frac12\sum_{j=0}^{k-1}\big([\cO_{\bP_{\alpha_{i-1}}}(jx_{i-1}-(j+1)x_{i})]+[\cO_{\bP_{\alpha_{i-1}}}((k-j-1)x_{i-1}-(k-j)x_i)]\big)\label{type-c-trick}\\
    &=\sum_{j=0}^{k-1}[\cO_{\bP_{\alpha_{i-1}}}(jx_{i-1}-(j+1)x_{i})],\nonumber
\end{align}
so by induction on the Bruhat order we conclude that $C_{w_{k\epsilon_i}}=[\cO_{\bP_{\alpha_{i-1}}}((k-1)x_{i-1}-kx_i)]$.

Next, we calculate $[\cO_{\widetilde U}(-k\epsilon_i)]$ where $i<n$. By a repeated application of Lemma~\ref{alpha-action} and Lemma~\ref{equivariant-structure-c}, we have:
\begin{align*}
[\cO_{\widetilde U}(-k\epsilon_i)]&=[\cO_{\widetilde U}(-(k-1)\epsilon_i-\epsilon_{i+1})]+[\cO_{\bP_{\alpha_i}}(-kx_i)]\\
&=\cdots\\
&=[\cO_{\widetilde U}(-k\epsilon_{i+1})]+\sum_{j=0}^{k-1}[\cO_{\bP_{\alpha_i}}(-(k-j)x_i+jx_{i+1})],
\end{align*}
where as in \eqref{type-c-trick},
\begin{align*}
\sum_{j=0}^{k-1}[\cO_{\bP_{\alpha_i}}(-(k-j)x_i+jx_{i+1})]&=\frac12\sum_{j=0}^{k-1}\big([\cO_{\bP_{\alpha_i}}(-(k-j)x_i+jx_{i+1})]+[\cO_{\bP_{\alpha_i}}(-(j+1)x_i+(k-j-1)x_{i+1})]\big)\\
&=\frac12\sum_{j=0}^{k-1}\big([\cO_{\bP_{\alpha_i}}(-(k-j)x_i+(k-j-1)x_{i+1})]+[\cO_{\bP_{\alpha_i}}(-(j+1)x_i+jx_{i+1})]\big)\\
&=\sum_{j=0}^{k-1}[\cO_{\bP_{\alpha_i}}(-(j+1)x_i+jx_{i+1})].
\end{align*}
Thus, $C_{w_{-k\epsilon_i}}=[\cO_{\bP_{\alpha_i}}(-kx_i+(k-1)x_{i+1})]$, as desired.

Next, we calculate $[\cO_{\widetilde U}(-k\epsilon_n)]$. By a repeated application of Lemma~\ref{alpha-action} and Lemma~\ref{equivariant-structure-c} and Remark~\ref{sym-o2-rep}, we have:
\begin{align*}
    [\cO_{\widetilde U}(-k\epsilon_n)]&=[\cO_{\widetilde U}(-(k-1)\epsilon_n)]+[\cO_{\bP_\beta}(k,-k)]\\
    &=\cdots\\
    &=[\cO_{\widetilde U}(k\epsilon_n)]+\sum_{j=1-k}^k[\cO_{\bP_\beta}(j,-j)].\\
    &=[\cO_{\widetilde U}(k\epsilon_n)]+\sum_{j=1}^{k}[\xi_{j-1/2}\otimes\cO_{\bP_\beta}(1,0)],
\end{align*}
where the last equality uses the observation:
\[
\sum_{j=1-k}^k[\cO_{\bP_\beta}(j,-j)]=[\cO_{\bP_\beta}(1-k,-k)]\otimes\sum_{j=0}^{2k-1}[\cO_{\bP_\beta}(j,2k-1-j)]=\sum_{j=1}^{k}[\xi_{j-1/2}\otimes\cO_{\bP_\beta}(1,0)]
\]
by Lemma~\ref{sym^n-calc}. Thus $C_{w_{-k\epsilon_n}}=[\xi_{k-1/2}\otimes\cO_{\bP_\beta}(-1,0)]$.

Next, for elements of the form $k\epsilon_1$ with $k>0$, by a repeated application of Lemma~\ref{alpha-action} and Lemma~\ref{equivariant-structure-c},
\begin{align*}
    [\cO_{\widetilde U}(k\epsilon_1)]&=[\cO_{\widetilde U}((k-1)\epsilon_1+\epsilon_2)]-[\cO_{\bP_{\alpha_1}}((k-1)x_1-x_2)]\\
    &=\cdots\\
    &=[\cO_{\widetilde U}((k-1)\epsilon_1)]-[\cO_{\bP_{\alpha_1}}((k-1)x_1-x_2)]-[\sgn^{k-1}\otimes\cO_{\bP_\beta}]-\sum_{i=2}^{n-1}[\cO_{\bP_{\alpha_i}}\!(-x_{i+1})]\\
    &=[\cO_{\widetilde U}((k-1)\epsilon_1)]-[\sgn^{k-1}\otimes\cO_{\cB_e}]+\big([\cO_{\bP_{\alpha_1}}\!(-x_2)]-[\cO_{\bP_{\alpha_1}}\big((k-1)x_1-x_2\big)]\big),
\end{align*}
where the last equality used \eqref{beta-exp}. Now note the $\Cent_e$-equivariant exact sequences:
\begin{align*}
    0\to\cO_{\bP_{\alpha_i}}\!(-x_{i+1})\to&\cO_{\bP_{\alpha_i}}((k-1)x_i-x_{i+1})\to\sum_{m=1}^{k-1}\C_{x_i}\langle m\rangle\to 0\\
    0\to\cO_{\bP_{\alpha_i}}((k-1)x_i-kx_{i+1})\to &\cO_{\bP_{\alpha_i}}((k-1)x_i-x_{i+1})\to\sum_{m=1}^{k-1}\C_{x_{i+1}}\langle m\rangle\to 0,
\end{align*}
together with the exact sequence
\[
0\to\xi_{m-1/2}\otimes\cO_{\bP_\beta}(1,0)\to\xi_{m}\otimes\cO_{\bP_\beta}\to \C_{x_n}\langle m\rangle\to0,
\]
which gives
\begin{align*}
[\cO_{\bP_{\alpha_1}}((k-1)x_1-x_2)]-[\cO_{\bP_{\alpha_1}}\!(-x_2)]=&\ \sum_{m=1}^{k-1}\big([\xi_m\otimes\cO_{\bP_\beta}]-[\xi_{m-1/2}\otimes\cO_{\bP_\beta}(1,0)]\big)\\&+\sum_{i=1}^{n-1}\big([\cO_{\bP_{\alpha_i}}((k-1)x_i-kx_{i+1})]-[\cO_{\bP_{\alpha_i}}\!(-x_{i+1})]\big).
\end{align*}
Combining everything,
\begin{align*}
    [\cO_{\widetilde U}(k\epsilon_1)]=&\ [\cO_{\widetilde U}((k-1)\epsilon_1)]-[\sgn^{k-1}\otimes\cO_{\bP_\beta}]-\sum_{i=1}^{n-1}\big([\cO_{\bP_{\alpha_i}}((k-1)x_i-kx_{i+1})]-[\cO_{\bP_{\alpha_i}}\!(-x_{i+1})]\big)\\
    &-\sum_{m=1}^{k-1}\big([\xi_m\otimes\cO_{\bP_\beta}]-[\xi_{m-1/2}\otimes\cO_{\bP_\beta}(1,0)]\big)\\
    =&\ [\cO_{\widetilde U}((k-1)\epsilon_1)]-\bigg([\sgn^{k-1}\otimes\cO_{\cB_e}]-\sum_{i=1}^{n-1}[\cO_{\bP_{\alpha_i}}\!(-x_{i+1})]\bigg)+\sum_{m=1}^{k-1}[\xi_{m-1/2}\otimes\cO_{\bP_\beta}(1,0)]\\
    &-\sum_{i=1}^{n-1}\big([\cO_{\bP_{\alpha_i}}((k-1)x_i-kx_{i+1})]-[\cO_{\bP_{\alpha_i}}\!(-x_{i+1})]\big)\\
    &-\sum_{m=1}^{k-1}\bigg([\xi_m\otimes\cO_{\cB_e}]-\sum_{i=1}^{n-1}\big(\cO_{\bP_{\alpha_i}}(-mx_i+(m-1)x_{i+1})+\cO_{\bP_{\alpha_i}}\!(mx_i-(m+1)x_{i+1})\big)\bigg)\\
    =&\ [\cO_{\widetilde U}((k-1)\epsilon_1)]-\sum_{m=1-k}^{2-k}\sum_{i=1}^{n-1}[\cO_{\bP_{\alpha_i}}\!(mx_i-(m+1)x_{i+1})]\\
    &+\sum_{m=0}^{k-2}[\xi_{m-1/2}\otimes\cO_{\bP_\beta}(1,0)]-\sum_{m=1}^{k-1}[\xi_m\otimes\cO_{\cB_e}],
\end{align*}
so $C_{w_{k\epsilon_1}}=-[\xi_{k-1}\otimes\cO_{\cB_e}]$.
\end{proof}

\begin{proof}[Proof of Proposition~\ref{main-thm} when $G$ is of type C]
It suffices to check that, under the map $\widetilde E^0_\subreg\otimes\Z_{\sgn}\to K^{G^\vee}\!(\widetilde U)$ sending $e_w$ to $C_w$, for each element $w\in c^0_\subreg$, the element $w\cdot e_1$ is sent to $T_w$. The formulas for $w\cdot1$ in terms of $e_v$ are given as (recall from Proposition~\ref{typec-subregular-cell} that $-\epsilon_1=s_1\cdots s_n\cdots s_0$), for example:
\begin{align*}
    w_{k\epsilon_1}\cdot1=s_0\big((1-k)\epsilon_1\big)\cdot1=&\ 1+\Big\lfloor\frac{k+1}2\Big\rfloor C_{s_0}+\Big\lfloor\frac k2\Big\rfloor C_{s_0s_1s_0}\\&+\sum_{m=0}^{k-1}(k-m-1)\big(C_{s_1s_0(-m\epsilon_1)}+\cdots C_{s_0\cdots s_n\cdots s_0(-m\epsilon_1)}\big)\\
    w_{k\epsilon_1}\cdot 1=(-k\epsilon_1)\cdot1=&\ 1+\Big\lfloor\frac{k+1}2\Big\rfloor C_{s_0}+\Big\lfloor\frac k2\Big\rfloor C_{s_0s_1s_0}\\&+\sum_{m=0}^{k-1}(k-m)\big(C_{s_1s_0(-m\epsilon_1)}+\cdots C_{s_0\cdots s_n\cdots s_0(-m\epsilon_1)}\big).
\end{align*}
Thus it suffices to check that
\[w_{k\epsilon_1}\cdot 1-w_{(k-1)\epsilon_1}\cdot 1=\begin{cases}
\displaystyle C_{s_0}+\sum_{m=0}^{k-1}\big(C_{s_1s_0(-m\epsilon_1)}+\cdots C_{s_0\cdots s_n\cdots s_0(-m\epsilon_1)}\big)&k\equiv1\pmod2\\
\displaystyle C_{s_0s_1s_0}+\sum_{m=0}^{k-1}\big(C_{s_1s_0(-m\epsilon_1)}+\cdots C_{s_0\cdots s_n\cdots s_0(-m\epsilon_1)}\big)&k\equiv0\pmod2
\end{cases}\]is sent to $t_{k\epsilon_1}[\cO_{\widetilde U}]-t_{(k-1)\epsilon_1}[\cO_{\widetilde U}]$. But this is exactly the last part of the proof of Proposition~\ref{prop:canonical-basis-b}.
\end{proof}

\subsection{When $G$ is of type F\textsubscript{4}} 
\begin{lemma}\label{equivariant-structure-f}
Let $G$ be of type F\textsubscript{4}, and let $\gamma=\sum_{i=1}^4a_i\epsilon_i\in Q^\vee$. Then
\begin{align*}
\cO_{\bP_{\alpha_1}}(\gamma)&\simeq\cO_{\bP_{\alpha_1}}\big((a_2-a_4)x_1+(a_4-a_3)y_1\big)\\ \cO_{\bP_{\alpha_2}}(\gamma)&\simeq\cO_{\bP_{\alpha_2}}\big((a_2-a_4)x_1+(a_3-a_2)y_2\big)\\
\cO_{\bP_{\alpha_3}}(\gamma)&\simeq \cO_{\bP_{\alpha_3}}\!(a_4x_2)\\
\cO_{\bP_{\alpha_4}}(\gamma)&\simeq \cO_{\bP_{\alpha_4}}\big(\frac12(a_1-a_2-a_3-a_4)x_3\big).
\end{align*}
\end{lemma}
\begin{proof}
By Lemma~\ref{restrict-line-bundle} we know the underlying line bundle of $\cO_{\bP_{\alpha_1}}(\gamma)$ is $\cO_{\bP^1}(a_2-a_3)$ and the underlying bundle of $\cO_{\bP_{\alpha_2}}(\gamma)$ is $\cO_{\bP^1}(a_3-a_4)$, so it suffices to calculate the $\Z/2$-equivariance of $\cO_{\cB_e}(\gamma)$ at $x_1$. Observing that $\cO_{\bP_{\alpha_1}}\!(m\alpha_1)\simeq\omega_{\bP_{\alpha_1}}^{\otimes-m}$, we see the equivariant structure on $\cO_{\cB_e}\!(m\alpha_1)$ at $x_1$ is $(-1)^{m}$. Similarly, $\cO_{\cB_e}\!(m\alpha_2)$ at $x_1$ is $(-1)^{m}$. Moreover, $\cO_{\cB_e}(\alpha_3)$ and $\cO_{\cB_e}(\alpha_4)$ has trivial equivariant structure at $x_1$ since $x_1$ is not in $\bP_{\alpha_3}$ or $\bP_{\alpha_4}$. Thus in general the $\Z/2$-equivariant structure of $\cO_{\cB_e}(\gamma)$ at $x_1$ is $(-1)^{a_1-a_3}$.
\end{proof}

We proceed as in \S\ref{sec:type-b-canonical}. Repeatedly using Lemma~\ref{alpha-action} and Lemma~\ref{equivariant-structure-f}, we obtain:
\begin{align*}
[\cO_{\widetilde U}(\epsilon_1+\epsilon_3)]&=[\cO_{\widetilde U}(\epsilon_1+\epsilon_2)]+[\cO_{\bP_{\alpha_1}}(-y_1)]\\
[\cO_{\widetilde U}(\epsilon_1+\epsilon_4)]&=[\cO_{\widetilde U}(\epsilon_1+\epsilon_3)]+[\cO_{\bP_{\alpha_2}}\!(-x_1)]\\
[\cO_{\widetilde U}(\epsilon_1-\epsilon_4)]&=[\cO_{\widetilde U}(\epsilon_1+\epsilon_4)]+[\cO_{\bP_{\alpha_3}}\!(-x_2)]\\
[\cO_{\widetilde U}(\epsilon_2+\epsilon_3)]&=[\cO_{\widetilde U}(\epsilon_1-\epsilon_4)]+[\cO_{\bP_{\alpha_4}}\!(-x_3)]\\
[\cO_{\widetilde U}(\epsilon_1-\epsilon_3)]&=[\cO_{\widetilde U}(\epsilon_1-\epsilon_4)]+[\cO_{\bP_{\alpha_2}}(-y_2)]\\
[\cO_{\widetilde U}(\epsilon_1-\epsilon_2)]&=[\cO_{\widetilde U}(\epsilon_1-\epsilon_3)]+[\cO_{\bP_{\alpha_1}}\!(-x_1)].
\end{align*}
Thus, we obtain the following description of the canonical basis of F\textsubscript{4}:
\begin{thm}
Let $G$ be of type F\textsubscript{4}. The canonical basis of $K^{G^\vee}\!(\widetilde U)$ are:
\begin{align*}
C_1&=C_{w_0}=[\cO_{\widetilde U}]\\
C_{s_0}&=C_{w_{\epsilon_1+\epsilon_2}}=-[\cO_{\cB_e}]\\
C_{s_1s_0}&=C_{w_{\epsilon_1+\epsilon_3}}=[\cO_{\bP_{\alpha_1}}(-y_1)]\\
C_{s_2s_1s_0}&=C_{w_{\epsilon_1+\epsilon_4}}=[\cO_{\bP_{\alpha_2}}\!(-x_1)]\\
C_{s_3s_2s_1s_0}&=C_{w_{\epsilon_1-\epsilon_4}}=[\cO_{\bP_{\alpha_3}}\!(-x_2)]\\
C_{s_4s_3s_2s_1s_0}&=C_{w_{\epsilon_2+\epsilon_3}}=[\cO_{\bP_{\alpha_4}}\!(-x_3)]\\
C_{s_2s_3s_2s_1s_0}&=C_{w_{\epsilon_2-\epsilon_3}}=[\cO_{\bP_{\alpha_2}}(-y_2)]\\
C_{s_1s_2s_3s_2s_1s_0}&=C_{w_{\epsilon_1-\epsilon_2}}=[\cO_{\bP_{\alpha_1}}\!(-x_1)]\\
C_{s_0s_1s_2s_3s_2s_1s_0}&=C_{w_{2\epsilon_1}}=-[\sgn\otimes\cO_{\cB_e}].
\end{align*}
\end{thm}

Finally, we may prove:
\begin{proof}[Proof of Proposition~\ref{main-thm} when $G$ if of type F\textsubscript{4}]
    The proof is exactly the same as in type B.
\end{proof}

\subsection{When $G$ is of type G\textsubscript{2}}

\begin{lemma}\label{equivariant-structure-g}
Let $\fg$ be of type G\textsubscript{2}, and let $\gamma=a\alpha^\vee+b\beta^\vee\in Q^\vee$. Then as $S_3$-equivariant line bundles:
\begin{itemize}
\item $\cO_{\bP_\alpha}(\gamma)\simeq\cO_{\bP_{\alpha}}(3b-a,a)$.
\item $\cO_{\bP_\beta}(\gamma)\simeq\cO_{\bP_\beta}((b-a)z+by)$.
\end{itemize}
\end{lemma}

We obtain the following description of the canonical basis of $K^{G^\vee}\!(\widetilde U)$:
\begin{prop}\label{prop:canonical-basis-g}
    Let $\fg$ be of type G\textsubscript2. Then
    \begin{align*}
    C_1&=[\cO_{\widetilde U}]\\
    C_{s_0}&=-[\cO_{\cB_e}]\\
    C_{s_1s_0}&=[\cO_{\bP_\alpha}(0,-1)]\\
    C_{s_2s_1s_0}&=[\cO_{\bP_\beta}\!(-x)]\\
    C_{s_1s_2s_1s_0}&=[\std\otimes\cO_{\bP_\alpha}(0,1)]\\
    C_{s_0s_1s_2s_1s_0}&=-[\std\otimes\cO_{\cB_e}]\\
    C_{s_2s_1s_2s_1s_0}&=[\cO_{\bP_\beta}(-y)]\\
    C_{s_1s_2s_1s_2s_1s_0}&=[\cO_{\bP_\alpha}(1,0)]\\
    C_{s_0s_1s_2s_1s_2s_1s_0}&=-[\sgn\otimes\cO_{\cB_e}].
\end{align*}
\end{prop}
\begin{proof}
We have
\begin{align}
[\cO_{\widetilde U}(\alpha^\vee+\beta^\vee)]&=[\cO_{\widetilde U}(2\alpha^\vee+\beta^\vee)]+[\cO_{\bP_\alpha}(\alpha^\vee+\beta^\vee)]\nonumber\\
&=[\cO_{\widetilde U}(2\alpha^\vee+\beta^\vee)]+[\cO_{\bP_\alpha}(0,-1)],\label{eq:alpha+beta}
\end{align}
so $C_{s_1s_0}=C_{w_{\alpha^\vee+\beta^\vee}}=[\cO_{\bP_\alpha}(0,-1)]$.
\begin{align}
    [\cO_{\widetilde U}(\alpha^\vee)]&=[\cO_{\widetilde U}(\alpha^\vee+\beta^\vee)]+[\cO_{\bP_\beta}(\alpha^\vee)]\nonumber\\
    &=[\cO_{\widetilde U}(\alpha^\vee+\beta^\vee)]+[\cO_{\bP_\beta}\!(-x)],\label{eq:alpha}
\end{align}
so $C_{s_2s_1s_0}=C_{w_{\alpha^\vee}}=[\cO_{\bP_\beta}\!(-x)]$. Note that in particular, since $[\cO_{\widetilde U}(\alpha^\vee)]=[\cO_{\widetilde U}]-[\cO_{\bP_\alpha}]$ this implies
\begin{equation}\label{g2:Oalpha-exp}
[\cO_{\bP_\alpha}]=[\cO_{\cB_e}]-[\cO_{\bP_\beta}\!(-x)]-[\cO_{\bP_\alpha}(0,-1)].
\end{equation}
Again,
\begin{align}
    [\cO_{\widetilde U}(-\alpha^\vee)]&=[\cO_{\widetilde U}(\alpha^\vee)]+[\cO_{\bP_\alpha}]+[\cO_{\bP_\alpha}(-\alpha^\vee)]\nonumber\\
    &=[\cO_{\widetilde U}(\alpha^\vee)]+[\cO_{\bP_\alpha}]+[\cO_{\bP_\alpha}(1,-1)]\label{eq:-alpha}\\
    &=[\cO_{\widetilde U}(\alpha^\vee)]+[\std\otimes\cO_{\bP_\alpha}(1,0)],\nonumber
\end{align}
so $C_{s_1s_2s_1s_0}=C_{w_{-\alpha^\vee}}=[\std\otimes\cO_{\bP_\alpha}(1,0)]$, where the last equality is by twisting the Euler sequence (which is in particular $S_3$-equivariant)
\begin{equation}\label{s3-equiv-euler}
0\to\cO_{\bP_\alpha}(1,0)\to\std\otimes\cO_{\bP_\alpha}\to\cO_{\bP_\alpha}(0,1)\to 0
\end{equation}
by $\cO_{\bP_\alpha}(0,-1)$. Again,
\begin{align}
    [\cO_{\widetilde U}(-\alpha^\vee-\beta^\vee)]&=[\cO_{\widetilde U}(-\alpha^\vee)]+[\cO_{\bP_\beta}(-\alpha^\vee-\beta^\vee)]\nonumber\\
    &=[\cO_{\widetilde U}(-\alpha^\vee)]+[\cO_{\bP_\beta}(-y)],\label{eq:-alpha-beta}
\end{align}
so $C_{s_2s_1s_2s_1s_0}=C_{w_{-\alpha^\vee-\beta^\vee}}=[\cO_{\bP_\beta}(-y)]$.
\begin{align}
    [\cO_{\widetilde U}(-2\alpha^\vee-\beta^\vee)]&=[\cO_{\widetilde U}(-\alpha^\vee-\beta^\vee)]+[\cO_{\bP_\alpha}(-2\alpha^\vee-\beta^\vee)]\nonumber\\
    &=[\cO_{\widetilde U}(-\alpha-\beta)]+[\cO_{\bP_\alpha}(1,0)]\label{eq:-2alpha-beta}
\end{align}
so $C_{s_1s_2s_1s_2s_1s_0}=C_{w_{-2\alpha^\vee-\beta^\vee}}=[\cO_{\bP_\alpha}(1,0)]$. The computation for $[\cO_{\widetilde U}(3\alpha^\vee+\beta^\vee)]$ and $[\cO_{\widetilde U}(4\alpha^\vee+2\beta^\vee)]$ are slightly more complicated: 
\begin{align}
    [\cO_{\widetilde U}(3\alpha^\vee+\beta^\vee)]=&\ [\cO_{\widetilde U}(2\alpha^\vee+\beta^\vee)]-[\cO_{\bP_\alpha}(2\alpha^\vee+\beta^\vee)]\nonumber\\
    =&\ [\cO_{\widetilde U}(2\alpha^\vee+\beta^\vee)]-[\cO_{\bP_\alpha}(-1,0)]\nonumber\\
    =&\ [\cO_{\widetilde U}(2\alpha^\vee+\beta^\vee)]+[\cO_{\bP_\alpha}(0,-1)]-[\std\otimes\cO_{\bP_\alpha}]\nonumber\\
    =&\ [\cO_{\widetilde U}(2\alpha^\vee+\beta^\vee)]+[\cO_{\bP_\alpha}(0,-1)]-[\std\otimes\cO_{\cB_e}]\label{eq:3alpha+beta}\\&+[\cO_{\bP_\beta}\!(-x)]+[\cO_{\bP_\beta}(-y)]+[\std\otimes\cO_{\bP_\alpha}(1,0)],\nonumber
\end{align}
where the third equality used the Euler sequence \eqref{s3-equiv-euler} and the last equality is \eqref{g2:Oalpha-exp} tensored by $\std$. Thus, $C_{s_0s_1s_2s_1s_0}=C_{w_{3\alpha^\vee+\beta^\vee}}=-[\std\otimes\cO_{\cB_e}]$. Similarly,
\begin{align}
    [\cO_{\widetilde U}(4\alpha^\vee+2\beta^\vee)]&=[\cO_{\widetilde U}(3\alpha^\vee+\beta^\vee)]-[\cO_\beta(3\alpha^\vee+\beta^\vee)]-[\cO_{\bP_\alpha}(3\alpha^\vee+2\beta^\vee)]\nonumber\\
&=[\cO_{\widetilde U}(3\alpha^\vee+\beta^\vee)]+[\cO_{\bP_\beta}(-y)]-[\sgn\otimes\cO_{\bP_\alpha}]\label{eq:4alpha+2beta}\\
&=[\cO_{\widetilde U}(3\alpha^\vee+\beta^\vee)]-[\sgn\otimes\cO_{\cB_e}]+[\cO_{\bP_\alpha}(1,0)],\nonumber
\end{align}
where the last equality is \eqref{g2:Oalpha-exp} tensored by $\sgn$. Hence $C_{s_0s_1s_2s_1s_2s_1s_0}=C_{w_{4\alpha^\vee+2\beta^\vee}}=-[\sgn\otimes\cO_{\cB_e}]$.
\end{proof}

Now, we have:
\begin{proof}[Proof of Proposition~\ref{main-thm} when $G$ is of type G\textsubscript{2}]
It suffices to check the basis in Proposition~\ref{prop:canonical-basis-b} satisfies \eqref{e-defn}.

\eqref{eq:alpha+beta} gives
\(
s_1s_0C_1=C_1+C_{s_0}+C_{s_1s_0},
\)
so
\[
    s_1C_{s_0}=s_1(-s_0C_1-C_1)=-s_1s_0C_1+C_1=-C_{s_0}-C_{s_1s_0}.
\]
Similarly,
\[
    s_2C_{s_0}=s_2(-s_0C_1-C_1)=s_0C_1+C_1=-C_{s_0},\]
and
\begin{align*}
    s_0C_{s_0}&=s_0(-s_0C_1-C_1)\\
    &=-C_1-s_0C_1\\
    &=-C_{s_0}.
\end{align*}
Together, we have \eqref{e-defn} for $C_{s_0}$. Similarly, \eqref{eq:alpha} gives
\[
-s_2s_1s_0C_1=C_1+C_{s_0}+C_{s_1s_0}+C_{s_2s_1s_0},
\]
which gives \eqref{e-defn} for $C_{s_1s_0}$; \eqref{eq:-alpha} implies
\begin{align*}
s_1s_2s_1s_0C_1&=-s_2s_1s_0C_1+C_{s_1s_2s_1s_0}\\
&=C_1+C_{s_0}+C_{s_1s_0}+C_{s_2s_1s_0}+C_{s_1s_2s_1s_0},
\end{align*}
which gives \eqref{e-defn} for $C_{s_2s_1s_0}$ and $C_{s_2s_1s_2s_1s_0}$; \eqref{eq:-alpha-beta} implies
\begin{align*}
    -s_2s_1s_2s_1s_0C_1&=s_1s_2s_1s_0C_1+C_{s_2s_1s_2s_1s_0}\\
    &=C_1+C_{s_0}+C_{s_1s_0}+C_{s_2s_1s_0}+C_{s_1s_2s_1s_0}+C_{s_2s_1s_2s_1s_0}
\end{align*}
and \eqref{eq:-2alpha-beta} implies
\begin{align*}
    -s_0s_1s_2s_1s_0C_1&=s_1s_2s_1s_0C_1+C_{s_0s_1s_2s_1s_0}\\
    &=C_1+C_{s_0}+C_{s_1s_0}+C_{s_2s_1s_0}+C_{s_1s_2s_1s_0}+C_{s_0s_1s_2s_1s_0},
\end{align*}
which together gives \eqref{e-defn} for $C_{s_1s_2s_1s_0}$ and $C_{s_0s_1s_2s_1s_0}$; and \eqref{eq:4alpha+2beta} implies
\begin{align}
    -s_0s_1s_2s_1s_2s_1s_0C_1=&-s_0s_1s_2s_1s_0C_1+C_{s_1s_2s_1s_2s_1s_0}+C_{s_0s_1s_2s_1s_2s_1s_0}\\
    =&\ C_1+C_{s_0}+C_{s_1s_0}+C_{s_2s_1s_0}+C_{s_1s_2s_1s_0}+C_{s_0s_1s_2s_1s_0}\\&+C_{s_1s_2s_1s_2s_1s_0}+C_{s_0s_1s_2s_1s_2s_1s_0},
\end{align}
the sum over all eight elements of the subregular cell. The above together imply \eqref{e-defn} for $C_{s_1s_2s_1s_2s_1s_0}$ and $C_{s_0s_1s_2s_1s_2s_1s_0}$.
\end{proof}

\section{Re-phrase in the language of the Cartan of Kac-Moody algebras}\label{sec:cartan-of-kac-moody}

In types B and F, we prove the following Conjecture of \cite[\S7]{BKK}, which describes Lusztig's module $E_\subreg^0$ in terms of the Cartan of Kac-Moody algebras. We recall the conjecture here. Let $\widehat\fg$ be the affine Lie algebra associated to $\fg$, and let $\widehat\fg^\vee$ be the dual Lie algebra.\footnote{Note that this is different from $\widehat{\fg^\vee}$, the affine Lie algebra associated to $\fg^\vee$.} Let $\mathfrak k$ be the Lie algebra of an unfolding of $\widehat\fg^\vee$ (so $\mathfrak k$ is an un-cofolding of $\widehat\fg$: long roots of $\widehat\fg$ correspond to multiple roots in $\mathfrak k$.)

Now there is an embedding $\widehat W\subset W(\mathfrak k)$ that sends a simple reflection of $\widehat W$ to the product of simple reflections over the corresponding orbit in the unfolding. Let $\widetilde\ft_\Z\subset\mathfrak k$ be the integral form of the reflection representation of $W(\mathfrak k)$, and let $\widehat\ft_\Z=\widetilde\ft_\Z\oplus\Z d$ be the Cartan representation.

Formally, let $\{\alpha_1,\dots,\alpha_m\}$ be the vertices of the Dynkin diagram of $\mathfrak k$, where $\alpha_1$ is in the orbit of the affine vertex of the Dynkin diagram of $\widehat{\fg}$. Then $W(\mathfrak k)$ is the group generated by $s_1,\dots,s_m$ with relations $s_i^2=1$, and $(s_is_j)^2=1$ if $\alpha_i$ and $\alpha_j$ are not connected on the Dynkin diagram and $(s_is_j)^3=1$ if $\alpha_i$ and $\alpha_j$ are connected.\footnote{By definition, $\mathfrak k$ is simply-laced.} Then $\widetilde\ft_\Z$ is the free abelian group generated by the $\alpha_1,\dots,\alpha_m$ such that 
\[
s_i\alpha_j=\begin{cases}
   -\alpha_j& i=j\\
   \alpha_j&\text{$\alpha_i$ and $\alpha_j$ are not connected}\\
   \alpha_j+\alpha_i&\text{$\alpha_i$ and $\alpha_j$ are connected}.
\end{cases}
\]
The Cartan representation $\widetilde\ft_\Z=\widehat\ft_\Z\oplus\Z h$ is a one-dimensional extension, defined by:
\[
s_ih=\begin{cases}
    h&i\ne1\\
    h-\alpha_1&i=1.
\end{cases}
\]
Then:
\begin{prop}[{\cite[Conjecture~7.3]{BKK}}]\label{prop_descr_BF_via_Cartan}
    Assume $\fg$ is of type $\text{B}_n$ with $n\geqslant3$, or $F_4$. Then there are isomorphisms of $\widehat W$-modules
    \[
    K^{\Cent_e}(\cB_e)\simeq\widetilde\ft_\Z\otimes\Z_{\sgn},\hspace{0.3cm}K^{G^\vee}\!(\widetilde U)\simeq\widehat\ft_\Z\otimes\Z_{\sgn}.
    \]
\end{prop}
\begin{proof}
By Proposition~\ref{main-thm}, it suffices to check the isomorphisms
\[
    E^0_\subreg\simeq\widetilde\ft_\Z,\hspace{0.3cm}\widetilde E^0_\subreg\simeq\widehat\ft_\Z.
    \]
    When $\fg$ is of type $\text{B}_n$, the Dynkin diagram of $\widehat\fg$ is
    \[
    \dynkin[extended,root
radius=.08cm,edge length=1cm,labels={\alpha_0,\alpha_1,\alpha_2,\alpha_3,\alpha_{n-2},\alpha_{n-1},\alpha_n}]B{}
    \]
    so the Dynkin diagram of $\mathfrak k$ looks like
    \[
    \dynkin[extended,root
radius=.08cm,edge length=1cm,labels={\widetilde\alpha_0,\widetilde\alpha_1,\widetilde\alpha_2,\widetilde\alpha_3,\widetilde\beta_3,\widetilde\beta_2,\widetilde\beta_0,\widetilde\beta_1}]D{}
    \]
    where the middle vertex is $\widetilde\alpha_n=\widetilde\beta_n$. Let $s_i\in W(\mathfrak k)$ be the reflection corresponding to $\widetilde\alpha_i$ and let $t_i\in W(\mathfrak k)$ be the reflection corresponding to $\widetilde\beta_i$. The associated embedding of the Coxeter groups is given as (see \cite[Section 3.3]{lusztig-some-examples})
    \begin{align}
    \widehat W&\hookrightarrow W(\mathfrak k)\nonumber\\
    s_{\alpha_i}&\mapsto\begin{cases}
        s_it_i&i<n\\
        s_n&i=n.
    \end{cases}\label{emb_w_unfold}
    \end{align}
Note that the Dynkin graph of $\mathfrak k$ is isomorphic to the graph $\Gamma_0$ in Proposition~\ref{typec-subregular-cell}, and there is an isomorphism
\begin{align*}
    \widetilde E^0_\subreg&\simeq \widehat\ft_\Z\\
    e_1&\mapsto d\\
    e_{w_{\epsilon_1+\epsilon_2}}&\mapsto -\widetilde\alpha_0\\
    e_{w_{\epsilon_2+\epsilon_3}}&\mapsto -\widetilde\alpha_1\\
    e_{w_{\epsilon_1+\epsilon_{i+1}}}&\mapsto -\widetilde\alpha_{i}\text{ for $2\leqslant i\leqslant n-1$}\\
    e_{w_{\epsilon_1-\epsilon_i}}&\mapsto -\widetilde\beta_i\text{ for $i\geqslant2$}\\
    e_{w_{\epsilon_2-\epsilon_1}}&\mapsto -\widetilde\beta_1\\
    e_{w_{2\epsilon_2}}&\mapsto -\widetilde\beta_0
\end{align*}
The proof for F\textsubscript{4} is similar, by observing that the Dynkin diagram of $\mathfrak k$ is of type $\widetilde{\text{E}}_7$, the same as Proposition~\ref{typef-subregular-cell}.
\end{proof}

Let us now generalize Proposition \ref{prop_descr_BF_via_Cartan} to the case of arbitrary (simple) Lie algebra $\mathfrak{g}$.
Let us identify $K^{\Cent_e}(\cB_e)$, $K^{\Cent_e}(\widetilde{U})$ with Cartan subalgebras of certain Lie algebras and describe the $\widehat{W}$-module structures in these terms. 

Let $\mathfrak l$ be the Kac-Moody Lie algebra corresponding to the diagram $\Gamma_0$. Let $\widetilde{\ft}_{\mathbb{Z}} \subset \mathfrak l$ be the integral form of the reflection representation of $W(\mathfrak{l})$. Recall that the vertices of $\Gamma_0$ are parameterized by the elements of $c^0_{\mathrm{subreg}}$. For $w \in c^0_{\mathrm{subreg}}$ we denote by $\alpha_w$ the corresponding simple root of $\mathfrak{l}$, and by $e_{\alpha_w}$, $f_{\alpha_w}$ the corresponding Chevalley generators. Consider the grading on $\mathfrak l$ given by 
\begin{equation*}
\operatorname{deg}e_{\alpha_{s_0}}=1,\, \operatorname{deg}f_{\alpha_{s_0}}=-1,\, \operatorname{deg}e_{\alpha_{w}}=\operatorname{deg}f_{\alpha_{w}}=0~\text{for}~w \neq s_0,\, \operatorname{deg}\widetilde{\ft}_{\mathbb{Z}}=0.
\end{equation*}
Let $\mathfrak{k}\colonequals\mathfrak{l} \oplus \mathbb{C}h$ be the corresponding extension of $\mathfrak{k}$: the operator $\operatorname{ad}(h)$ acts on $\mathfrak{k}$ via the grading as above. Set $\widehat{\ft}_{\Z}\colonequals\widetilde{\ft}_{\Z} \oplus \Z h$. By abusing notation, for $w \in c^0_{\subreg}$ we denote by $\alpha_w \in \widehat{\ft}_{\Z}^*$ the $\widehat{\ft}_{\Z}$-weight of $e_{\alpha_w}$.

\begin{remark}
Recall that for $\mathfrak{g}=\text{B}_n$ the Lie algebra $\mathfrak{k}$ will be affine of type $\widetilde{\text{D}}_{2n}$, and for $\mathfrak{g}=\text{F}_4$ the Lie algebra $\mathfrak{k}$ will be affine of type $\widetilde{\text{E}}_{7}$. Note that for   $\mathfrak{g}=\text{G}_2$, $\mathfrak{k}$ will also be the affine Lie algebra, corresponding to $\widetilde{\text{E}}_{7}$. So, in all of these cases,  $\mathfrak{k}$ will be an affine Lie algebra (without any twisting).
\end{remark}

Recall that $W(\mathfrak l)=W(\Gamma_0)$, the Coxeter group associated to the graph $\Gamma_0$, has generators $s_{\alpha_w}$ for $w\in c^0_\subreg$, with relations $s_{\alpha_w}^2=1$ and $(s_{\alpha_w}s_{\alpha_{w'}})^2=1$ if $w$ and $w'$ are disconnected in $\Gamma_0$ and $(s_{\alpha_w}s_{\alpha_{w'}})^3=1$ if $w$ and $w'$ are connected in $\Gamma_0$. Since the diagram $\Gamma_0$ is possibly infinite, we consider a completion of $W(\mathfrak l)$:
\[
W(\mathfrak l)^{\land}\colonequals\lim_{\substack{\longleftarrow\\ n}}W(\Gamma_0^{\leqslant n}),
\]
where $\Gamma_0^{\leqslant n}\subset\Gamma$ is the induced subgraph supported on
\[(c^0_\subreg)^{\leqslant n}\colonequals\{w\in c^0_\subreg:\ell(w)\leqslant n\}\subset c^0_\subreg.\]
Since each vertex in $\Gamma_0$ is only connected to finitely many other elements of $c^0_\subreg$, the action of $W(\mathfrak l)$ on $\widetilde\ft_\Z$ and $\widehat\ft_\Z$ extends to the completion $W(\mathfrak l)^\land$.
Now, similar to \eqref{emb_w_unfold}, there is an embedding
\begin{align*}
    \hatW&\longhookrightarrow W(\mathfrak l)^{\land}\\
    s\in\widehat S&\longmapsto\prod_{\substack{w\in c^0_\subreg\\\mu(w)=s}}s_{\alpha_w}.
\end{align*}
Note that even though the product is possibly infinite (in types A and C), the product is well-defined in the completion $W(\mathfrak l)^\land$.

Using the embedding as above, we obtain an action of $\widehat{W}$ on $\widetilde{\ft}_{\Z}$ and $\widehat{\ft}_{\Z}$.

\begin{prop}
    There are isomorphisms of $\widehat W$-modules
    \[
    K^{\Cent_e}(\cB_e)\simeq\widetilde\ft_\Z\otimes\Z_{\sgn},\hspace{0.3cm}K^{G^\vee}\!(\widetilde U)\simeq\widehat\ft_\Z\otimes\Z_{\sgn}.
    \]
These isomorphisms send $[C_w]$ to $-\alpha_{w}$ for $w \in c^0_{\mathrm{subreg}}$, and $[C_1]$ maps to $h$.
\end{prop}
\begin{proof}
By direct computation. 
\end{proof}

\section{Calculation of special values of Kazhdan-Lusztig polynomials}\label{sec:KL-computation}

\subsection{When $G$ is of type B}

We hope to express $\cO_{\widetilde U}(\gamma)\in K^{G^\vee}\!(\widetilde U)$ in terms of the canonical basis, where $\gamma=a_1\alpha_1^\vee+\cdots+a_{n-1}\alpha_{n-1}^\vee+b\beta^\vee$. By a repeated application of Lemma~\ref{alpha-action}, we know $[\cO_{\widetilde U}]-t_\gamma[\cO_{\widetilde U}]$ equals
\begin{align*}
[\cO_{\widetilde U}]-[\cO_{\widetilde U}(\gamma)]&=\sum_{i=0}^{a_1-1}[\cO_{\bP_{\alpha_1}}(i\alpha_1^\vee)]+\sum_{i=0}^{a_2-1}[\cO_{\bP_{\alpha_2}}\!(a_1\alpha_1^\vee+i\alpha_2^\vee)]+\cdots+\sum_{i=0}^{b-1}[\cO_{\bP_\beta}\!(a_1\alpha_1^\vee+\cdots+a_{n-1}\alpha_{n-1}^\vee+i\beta^\vee)]\\
&=\sum_{i=0}^{a_1-1}[\cO_{\bP_{\alpha_1}}(ix_1+ix_2)]+\sum_{i=0}^{a_2-1}[\cO_{\bP_{\alpha_2}}((i-a_1)x_2+ix_3)]+\cdots+\sum_{i=0}^{b-1}[\cO_{\bP_\beta}((2i-a_{n-1})y)],
\end{align*}
which is now a class in $K^{\Cent_e}(\cB_e)$. We calculate $(1+\sgn)$ and $(1-\sgn)$ times the above class.

\begin{prop}\label{prop:1+sigma} Adopt the convention that $a_0=0$. Then
\[
(1+\sgn)\big([\cO_{\widetilde U}]-t_\gamma[\cO_{\widetilde U}]\big)=\sum_{j=1}^{n-1}a_j(1+\sgn)[\cO_{\bP_{\alpha_j}}\!(-x_j)]+2b[\cO_{\bP_\beta}(-y)]+\|\gamma\|[\C_y],
\]
where $\|\gamma\|=a_1^2+\cdots+a_{n-1}^2+2b^2-(a_1a_2+\cdots+a_{n-1}a_{n-1}+2a_{n-1}b)$ is the square length of $\gamma$ normalized so the short root has length $1$.
\end{prop}
\begin{proof}
Note that $(1+\sgn)(\sgn^{k}\otimes[\C_{x_i}])=[\C_y]$ in $K$-theory for all $1\leqslant i\leqslant n$, so:
\begin{align*}
(1+\sgn)\big([\cO_{\widetilde U}]-t_\gamma[\cO_{\widetilde U}]\big)&=\sum_{j=1}^{n-1}\sum_{i=0}^{a_j-1}\big((1+\sgn)[\cO_{\bP_{\alpha_j}}]+(2i-a_{j-1})[\C_y]\big)+2\sum_{i=0}^{b-1}\big([\cO_{\bP_\beta}]+(2i-a_{n-1})[\C_y]\big)\\
&=(1+\sgn)\sum_{j=1}^{n-1}a_j[\cO_{\bP_{\alpha_j}}]+2b[\cO_{\bP_\beta}]+\big(\sum_{j=1}^{n-1}\sum_{i=0}^{a_j-1}(2i-a_{j-1})+2\sum_{i=0}^{b-1}(2i-a_{n-1})\big)[\C_y]\\
&=(1+\sgn)\sum_{j=1}^{n-1}a_j[\cO_{\bP_{\alpha_j}}]+2b[\cO_{\bP_\beta}]+N[\C_y]
\end{align*}
where $N\colonequals\sum_{j=1}^{n-1}(a_j^2-a_j-a_{j-1}a_j)+2(b^2-b-a_{n-1}b)$. Finish by using: \begin{align*}(1+\sgn)[\cO_{\bP_{\alpha_j}}]-[\C_y]&=(1+\sgn)[\cO_{\bP_{\alpha_j}}\!(-x_j)]\\
[\cO_{\bP_{\beta}}]-[\C_y]&=[\cO_{\bP_{\beta}}(-y)],
\end{align*}
and $N=\|\gamma\|-a_1-\cdots-a_{n-1}-2b$.
\end{proof}

To calculate $(1-\sgn)([\cO_{\widetilde U}]-t_\gamma[\cO_{\widetilde U}])$, we first have:
\begin{lemma}\label{lem:sigma-1-calc}
Let $k,\ell\in\Z$, and let $1\leqslant i\leqslant n-1$ be arbitrary. Then in $K^{\Cent_e}(\cB_e)$,
\[
(1-\sgn)[\cO_{\bP_{\alpha_i}}\!(mx_i+nx_{i+1})]=(1-\sgn)[\cO_{\bP_{\alpha_i}}]+\sgn(1-\sgn^m)[\C_{x_i}]+\sgn(1-\sgn^n)[\C_{x_{i+1}}].
\]
\end{lemma}
\begin{proof}
By induction on the exact sequences
\begin{align*}0\to\cO_{\bP_{\alpha_i}}((m-1)x_i+nx_{i+1})\to&\cO_{\bP_{\alpha_i}}\!(mx_i+nx_{i+1})\to\sgn^m\otimes\C_{x_i}\to0\\
0\to\cO_{\bP_{\alpha_i}}\!(mx_i+(n-1)x_{i+1})\to&\cO_{\bP_{\alpha_i}}\!(mx_i+nx_{i+1})\to\sgn^n\otimes\C_{x_{i+1}}\to0.\qedhere
\end{align*}
\end{proof}

\begin{prop}\label{prop:1-sigma} In $K^{\Cent_e}(\cB_e)$,
\[
(1-\sgn)\big([\cO_{\widetilde U}]-t_\gamma[\cO_{\widetilde U}]\big)=\frac12\sum_{j=1}^{n}(1-\sgn^{a_j-a_{j-1}})[\C_{x_j}].
\]
\end{prop}
\begin{proof}
By Lemma~\ref{lem:sigma-1-calc}, and observing that $1-\sgn=\frac12(1-\sgn)^2$,
\begin{align*}
(1-\sgn)\sum_{i=0}^{a_j-1}[\cO_{\bP_{\alpha_j}}((i-a_{j-1})x_j+ix_{j+1})]&=\frac12(1-\sgn)\sum_{i=0}^{a_j-1}\big((1-\sgn)[\cO_{\bP_{\alpha_j}}]+\sgn(1-\sgn^{i-a_{j-1}})[\C_{x_j}]\\&\hspace{3cm}+\sgn(1-\sgn^{i})[\C_{x_{j+1}}]\big)\\
&=a_j(1-\sgn)[\cO_{\bP_{\alpha_j}}]+\frac12(1-\sgn)\sum_{i=0}^{a_j-1}\big(\sgn^{i-a_{j-1}}[\C_{x_j}]+\sgn^i[\C_{x_{j+1}}]\big)\\
&\hspace{3cm}-\frac12a_j(1-\sgn)\big([\C_{x_j}]+[\C_{x_{j+1}}]\big)\\
&=\frac12(\sgn^{-a_{j-1}}-\sgn^{a_j-a_{j-1}})[\C_{x_j}]+\frac12(1-\sgn^{a_j})[\C_{x_{j+1}}],
\end{align*}
where the last equation uses that $2[\cO_{\bP_{\alpha_j}}]-[\C_{x_j}]-[\sgn\otimes\C_{x_j}]=(1+\sgn)[\cO_{\bP_{\alpha_j}}\!(-x_j)]$ is killed by $1-\sgn$. Thus, in total,
\[
(1-\sgn)\big([\cO_{\widetilde U}]-t_\gamma[\cO_{\widetilde U}]\big)=\frac12\sum_{j=1}^{n}(1-\sgn^{a_j-a_{j-1}})[\C_{x_j}].\qedhere
\]
\end{proof}

\begin{thm}\label{KL-polynomial-computation-B} Let $G$ be of type $\text{B}_n$ where $n\geqslant3$. For any $\gamma=a_1\alpha_1^\vee+\cdots+a_{n-1}\alpha_{n-1}^\vee+b\beta^\vee\in Q^\vee$,
\begin{align*}
\mathbf m_{w_{\epsilon_1+\epsilon_2}}^{w_\gamma}&=\frac 12\|\gamma\|+\frac14\#\{1\leqslant j\leqslant n:a_{j-1}\not\equiv a_j\mod 2\}\\
\mathbf m_{w_{\epsilon_1+\epsilon_{i+1}}}^{w_\gamma}&=\|\gamma\|-\frac12a_i+\frac12\#\{i<j\leqslant n:a_{j-1}\not\equiv a_j\mod 2\}\\
\mathbf m_{w_{\epsilon_1-\epsilon_n}}^{w_\gamma}&=\|\gamma\|-b\\
\mathbf m_{w_{\epsilon_1-\epsilon_i}}^{w_\gamma}&=\|\gamma\|-\frac12a_i-\frac12\#\{i<j\leqslant n:a_{j-1}\not\equiv a_j\mod 2\}\\
\mathbf m_{w_{2\epsilon_2}}^{w_\gamma}&=\frac12\|\gamma\|-\frac14\#\{1\leqslant j\leqslant n:a_{j-1}\not\equiv a_j\mod 2\}\\
\mathbf m_{w_{\epsilon_2+\epsilon_3}}^{w_\gamma}&=\frac12\|\gamma\|-\Big\lceil\frac{a_1}2\Big\rceil+\frac14\#\{1\leqslant j\leqslant n:a_{j-1}\not\equiv a_j\mod 2\}\\
\mathbf m_{w_{-\epsilon_1+\epsilon_2}}^{w_\gamma}&=\frac12\|\gamma\|-\Big\lfloor\frac{a_1}2\Big\rfloor-\frac14\#\{1\leqslant j\leqslant n:a_{j-1}\not\equiv a_j\mod 2\}
\end{align*}
where we set $a_0,a_n\colonequals0$.
\end{thm}
\begin{proof}
Combining Lemma~\ref{prop:1+sigma} and Lemma~\ref{prop:1-sigma}, with notation as in the lemmas, we obtain
\[
[\cO_{\widetilde U}]-t_\gamma[\cO_{\widetilde U}]=\frac12\sum_{j=1}^{n-1}a_j(1+\sgn)[\cO_{\bP_{\alpha_j}}\!(-x_j)]+b[\cO_{\bP_\beta}(-y)]+\frac12\|\gamma\|[\C_y]+\frac14\sum_{j=1}^{n}(1-\sgn^{a_j-a_{j-1}})[\C_{x_j}].\]
Now, note that the formulae in Lemma~\ref{lem:change-of-basis} greatly simplify modulo $\ker(1-\sgn)$:
\begin{align*}
[\C_{x_1}]&\equiv[\cO_{\cB_e}]-[\cO_{\bP_{\alpha_1}}\!(-x_1)]\\
[\C_{x_j}]&\equiv[\cO_{\cB_e}]-\sum_{k=1}^{j-1}d_k[\cO_{\bP_{\alpha_k}}\!(-x_{k+1})]\text{ for }2\leqslant j\leqslant n.
\end{align*}
Thus:
\begin{align*}
[\cO_{\widetilde U}]-t_\gamma[\cO_{\widetilde U}]=&\ \frac12\sum_{j=1}^{n-1}a_j(1+\sgn)[\cO_{\bP_{\alpha_j}}\!(-x_j)]+b[\cO_{\bP_\beta}(-y)]+\frac14(1-\sgn^{a_1})\big([\cO_{\cB_e}]-[\cO_{\bP_{\alpha_1}}\!(-x_1)]\big)\\
&+\frac14\sum_{j=2}^{n}(1-\sgn^{a_j-a_{j-1}})\big([\cO_{\cB_e}]-\sum_{k=1}^{j-1}d_k[\cO_{\bP_{\alpha_k}}\!(-x_{k+1})]\big)\\
&+\frac12\|\gamma\|\big((1+\sgn)[\cO_{\cB_e}]-\sum_{j=1}^{n-1}d_j(1+\sgn)[\cO_{\bP_{\alpha_j}}\!(-x_j)]-2[\cO_{\bP_\beta}(-y)]\big)\\
=&\ \frac12N(1+\sgn)[\cO_{\cB_e}]+\frac14\sum_{j=1}^n(1-\sgn^{a_j-a_{j-1}})[\cO_{\cB_e}]\\
&-\frac14(1-\sgn^{a_1})[\cO_{\bP_{\alpha_1}}\!(-x_1)]+\frac12\sum_{j=1}^{n-1}\!(a_j-d_j\|\gamma\|)(1+\sgn)[\cO_{\bP_{\alpha_j}}\!(-x_j)]\\
&-\frac14\sum_{1\leqslant k<j\leqslant n}(1-\sgn^{a_j-a_{j-1}})d_k[\cO_{\bP_{\alpha_k}}\!(-x_{k+1})]+(b-\|\gamma\|)[\cO_{\bP_\beta}(-y)].
\end{align*}
Since this equals $-\sum_{w_\nu\leqslant w_\gamma,\nu\ne1}\mathbf m_{w_\nu}^{w_\gamma}C_\nu$, we obtain the formulae above, where $d_k$ was defined in \eqref{dj-defn}.
\end{proof}
\begin{remark}
Although the expressions in the theorem are not obviously integers, they are. For example, $\mathbf m_{w_{\epsilon_1+\epsilon_{i+1}}}^{w_\gamma}$ are always integers since $\#\{i<j\leqslant n:a_{j-1}\not\equiv a_j\mod 2\}$ is the number of sign changes in the sequence $a_i,a_{i+1},\dots,a_{n-1},0$, which matches the parity of $a_i$. Moreover, $\|\gamma\|=a_1^2+\cdots+a_{n-1}^2+2b^2-(a_1a_2+\cdots+a_{n-2}a_{n-1}+2a_{n-1}b)$ is always an integer.
\end{remark}

For $w\in c^0_\subreg$ with $\mu(w)=i\in\widehat S$, we calculate $\mathbf m^{w_\gamma}_w-\mathbf m_w^{w_{s_i\gamma}}$, which will be applied to computations of characters of representations of $\widehat\fg$:
\begin{cor}\label{cor:diff_m_expl_B} Let $G$ be of type $\text{B}_n$ where $n\geqslant3$. For $\gamma=a_1\alpha^\vee+\cdots+a_{n-1}\alpha_{n-1}^\vee+b\beta^\vee\in Q^\vee$,
\begin{align*}
    \mathbf m_{w_{\epsilon_1+\epsilon_2}}^{w_\gamma}-\mathbf m_{w_{\epsilon_1+\epsilon_2}}^{w_{s_0\gamma}}&=\frac12(a_1-1)-\frac12(-1)^{a_1}\mathbf 1_{a_2\not\equiv0\bmod 2}\\
    \mathbf m_{w_{\epsilon_1+\epsilon_{i+1}}}^{w_\gamma}-\mathbf m_{w_{\epsilon_1+\epsilon_{i+1}}}^{w_{s_i\gamma}}&=\begin{cases}-a_i+\frac12(a_{i-1}+a_{i+1})+\frac12(-1)^{a_i-a_{i-1}}\mathbf 1_{a_{i-1}\not\equiv a_{i+1}\bmod2}&1<i<n-1\\
    -a_{n-1}+\frac12(a_{n-2}+b)+\frac12(-1)^{a_{n-1}}\mathbf 1_{a_{n-2}\not\equiv b\bmod2}&i=n-1
    \end{cases}\\
    \mathbf m_{w_{\epsilon_1-\epsilon_i}}^{w_\gamma}-\mathbf m_{w_{\epsilon_1-\epsilon_i}}^{w_{s_i\gamma}}&=\begin{cases}-a_i+\frac12(a_{i-1}+a_{i+1})-\frac12(-1)^{a_i-a_{i-1}}\mathbf 1_{a_{i-1}\not\equiv a_{i+1}\bmod2}&i<n\\
    -2b&i=n
    \end{cases}\\
    \mathbf m_{w_{2\epsilon_2}}^{w_\gamma}-\mathbf m_{w_{2\epsilon_2}}^{w_{s_0\gamma}}&=\frac12(a_2-1)+\frac12(-1)^{a_1}\mathbf 1_{a_1\not\equiv0\bmod2}\\
    \mathbf m_{w_{\epsilon_2+\epsilon_3}}^{w_\gamma}-\mathbf m_{w_{\epsilon_2+\epsilon_3}}^{w_{s_1\gamma}}&=\Big\lceil\frac{a_2-a_1}2\Big\rceil-\Big\lceil\frac{a_1}2\Big\rceil\\
    \mathbf m_{w_{-\epsilon_1+\epsilon_2}}^{w_\gamma}-\mathbf m_{w_{-\epsilon_1+\epsilon_2}}^{w_{s_1\gamma}}&=\Big\lfloor\frac{a_2-a_1}2\Big\rfloor-\Big\lfloor\frac{a_1}2\Big\rfloor.
    \end{align*}
\end{cor}
\begin{proof}
For example, let us compute $\mathbf m_{w_{\epsilon_1+\epsilon_{i+1}}}^{w_\gamma}-\mathbf m_{w_{\epsilon_1+\epsilon_{i+1}}}^{w_{s_i\gamma}}$ for $1<i<n-1$; the rest of the computation is similar. First of all:
\begin{align*}
s_i\gamma&=a_1\alpha_1^\vee+\cdots+a_{i-2}\alpha_{i-2}^\vee+a_{i-1}(\alpha_{i-1}^\vee+\alpha_i^\vee)-a_i\alpha_i^\vee+a_{i+1}(\alpha_i^\vee+\alpha_{i+1}^\vee)+a_{i+2}\alpha_{i+2}^\vee+\cdots+b\beta^\vee\\
&=a_1\alpha_1^\vee+\cdots+a_{i-1}\alpha_{i-1}^\vee+(a_{i-1}-a_i+a_{i+1})\alpha_i^\vee+a_{i+1}\alpha_{i+1}^\vee+\cdots+b\beta^\vee.
\end{align*}
Note that since $|\gamma|=|s_i\gamma|$, we have:
\[
\mathbf m_{w_{\epsilon_1+\epsilon_{i+1}}}^{w_\gamma}-\mathbf m_{w_{\epsilon_1+\epsilon_{i+1}}}^{w_{s_i\gamma}}=-a_i+\frac12(a_{i-1}+a_{i+1})+\frac12(N_1-N_2),
\]
where $N_1$ is the number of parity changes in $a_i,a_{i+1},\dots,a_n,0$, while $N_2$ is the number of parity changes in $a_{i-1}-a_i+a_{i+1},a_{i+1},\dots,a_n,0$. Thus, simply:
\begin{align*}
    N_1-N_2&=\mathbf 1_{a_i\not\equiv a_{i+1}\bmod 2}-\mathbf 1_{a_{i-1}\not\equiv a_i\bmod 2}=(-1)^{a_i-a_{i+1}}\mathbf 1_{a_{i-1}\not\equiv a_{i+1}\bmod 2}.\qedhere
\end{align*}
\end{proof}

\subsection{When $G$ is of type C} 
Let $G$ be of type $\text{C}_n$ where $n\geqslant2$.
As in type B, we observe
\begin{align*}
[\cO_{\widetilde U}]-t_\gamma[\cO_{\widetilde U}]&=\sum_{i=0}^{a_1-1}[\cO_{\bP_{\alpha_1}}(i\alpha_1^\vee+a_2\alpha_2^\vee+\cdots+b\beta^\vee)]+\sum_{i=0}^{a_2-1}[\cO_{\bP_{\alpha_2}}(i\alpha_2^\vee+\cdots+b\beta^\vee)]+\cdots+\sum_{i=0}^{b-1}[\cO_{\bP_\beta}(i\beta^\vee)]\\
&=\sum_{i=0}^{a_1-1}[\cO_{\bP_{\alpha_1}}(ix_1+(i-a_2)x_2)]+\sum_{i=0}^{a_2-1}[\cO_{\bP_{\alpha_2}}(ix_2+(i-a_3)x_3)]+\cdots+\sum_{i=0}^{b-1}[\cO_{\bP_\beta}(-i,i)].
\end{align*}
For the computation, we note that
\[
    [\underline\C_{x_j}\langle k\rangle]=[\xi_k\otimes\cO_{\bP_\beta}]-[\xi_{k-1/2}\otimes\cO_{\bP_\beta}(1,0)]+\sum_{i=j}^{n-1}\big([\cO_{\bP_{\alpha_i}}(kx_i-(k+1)x_{i+1})]-[\cO_{\bP_{\alpha_i}}((k-1)x_i-kx_{i+1})]\big),
    \]
so
\begin{align}
[\cO_{\bP_{\alpha_j}}(Mx_j+Nx_{j+1})]=&\ [\cO_{\bP_{\alpha_j}}(-(N+1)x_j+Nx_{j+1})]+\sum_{k=-M}^N\big([\xi_k\otimes\cO_{\bP_\beta}]-[\xi_{k-1/2}\otimes\cO_{\bP_\beta}(1,0)]\big)\nonumber\\&+\sum_{i=j}^{n-1}\big([\cO_{\bP_{\alpha_i}}(Mx_i-(M+1)x_{i+1})]-[\cO_{\bP_{\alpha_i}}(-(N+1)x_i+Nx_{i+1})]\big).\label{eq:O(M+N)-formula}
\end{align}

\begin{thm}\label{thm_KL_type_C}
Let $G$ be of type $\text{C}_n$ where $n\ge2$, and let $\gamma=\sum_{i=1}^{n-1}a_i\alpha^\vee_i+b\beta^\vee\in Q^\vee$. For $1\leqslant i\leqslant n-1$,
\begin{align*}
\mathbf m_{w_{\epsilon_1}}^{w_\gamma}&=\Big\lceil\frac12\sum_{j=0}^{n-1}|a_j-a_{j+1}|\Big
\rceil\\
    \mathbf m_{w_{\epsilon_1+\epsilon_2}}^{w_\gamma}&=\Big\lfloor\frac12\sum_{j=0}^{n-1}|a_j-a_{j+1}|\Big\rfloor\\
    \sum_{k\geqslant0}\mathbf m^{w_\gamma}_{w_{(k+1)\epsilon_{i+1}}}\Xi^k+\sum_{k<0}\mathbf m^{w_\gamma}_{w_{k\epsilon_i}}\Xi^k&=\sum_{j=1}^i\frac{\Xi^{a_{j}-a_{j-1}}-1}{\Xi-1}-\sum_{j=0}^{n-1}\bigg(\frac{\Xi^{(a_j-a_{j+1})/2}-\Xi^{(a_{j+1}-a_{j})/2}}{\Xi^{1/2}-\Xi^{-1/2}}\bigg)^2\\
    \sum_{k\geqslant2} \mathbf m^{w_\gamma}_{w_{k\epsilon_1}}(\Xi^{k-1}+\Xi^{1-k})&=\sum_{j=0}^{n-1}\bigg(\Big(\frac{\Xi^{(a_j-a_{j+1})/2}-\Xi^{(a_{j+1}-a_{j})/2}}{\Xi^{1/2}-\Xi^{-1/2}}\Big)^2-|a_j-a_{j+1}|\bigg)\\
\sum_{k\geqslant1}\mathbf m^{w_\gamma}_{w_{-k\epsilon_n}}\Xi^{k-1/2}&=\sum_{j=0}^{n}\frac{\Xi^{a_j-a_{j+1}-1/2}+\Xi^{a_{j+1}-a_j+1/2}-\Xi^{1/2}-\Xi^{-1/2}}{(\Xi^{1/2}-\Xi^{-1/2})^2},
\end{align*}
where we let $a_0\colonequals0$, $a_n\colonequals b$, and $a_{n+1}\colonequals0$.

\end{thm}

\begin{remark}
    Note that Theorem~\ref{thm_KL_type_C} gives a formula for $\mathbf m^{w_\gamma}_w$ for $w\in c^0_\subreg$, since one can explicitly compute the coefficients of the polynomials on the right-hand side. Unlike in other types the $\mathbf m^{w_\gamma}_w$ are \emph{not} quasi-polynomial functions of $\gamma$, as explained in the Appendix \ref{appendix_C}.
\end{remark}

\begin{proof}
    First, note that by Lemma~\ref{sym^n-calc},
    \begin{align*}
    \sum_{i=0}^{b-1}[\cO_{\bP_\beta}(-i,i)]&=\Big(\sum_{i=0}^{b-1}[\cO_{\bP_\beta}(b-1-i,i)]\Big)\cdot[\cO_{\bP_\beta}(1-b,0)]\\
    &=[\sgn^{b-1}\otimes\Sym^{b-1}(\xi_{1/2})]\cdot \big([\Sym^{b-1}(\xi_{1/2})\otimes\cO_{\bP_\beta}]-[\Sym^{b-2}(\xi_{1/2})\otimes\cO_{\bP_\beta}(1,0)]\big).
    \end{align*}
On the other hand, by \eqref{eq:O(M+N)-formula}, for each $j$, the sum $\sum_{m=0}^{a_j-1}[\cO_{\bP_{\alpha_j}}\!(mx_j+(m-a_{j+1})x_{j+1})]$ equals
\begin{align*}
&\sum_{m=0}^{a_1-1}\bigg(\cO_{\bP_{\alpha_j}}(-(m-a_{j+1}+1)x_j+(m-a_{j+1})x_{j+1})\\
&+\sum_{i=j}^{n-1}\big(\cO_{\bP_{\alpha_i}}\!(mx_i-(m+1)x_{i+1})-\cO_{\bP_{\alpha_i}}(-(m-a_{j+1}+1)x_i+(m-a_{j+1})x_{i+1})\big)\bigg)\\
&+[V_j(\gamma)\otimes\cO_{\bP_\beta}]-[W_j(\gamma)\otimes\cO_{\bP_\beta}(1,0)],
\end{align*}
where $V_j(\gamma)$ and $W_j(\gamma)$ are representations of $\mathrm{Pin}_2$ characterized by having characters with value $0$ on $\mathrm{Pin}_2\backslash\mathrm{Spin}_2$ and taking values
\begin{align*}
    \sum_{m=0}^{a_j-1}\frac{t^{m+1/2}-t^{-m-1/2}+t^{m-a_{j+1}+1/2}-t^{a_{j+1}-m-1/2}}{t^{1/2}-t^{-1/2}}=&\ \frac{t^{1/2}+t^{1/2-a_{j+1}}}{t^{1/2}-t^{-1/2}}\sum_{m=0}^{a_j-1}t^m-\frac{t^{-1/2}+t^{a_{j+1}-1/2}}{t^{1/2}-t^{-1/2}}\sum_{m=0}^{a_j-1}t^{-m}\\
    =&\ \frac{t^{1/2}+t^{1/2-a_{j+1}}}{t^{1/2}-t^{-1/2}}\cdot\frac{t^{a_j-1/2}-t^{-1/2}}{t^{1/2}-t^{-1/2}}\\
    &-\frac{t^{-1/2}+t^{a_{j+1}-1/2}}{t^{1/2}-t^{-1/2}}\cdot\frac{t^{1/2}-t^{1/2-a_j}}{t^{1/2}-t^{-1/2}}\\
    =&\ \frac{t^{a_j-a_{j+1}}+t^{a_{j+1}-a_j}+t^{a_j}+t^{-a_j}-t^{a_{j+1}}-t^{-a_{j+1}}-2}{(t^{1/2}-t^{-1/2})^2}
\end{align*}
and
\begin{align*}
    \sum_{m=0}^{a_j-1}\frac{t^{m+1}-t^{-m-1}+t^{m-a_{j+1}}-t^{a_{j+1}-m}}{t^{1/2}-t^{-1/2}}=&\ \frac{t+t^{-a_{j+1}}}{t^{1/2}-t^{-1/2}}\sum_{m=0}^{a_j-1}t^m-\frac{t^{-1}+t^{a_{j+1}}}{t^{1/2}-t^{-1/2}}\sum_{m=0}^{a_j-1}t^{-m}\\
    =&\ \frac{t+t^{-a_{j+1}}}{t^{1/2}-t^{-1/2}}\cdot\frac{t^{a_j-1/2}-t^{-1/2}}{t^{1/2}-t^{-1/2}}-\frac{t^{-1}+t^{a_{j+1}}}{t^{1/2}-t^{-1/2}}\cdot\frac{t^{1/2}-t^{1/2-a_j}}{t^{1/2}-t^{-1/2}}\\
    =&\ \frac{t^{a_j-a_{j+1}-1/2}+t^{a_{j+1}-a_j+1/2}+t^{a_j+1/2}+t^{-a_j-1/2}}{(t^{1/2}-t^{-1/2})^2}\\
    &+\frac{-t^{a_{j+1}+1/2}-t^{-a_{j+1}-1/2}-(t^{1/2}+t^{-1/2})}{(t^{1/2}-t^{-1/2})^2},
\end{align*}
respectively, on $t^{1/2}\in\mathrm{Spin}_2$. In particular, their dimension is both $2a_j(a_j-a_{j+1})$. Thus, in total, $[\cO_{\widetilde U}]-[\cO_{\widetilde U}(\gamma)]$ equals
\begin{align}
    &\sum_{j=1}^{n-1}\sum_{m=0}^{a_j-1}\bigg([\cO_{\bP_{\alpha_j}}(-(m-a_{j+1}+1)x_j+(m-a_{j+1})x_{j+1})]\nonumber\\
&+\sum_{i=j}^{n-1}\big([\cO_{\bP_{\alpha_i}}\!(mx_i-(m+1)x_{i+1})]-[\cO_{\bP_{\alpha_i}}(-(m-a_{j+1}+1)x_i+(m-a_{j+1})x_{i+1})]\big)\bigg)\label{eq:typec-olambda-formula}\\
&+[V(\gamma)\otimes\cO_{\bP_\beta}]-[W(\gamma)\otimes\cO_{\bP_\beta}(1,0)],\nonumber
\end{align}
where $V(\gamma)=\sum_{j=1}^{n-1}V_j(\gamma)+\Sym^{b-1}(\xi_{1/2})^{\otimes2}$ and $W(\gamma)=\sum_{j=1}^{n-1}W_j(\gamma)+\Sym^{b-1}(\xi_{1/2})\otimes\Sym^{b-2}(\xi_{1/2})$. The characters of $V(\gamma)$ and $W(\gamma)$ at $(t^{1/2},1)\in\widetilde\SO_2$ are (letting $a_0\colonequals0$ and $a_n\colonequals b$):
\begin{align}
    \tr V(\gamma)(t)=&\sum_{j=1}^{n-1}\frac{t^{a_j-a_{j+1}}+t^{a_{j+1}-a_j}}{(t^{1/2}-t^{-1/2})^2}+\frac{t^{a_1}+t^{-a_1}-t^{b}-t^{-b}}{(t^{1/2}-t^{-1/2})^2}-\frac{2(n-1)}{(t^{1/2}-t^{-1/2})^2}+\big(\frac{t^{b/2}-t^{-b/2}}{t^{1/2}-t^{-1/2}}\big)^2,\nonumber\\
    =&\sum_{j=0}^{n-1}\frac{t^{a_j-a_{j+1}}+t^{a_{j+1}-a_j}-2}{(t^{1/2}-t^{-1/2})^2}=\sum_{j=0}^{n-1}\bigg(\frac{t^{(a_j-a_{j+1})/2}-t^{(a_{j+1}-a_{j})/2}}{t^{1/2}-t^{-1/2}}\bigg)^2,\label{eq:vlambda-char}
    \end{align}
    and
    \begin{align}
    \tr W(\gamma)(t)=&\sum_{j=1}^{n-1}\frac{t^{a_j-a_{j+1}-1/2}+t^{a_{j+1}-a_j+1/2}}{(t^{1/2}-t^{-1/2})^2}+\frac{t^{a_1+1/2}+t^{-a_1-1/2}-t^{b+1/2}-t^{-b-1/2}}{(t^{1/2}-t^{-1/2})^2}\nonumber\\
    &-\frac{(n-1)(t^{1/2}+t^{-1/2})}{(t^{1/2}-t^{-1/2})^2}+\frac{t^{b/2}-t^{-b/2}}{t^{1/2}-t^{-1/2}}\cdot \frac{t^{(b-1)/2}-t^{(1-b)/2}}{t^{1/2}-t^{-1/2}}\nonumber\\
    =&\sum_{j=1}^{n-1}\frac{t^{a_j-a_{j+1}-1/2}+t^{a_{j+1}-a_j+1/2}}{(t^{1/2}-t^{-1/2})^2}+\frac{t^{a_1+1/2}+t^{-a_1-1/2}-t^{b+1/2}-t^{-b-1/2}}{(t^{1/2}-t^{-1/2})^2}\label{eq:wlambda-char}\\
    &-\frac{(n-1)(t^{1/2}+t^{-1/2})}{(t^{1/2}-t^{-1/2})^2}+\frac{t^{b/2}-t^{-b/2}}{t^{1/2}-t^{-1/2}}\cdot \frac{t^{(b-1)/2}-t^{(1-b)/2}}{t^{1/2}-t^{-1/2}}\nonumber\\
    =&\sum_{j=0}^{n}\frac{t^{a_j-a_{j+1}-1/2}+t^{a_{j+1}-a_j+1/2}-t^{1/2}-t^{-1/2}}{(t^{1/2}-t^{-1/2})^2},\nonumber
\end{align}
where $a_{n+1}\colonequals0$, and their character values on any $g=(t^{1/2},-1)\in \tO_2\backslash\SO_2$ are:
\begin{align*}
    \tr V(\gamma)(g)&=\begin{cases}
    1&b\equiv1\pmod2\\
    0&b\equiv1\pmod 2.
    \end{cases}\\
    \tr W(\gamma)(g)&=0.
\end{align*}
Recalling that $[\cO_{\bP_\beta}]=[\cO_{\cB_e}]-\sum_{i=1}^{n-1}[\cO_{\bP_{\alpha_i}}\!(-x_{i+1})]$ and from Proposition~\ref{prop:canonical-basis-c} that
\[
[\cO_{\bP_{\alpha_i}}(kx_i-(k+1)x_{i+1})]=\begin{cases}
C_{w_{(k+1)\epsilon_{i+1}}}&k\geqslant0\\
C_{w_{k\epsilon_i}}&k<0,
\end{cases}
\]
from \eqref{eq:typec-olambda-formula} we obtain:
\begin{align*}
    \sum_{k\geqslant0}\mathbf m^{w_\gamma}_{w_{(k+1)\epsilon_{i+1}}}\Xi^k+\sum_{k<0}\mathbf m^{w_\gamma}_{w_{k\epsilon_i}}\Xi^k=&\sum_{m=0}^{a_i-1}\Xi^{-(m-a_{i+1}+1)}+\sum_{j=1}^i\sum_{m=0}^{a_j-1}\big(\Xi^m-\Xi^{-(m-a_{j+1}+1)}\big)\\
    &-\sum_{j=0}^{n-1}\bigg(\frac{\Xi^{(a_j-a_{j+1})/2}-\Xi^{(a_{j+1}-a_{j})/2}}{\Xi^{1/2}-\Xi^{-1/2}}\bigg)^2\\
    =&\sum_{j=1}^i\frac{\Xi^{a_{j}-a_{j-1}}-1}{\Xi-1}-\sum_{j=0}^{n-1}\bigg(\frac{\Xi^{(a_j-a_{j+1})/2}-\Xi^{(a_{j+1}-a_{j})/2}}{\Xi^{1/2}-\Xi^{-1/2}}\bigg)^2.
\end{align*}
Next, since $\mathbf m_{w_{k\epsilon_1}}^{w_\gamma}=-[\xi_{k-1}\otimes\cO_{\cB_e}]$, equations \eqref{eq:vlambda-char} and \eqref{eq:wlambda-char} gives:
\begin{align*}
    \sum_{k\geqslant2} \mathbf m^{w_\gamma}_{w_{k\epsilon_1}}(\Xi^{k-1}+\Xi^{1-k})&=\sum_{j=0}^{n-1}\bigg(\Big(\frac{\Xi^{(a_j-a_{j+1})/2}-\Xi^{(a_{j+1}-a_{j})/2}}{\Xi^{1/2}-\Xi^{-1/2}}\Big)^2-|a_j-a_{j+1}|\bigg)\\
    &=\frac12
\sum_{k\geqslant2}\mathbf m^{w_\gamma}_{w_{-k\epsilon_n}}\Xi^{k-1/2}\\&=\sum_{j=0}^{n-1}\frac{\Xi^{a_j-a_{j+1}-1/2}+\Xi^{a_{j+1}-a_j+1/2}-\Xi^{1/2}-\Xi^{-1/2}}{(\Xi^{1/2}-\Xi^{-1/2})^2}-\frac{\Xi^{\bounded}+\Xi^{-b}}{\Xi^{1/2}-\Xi^{-1/2}}.
\end{align*}
The computation for $\sum_{k\geqslant1}\mathbf m_{w_{-k\epsilon_n}}^{w_\gamma}(\Xi^{k-1/2}+\Xi^{1/2-k})$ is the same, since $C_{w_{-k\epsilon_n}}=[\xi_{k-1/2}\otimes\cO_{\bP_\beta}(1,0)]$.

Finally, note that $\dim V(\gamma)_1+\dim V(\gamma)_{\sgn}$, the sum of the $1$ and $\sgn$-isotropic subspaces of $V(\gamma)$, is exactly the constant term of $\tr V(\gamma)(t)$ above, which is simply $\sum_{j=0}^{n-1}|a_j-a_{j+1}|$. Thus, the multiplicity of $1$ and $\sgn$ in $V(\gamma)$ is $\lceil\frac12\sum_{j=0}^{n-1}|a_j-a_{j+1}|\rceil$ and $\lceil\frac12\sum_{j=0}^{n-1}|a_j-a_{j+1}|\rceil$, respectively. Thus,
\begin{align}
    \mathbf m_{w_{\epsilon_1}}^{w_\gamma}&=\Big\lceil\frac12\sum_{j=0}^{n-1}|a_j-a_{j+1}|\Big\rceil\label{eq:type-c-calculation-(1)}\\
    &=\frac12\sum_{j=0}^{n-1}|a_j-a_{j+1}|+\frac12\mathbf1_{b\not\equiv0\bmod2}\nonumber\\
    \mathbf m_{w_{\epsilon_1+\epsilon_2}}^{w_\gamma}&=\Big\lfloor\frac12\sum_{j=0}^{n-1}|a_j-a_{j+1}|\Big\rfloor\label{eq:type-c-calculation-(2)}\\
    &=\frac12\sum_{j=0}^{n-1}|a_j-a_{j+1}|-\frac12\mathbf1_{b\not\equiv0\bmod2}.\nonumber\qedhere
\end{align}

\end{proof}

As a consequence, we may compute:
\begin{cor}\label{cor:diff_m_expl_C} Let $\gamma=\sum_{i=1}^{n-1}a_i\alpha_i^\vee+b\beta^\vee$. For $1\leqslant i\leqslant n-1$,
    \begin{align*}
        \mathbf m^{w_\gamma}_{w_{\epsilon_1}}-\mathbf m^{w_{s_0\gamma}}_{w_{\epsilon_1}}&=\frac12(|a_1|-|a_1-1|-(-1)^{b})\\
        \mathbf m^{w_\gamma}_{w_{\epsilon_1+\epsilon_2}}-\mathbf m^{w_{s_0\gamma}}_{w_{\epsilon_1+\epsilon_2}}&=\frac12(|a_1|-|a_1-1|+(-1)^{b})\\
        \sum_{k\geqslant0}(\mathbf m^{w_\gamma}_{w_{(k+1)\epsilon_{i+1}}}-\mathbf m^{w_{s_i\gamma}}_{w_{(k+1)\epsilon_{i+1}}})\Xi^k+\sum_{k<0}(\mathbf m^{w_\gamma}_{w_{k\epsilon_i}}-\mathbf m^{w_{s_i\gamma}}_{w_{k\epsilon_i}})\Xi^k&=\frac{\Xi^{a_i-a_{i-1}}-\Xi^{a_{i+1}-a_i}}{\Xi-1}\\
    \sum_{k\geqslant2} (\mathbf m^{w_\gamma}_{w_{k\epsilon_1}}-\mathbf m^{w_{s_0\gamma}}_{w_{k\epsilon_1}})(\Xi^{k-1}+\Xi^{1-k})&=\frac{\Xi^{a_1-1/2}-\Xi^{1/2-a_1}}{\Xi^{1/2}-\Xi^{-1/2}}-(|a_1|-|a_1-1|)\\
\sum_{k\geqslant1}(\mathbf m^{w_\gamma}_{w_{-k\epsilon_n}}-\mathbf m^{w_{s_n\gamma}}_{w_{-k\epsilon_n}})(\Xi^{k-1/2}+\Xi^{1/2-k})&=\frac{\Xi^{b-a_{n-1}}-\Xi^{a_{n-1}-b}}{\Xi^{1/2}-\Xi^{-1/2}}\cdot \frac{\Xi^{a_{n-1}-1/2}-\Xi^{1/2-a_{n-1}}}{\Xi^{1/2}-\Xi^{-1/2}}.
    \end{align*}
\end{cor}
\begin{proof}
Note that:
\[
s_0\gamma=(1-a_1)\alpha_1^\vee+(1+a_2-2a_1)\alpha_2^\vee+(1+a_3-2a_1)\alpha_3^\vee+\cdots+(1+b-2a_1)\beta^\vee.
\]
Thus, visibly, $|a_j-a_{j+1}|$ is invariant under $s_0$ for $j>0$. Thus, using \eqref{eq:type-c-calculation-(1)} and \eqref{eq:type-c-calculation-(2)} we obtain:
\begin{align*}
    \mathbf m^{w_\gamma}_{w_{\epsilon_1}}-\mathbf m^{w_{s_0\gamma}}_{w_{\epsilon_1}}&=\frac12|a_1|-\frac12|1-a_1|+\frac12(\mathbf 1_{b\not\equiv0\bmod2}-\mathbf1_{1+b-2a_1\not\equiv0}\bmod2)\\
    &=\frac12(|a_1|-|a_1-1|-(-1)^{b})\\
    \mathbf m^{w_\gamma}_{w_{\epsilon_1+\epsilon_2}}-\mathbf m^{w_{s_0\gamma}}_{w_{\epsilon_1+\epsilon_2}}&=\frac12(|a_1|-|a_1-1|+(-1)^{b}).
\end{align*}
Similarly, most terms in the expression for $\sum_{k\geqslant2}\mathbf m^{w_\gamma}_{w_{k\epsilon_1}}(\Xi^{k-1}+\Xi^{1-k})$ is invariant under $s_0$, so
\begin{align*}
    \sum_{k\geqslant2} (\mathbf m^{w_\gamma}_{w_{k\epsilon_1}}-\mathbf m^{w_{s_0\gamma}}_{w_{k\epsilon_1}})(\Xi^{k-1}+\Xi^{1-k})&=\bigg(\Big(\frac{\Xi^{a_1/2}-\Xi^{-a_1/2}}{\Xi^{1/2}-\Xi^{-1/2}}\Big)^2-|a_1|\bigg)-\bigg(\Big(\frac{\Xi^{(1-a_1)/2}-\Xi^{-(1-a_1)/2}}{\Xi^{1/2}-\Xi^{-1/2}}\Big)^2-|1-a_1|\bigg)\\
    &=\frac{\Xi^{a_1-1/2}-\Xi^{1/2-a_1}}{\Xi^{1/2}-\Xi^{-1/2}}-\big(|a_1|-|a_1-1|\big).
\end{align*}

Similarly, note that
\[
s_i\gamma=a_1\alpha_1^\vee+\cdots+a_{i-1}\alpha_{i-1}^\vee+(a_{i-1}-a_i+a_{i+1})\alpha_{i}^\vee+a_{i+1}\alpha_{i+1}^\vee+\cdots+b\beta^\vee,
\]
so everything except for $(\Xi^{a_i-a_{i-1}}-1)/(\Xi-1)$ in the expression for
\(\sum_{k\geqslant0}\mathbf m_{w_{(k+1)\epsilon_{i+1}})}^{w_\gamma}\Xi^k+\sum_{k<0}\mathbf m^{w_\gamma}_{k\epsilon_i}\Xi^k\) in Theorem~\ref{thm_KL_type_C}
is invariant under $s_i$. Thus,
\begin{align*}
\sum_{k\geqslant0}(\mathbf m^{w_\gamma}_{w_{(k+1)\epsilon_{i+1}}}-\mathbf m^{w_{s_i\gamma}}_{w_{(k+1)\epsilon_{i+1}}})\Xi^k+\sum_{k<0}(\mathbf m^{w_\gamma}_{w_{k\epsilon_i}}-\mathbf m^{w_{s_i\gamma}}_{w_{k\epsilon_i}})\Xi^k&=\frac{\Xi^{a_{i}-a_{i-1}}-1}{\Xi-1}-\frac{\Xi^{a_{i+1}-a_i}-1}{\Xi-1}\\
&=\frac{\Xi^{a_i-a_{i-1}}-\Xi^{a_{i+1}-a_i}}{\Xi-1}.
\end{align*}
Finally,
\[
s_n\gamma=a_1\alpha_1^\vee+\cdots+a_{n-1}\alpha_{n-1}^\vee+(2a_{n-1}-b)\beta^\vee,
\]
so again all terms in the expression for $\sum_{k\geqslant1}\mathbf m^{w_\gamma}_{w_{-k\epsilon_n}}\Xi^{k-1/2}$ in Theorem~\ref{thm_KL_type_C} except for
\(\frac{\Xi^{b-1/2}+\Xi^{1/2-b}-\Xi^{1/2}-\Xi^{-1/2}}{(\Xi^{1/2}-\Xi^{-1/2})^2}
\) are $s_n$-invariant, and hence
\begin{align*}
    \sum_{k\geqslant1}(\mathbf m^{w_\gamma}_{w_{-k\epsilon_n}}-\mathbf m^{w_{s_n\gamma}}_{w_{-k\epsilon_n}})(\Xi^{k-1/2}+\Xi^{1/2-k})=&\ \frac{\Xi^{b-1/2}+\Xi^{1/2-b}-\Xi^{1/2}-\Xi^{-1/2}}{(\Xi^{1/2}-\Xi^{-1/2})^2}\\&-\frac{\Xi^{2a_{n-1}-b-1/2}+\Xi^{1/2-2a_{n-1}+b}-\Xi^{1/2}-\Xi^{-1/2}}{(\Xi^{1/2}-\Xi^{-1/2})^2}\\
    =&\ \frac{\Xi^{b-a_{n-1}}-\Xi^{a_{n-1}-b}}{\Xi^{1/2}-\Xi^{-1/2}}\cdot \frac{\Xi^{a_{n-1}-1/2}-\Xi^{1/2-a_{n-1}}}{\Xi^{1/2}-\Xi^{-1/2}}.\qedhere
\end{align*}
\end{proof}

\subsection{When $G$ is of type F\textsubscript{4}}

There are the identities:
\begin{align*}
[\cO_{\widetilde U}(\epsilon_2+\epsilon_4)]&=[\cO_{\widetilde U}(\epsilon_2+\epsilon_3)]+[\cO_{\bP_{\alpha_2}}(-y_2)]\\
[\cO_{\widetilde U}(\epsilon_2-\epsilon_4)]&=[\cO_{\widetilde U}(\epsilon_2+\epsilon_4)]+[\cO_{\bP_{\alpha_3}}\!(-x_2)]\\
[\cO_{\widetilde U}(\epsilon_2-\epsilon_3)]&=[\cO_{\widetilde U}(\epsilon_2-\epsilon_4)]+[\cO_{\bP_{\alpha_2}}\!(-x_1)],
\end{align*}
which allows us to express $[\cO_{\bP_{\alpha_1}}]$ in terms of the canonical basis:
\[
[\cO_{\bP_{\alpha_1}}]=[\cO_{\cB_e}]-[\cO_{\bP_{\alpha_1}}(-y_1)]-\big(2[\cO_{\bP_{\alpha_2}}\!(-x_1)]+[\cO_{\bP_{\alpha_2}}(-y_1)]\big)-2[\cO_{\bP_{\alpha_3}}\!(-x_2)]-[\cO_{\bP_{\alpha_4}}\!(-x_3)].
\]
Thus:
\begin{lemma}\label{f4-change-of-basis}For $G$ of type F\textsubscript{4}, in $K^{\Cent_e}(\cB_e)$,
\begin{align*}
[\C_{y_1}]&=[\cO_{\cB_e}]-2[\cO_{\bP_{\alpha_1}}(-y_1)]-\big(2[\cO_{\bP_{\alpha_2}}\!(-x_1)]+[\cO_{\bP_{\alpha_2}}(-y_2)]\big)-2[\cO_{\bP_{\alpha_3}}\!(-x_2)]-[\cO_{\bP_{\alpha_4}}\!(-x_3)]\\
[\C_{x_1}]&=[\cO_{\cB_e}]-\big([\cO_{\bP_{\alpha_1}}\!(-x_1)]+[\cO_{\bP_{\alpha_1}}(-y_1)]\big)-\big(2[\cO_{\bP_{\alpha_2}}\!(-x_1)]+[\cO_{\bP_{\alpha_2}}(-y_2)]\big)-2[\cO_{\bP_{\alpha_3}}\!(-x_2)]-[\cO_{\bP_{\alpha_4}}\!(-x_3)]\\
[\C_{y_2}]&=[\cO_{\cB_e}]-\big([\cO_{\bP_{\alpha_1}}\!(-x_1)]+[\cO_{\bP_{\alpha_1}}(-y_1)]\big)-\big([\cO_{\bP_{\alpha_2}}\!(-x_1)]+2[\cO_{\bP_{\alpha_2}}(-y_2)]\big)-2[\cO_{\bP_{\alpha_3}}\!(-x_2)]-[\cO_{\bP_{\alpha_4}}\!(-x_3)]\\
\end{align*}
\end{lemma}

We can express $[\cO_{\widetilde U}]-t_\gamma[\cO_{\widetilde U}]\in K^{G^\vee}\!(\widetilde U)$ as follows, where $\gamma=a_1\alpha_1^\vee+a_2\alpha_2^\vee+a_3\alpha_3^\vee+a_4\alpha_4^\vee$:
\begin{align*}
[\cO_{\widetilde U}]-[\cO_{\widetilde U}(\gamma)]=&\sum_{i=0}^{a_1-1}[\cO_{\bP_{\alpha_1}}(i\alpha_1^\vee)]+\sum_{i=0}^{a_2-1}[\cO_{\bP_{\alpha_2}}\!(a_1\alpha_1^\vee+i\alpha_2^\vee)]\\&+\sum_{i=0}^{a_3-1}[\cO_{\bP_{\alpha_3}}\!(a_1\alpha_1^\vee+a_2\alpha_2^\vee+i\alpha_3^\vee)]+\sum_{i=0}^{a_4-1}[\cO_{\bP_{\alpha_4}}\!(a_1\alpha_1^\vee+a_2\alpha_2^\vee+a_3\alpha_3^\vee+i\alpha_4^\vee)]\\
=&\sum_{i=0}^{a_1-1}[\cO_{\bP_{\alpha_1}}(ix_1+iy_1)]+\sum_{i=0}^{a_2-1}[\cO_{\bP_{\alpha_2}}((a_1+i)x_1+(i-2a_1)y_2)]\\&+\sum_{i=0}^{a_3-1}[\cO_{\bP_{\alpha_3}}((2i-a_2)x_2)]+\sum_{i=0}^{a_4-1}[\cO_{\bP_{\alpha_4}}((2i-a_3)x_3)].
\end{align*}
We calculate both $(1-\sgn)([\cO_{\widetilde U}]-t_\gamma[\cO_{\widetilde U}])$ and $(1+\sgn)([\cO_{\widetilde U}]-t_\gamma[\cO_{\widetilde U}])$:
\begin{lemma}For $G$ of type F\textsubscript{4} and $\gamma=a_1\alpha_1^\vee+a_2\alpha_2^\vee+a_3\alpha_3^\vee+\alpha_4^\vee\in Q^\vee$
\begin{align*}
(1+\sgn)\big([\cO_{\widetilde U}]-t_\gamma[\cO_{\widetilde U}]\big)=&a_1(1+\sgn)[\cO_{\bP_{\alpha_1}}\!(-x_1)]+a_2(1+\sgn)[\cO_{\bP_{\alpha_2}}\!(-x_1)]\\&+2a_3[\cO_{\bP_{\alpha_3}}\!(-x_2)]+2a_4[\cO_{\bP_{\alpha_4}}\!(-x_3)]+\|\gamma\|[\C_{x_2}]
\end{align*}
in $K^{\Cent_e}(\cB_e)$, where $\|\gamma\|=a_1^2+a_2^2+2a_3^2+2a_4^2-a_1a_2-2a_2a_3-2a_3a_4$ is the square-length of $\gamma$ normalized so the short roots have length $1$.
\end{lemma}
\begin{proof} Indeed
\begin{align*}
(1+\sgn)\big([\cO_{\widetilde U}]-t_\gamma[\cO_{\widetilde U}]\big)=&\sum_{i=0}^{a_1-1}(1+\sgn)[\cO_{\bP_{\alpha_1}}\!(-x_1)]+\sum_{i=0}^{a_2-1}(1+\sgn)[\cO_{\bP_{\alpha_2}}\!(-x_2)]\\
&+2\sum_{i=0}^{a_3-1}[\cO_{\bP_{\alpha_3}}\!(-x_2)]+2\sum_{i=0}^{a_4-1}[\cO_{\bP_{\alpha_4}}\!(-x_3)]\\
=&(1+\sgn)a_1[\cO_{\bP_{\alpha_1}}\!(-x_1)]+(1+\sgn)a_2[\cO_{\bP_{\alpha_2}}\!(-x_1)]\\
&+2a_3[\cO_{\bP_{\alpha_3}}\!(-x_2)]+2a_4[\cO_{\bP_{\alpha_4}}\!(-x_3)]+N[\C_{x_2}],
\end{align*}
where
\begin{align*}
N&=\sum_{i=0}^{a_1-1}(2i+1)+\sum_{i=0}^{a_2-1}(2i+1-a_1)+2\sum_{i=0}^{a_3-1}(2i+1-a_2)+2\sum_{i=0}^{a_4-1}(2i+1-a_3)\\
&=a_1^2+a_2^2+2a_3^2+2a_4^2-a_1a_2-2a_2a_3-2a_3a_4\\&=\|\gamma\|.\qedhere
\end{align*}
\end{proof}
Moreover, by Lemma~\ref{f4-change-of-basis},
\[
[\C_{x_2}]=(1+\sgn)[\cO_{\cB_e}]-2(1+\sgn)[\cO_{\bP_{\alpha_1}}\!(-x_1)]-3(1+\sgn)[\cO_{\bP_{\alpha_2}}\!(-x_1)]-4[\cO_{\bP_{\alpha_3}}\!(-x_2)]-2[\cO_{\bP_{\alpha_4}}\!(-x_3)],
\]
so
\begin{align*}
(1+\sgn)([\cO_{\widetilde U}]-t_\gamma[\cO_{\widetilde U}])=&\|\gamma\|(1+\sgn)[\cO_{\cB_e}]+(a_1-2\|\gamma\|)(1+\sgn)[\cO_{\bP_{\alpha_1}}\!(-x_1)]\\&+(a_2-3\|\gamma\|)(1+\sgn)[\cO_{\bP_{\alpha_2}}\!(-x_1)]
+(2a_3-4\|\gamma\|)[\cO_{\bP_{\alpha_3}}\!(-x_2)]\\&+(2a_4-2\|\gamma\|)[\cO_{\bP_{\alpha_4}}\!(-x_3)].
\end{align*}

\begin{lemma} For $G$ of type F\textsubscript{4} and $\gamma=a_1\alpha_1^\vee+a_2\alpha_2^\vee+a_3\alpha_3^\vee+\alpha_4^\vee\in Q^\vee$,
\begin{equation}\label{f4:1-sigma}
(1-\sgn)\big([\cO_{\widetilde U}]-t_\gamma[\cO_{\widetilde U}]\big)=\frac12(1-\sgn^{a_1})[\C_{y_1}]+\frac12(1-\sgn^{a_1+a_2})[\C_{x_1}]+\frac12(1-\sgn^{a_2})[\C_{y_2}].
\end{equation}
\end{lemma}
\begin{proof}
As in the proof of Lemma~\ref{prop:1-sigma}, we have
\begin{align*}
(1-\sgn)\sum_{i=0}^{a_1-1}[\cO_{\bP_{\alpha_1}}(ix_1+iy_1)]=&\ \frac12(1-\sgn^{a_1})[\C_{x_1}]+\frac12(1-\sgn^{a_1})[\C_{y_1}]\\
(1-\sgn)\sum_{i=0}^{a_2-1}[\cO_{\bP_{\alpha_1}}((a_1+i)x_1+(i-2a_1)y_2)]=&\ \frac12(\sgn^{a_1}-\sgn^{a_1+a_2})[\C_{x_1}]\\&+\frac12(\sgn^{-2a_1}-\sgn^{a_2-2a_1})[\C_{y_1}].\qedhere
\end{align*}
\end{proof}
However, we need to rewrite \eqref{f4:1-sigma} in terms of the canonical basis. Now, by Lemma~\ref{f4-change-of-basis}, modulo $\ker(1-\sgn)$, the following hold:
\begin{align*}
[\C_{y_1}]&\equiv[\cO_{\cB_e}]-2[\cO_{\bP_{\alpha_1}}(-y_1)]-[\cO_{\bP_{\alpha_2}}\!(-x_1)]\\
[\C_{x_1}]&\equiv[\cO_{\cB_e}]-[\cO_{\bP_{\alpha_2}}\!(-x_1)]\\
[\C_{y_2}]&\equiv[\cO_{\cB_e}]-[\cO_{\bP_{\alpha_2}}(-y_2)].
\end{align*}
We therefore obtain:
\begin{thm}\label{thm_KL_type_F}
Let $G$ be of type F\textsubscript{4}. Then for $\gamma=a_1\alpha_1^\vee+a_2\alpha_2^\vee+a_3\alpha_3^\vee+a_4\alpha_4^\vee$,
\begin{align*}
\mathbf m_{w_{\epsilon_1+\epsilon_2}}^{w_\gamma}&=\frac12\|\gamma\|+\frac14\big(\mathbf 1_{a_1\not\equiv0\bmod2}+\mathbf 1_{a_1\not\equiv a_2\bmod2}+\mathbf 1_{a_2\not\equiv0\bmod2}\big)\\
\mathbf m_{w_{\epsilon_1+\epsilon_3}}^{w_\gamma}&=\|\gamma\|-\lfloor\frac{a_1}2\rfloor\\
\mathbf m_{w_{\epsilon_1+\epsilon_4}}^{w_\gamma}&=\frac32\|\gamma\|-\frac12a_2+\frac14\big(\mathbf1_{a_1\not\equiv0\bmod2}+\mathbf1_{a_1\not\equiv a_2\bmod2}-\mathbf1_{a_2\not\equiv0\mod 2}\big)\\
\mathbf m_{w_{\epsilon_1-\epsilon_4}}^{w_\gamma}&=2\|\gamma\|-a_3\\
\mathbf m_{w_{\epsilon_2+\epsilon_3}}^{w_\gamma}&=\|\gamma\|-a_4\\
\mathbf m_{w_{\epsilon_2-\epsilon_3}}^{w_\gamma}&=\frac32\|\gamma\|-\frac12a_2-\frac14\big(\mathbf1_{a_1\not\equiv0\bmod2}+\mathbf1_{a_1\not\equiv a_2\mod 2}-\mathbf 1_{a_2\not\equiv0\bmod2}\big)\\
\mathbf m_{w_{\epsilon_1-\epsilon_2}}^{w_\gamma}&=\|\gamma\|-\lceil\frac{a_1}2\rceil\\
\mathbf m_{w_{2\epsilon_1}}^{w_\gamma}&=\frac12\|\gamma\|-\frac14\big(\mathbf 1_{a_1\not\equiv0\bmod2}+\mathbf 1_{a_1\not\equiv a_2\bmod2}+\mathbf 1_{a_2\not\equiv0\bmod2}\big).
\end{align*}
\end{thm}
\begin{proof}
\eqref{f4:1-sigma} can be re-written as:
\begin{align*}
(1-\sgn)\big([\cO_{\widetilde U}]-t_\gamma[\cO_{\widetilde U}]\big)=&\ \frac12(1-\sgn^{a_1})[\C_{y_1}]+\frac12(1-\sgn^{a_1+a_2})[\C_{x_1}]+\frac12(1-\sgn^{a_2})[\C_{y_2}]\\
=&\ \frac12(1-\sgn^{a_1})\big([\cO_{\cB_e}]-2[\cO_{\bP_{\alpha_1}}(-y_1)]-[\cO_{\bP_{\alpha_2}}\!(-x_1)]\big)\\&+\frac12(1-\sgn^{a_1+a_2})\big([\cO_{\cB_e}]-[\cO_{\bP_{\alpha_2}}\!(-x_1)]\big)+\frac12(1-\sgn^{a_2})\big([\cO_{\cB_e}]-[\cO_{\bP_{\alpha_2}}(-y_2)]\big).
\end{align*}
Thus, since
\begin{align*}
[\cO_{\widetilde U}]-t_\gamma[\cO_{\widetilde U}]=&\ \frac12(1+\sgn)\big([\cO_{\widetilde U}]-t_\gamma[\cO_{\widetilde U}]\big)+\frac12(1-\sgn)\big([\cO_{\widetilde U}]-t_\gamma[\cO_{\widetilde U}]\big)\\
=&\ \frac 12\|\gamma\|(1+\sgn)[\cO_{\cB_e}]+\big(\frac12a_1-\|\gamma\|\big)(1+\sgn)[\cO_{\bP_{\alpha_1}}\!(-x_1)]\\&+(\frac12a_2-\frac32\|\gamma\|)(1+\sgn)[\cO_{\bP_{\alpha_2}}\!(-x_1)]\\
&+(a_3-2\|\gamma\|)[\cO_{\bP_{\alpha_3}}\!(-x_2)]+(a_4-\|\gamma\|)[\cO_{\bP_{\alpha_4}}\!(-x_3)]\\
&+\frac14(1-\sgn^{a_1})\big([\cO_{\cB_e}]-2[\cO_{\bP_{\alpha_1}}(-y_1)]-[\cO_{\bP_{\alpha_2}}\!(-x_1)]\big)\\&+\frac14(1-\sgn^{a_1+a_2})\big([\cO_{\cB_e}]-[\cO_{\bP_{\alpha_2}}\!(-x_1)]\big)+\frac14(1-\sgn^{a_2})\big([\cO_{\cB_e}]-[\cO_{\bP_{\alpha_2}}(-y_2)]\big),
\end{align*}
we obtain the desired formulae.
\end{proof}

Let us compute $\mathbf m_w^{w_\gamma}-\mathbf m_w^{w_{s_i\gamma}}$ for $w\in\hatW$ and $\mu(w)=i\in\widehat S$:
\begin{cor}\label{cor:diff_m_expl_F} Let $G$ be of type F\textsubscript{4}. Then for $\gamma=a_1\alpha_1^\vee+a_2\alpha_2^\vee+a_3\alpha_3^\vee+a_4\alpha_4^\vee\in Q^\vee$,
    \begin{align*}
\mathbf m_{w_{\epsilon_1+\epsilon_2}}^{w_\gamma}-\mathbf m_{w_{\epsilon_1+\epsilon_2}}^{w_{s_0\gamma}}&=\frac12\big(a_1-\mathbf 1_{a_1\equiv a_2\bmod2}-\mathbf 1_{a_2\equiv0\bmod2}\big)\\
\mathbf m_{w_{\epsilon_1+\epsilon_3}}^{w_\gamma}-\mathbf m_{w_{\epsilon_1+\epsilon_3}}^{w_{s_1\gamma}}&=\Big\lfloor\frac{a_2-a_1}2\Big\rfloor-\Big\lfloor\frac{a_1}2\Big\rfloor\\
\mathbf m_{w_{\epsilon_1+\epsilon_4}}^{w_\gamma}-\mathbf m_{w_{\epsilon_1+\epsilon_4}}^{w_{s_2\gamma}}&=a_3-\frac12a_1+\frac12\big(\mathbf1_{a_1\not\equiv a_2\bmod2}-\mathbf1_{a_2\not\equiv0\mod 2}\big)\\
\mathbf m_{w_{\epsilon_1-\epsilon_4}}^{w_\gamma}-\mathbf m_{w_{\epsilon_1-\epsilon_4}}^{w_{s_3\gamma}}&=a_2-2a_3+a_4\\
\mathbf m_{w_{\epsilon_2+\epsilon_3}}^{w_\gamma}-\mathbf m_{w_{\epsilon_2+\epsilon_3}}^{w_{s_4\gamma}}&=a_3-2a_4\\
\mathbf m_{w_{\epsilon_2-\epsilon_3}}^{w_\gamma}-\mathbf m_{w_{\epsilon_2-\epsilon_3}}^{w_{s_2\gamma}}&=a_3-\frac12a_1-\frac12\big(\mathbf1_{a_1\not\equiv a_2\bmod2}-\mathbf1_{a_2\not\equiv0\mod 2}\big)\\
\mathbf m_{w_{\epsilon_1-\epsilon_2}}^{w_\gamma}-\mathbf m_{w_{\epsilon_1-\epsilon_2}}^{w_{s_1\gamma}}&=\Big\lceil\frac{a_2-a_1}2\Big\rceil-\Big\lceil\frac{a_1}2\Big\rceil\\
\mathbf m_{w_{2\epsilon_1}}^{w_\gamma}-\mathbf m_{w_{2\epsilon_1}}^{w_{s_0\gamma}}&=\frac12\big(a_1+\mathbf 1_{a_1\equiv a_2\bmod2}+\mathbf 1_{a_2\equiv0\bmod2}\big)
\end{align*}
\end{cor}
\begin{proof}
For $i\ne0$, note that $\|s_i\gamma\|=\|\gamma\|$, so the computation greatly simplifies. For example, for $w=w_{\epsilon_2+\epsilon_3}$, since
\(
s_2(\gamma)=a_1\alpha_1^\vee+(a_1-a_2+2a_3)\alpha_2^\vee+a_3\alpha_3^\vee+a_4\alpha_4^\vee,
\)
we have
\begin{align*}
    \mathbf m_{w_{\epsilon_1+\epsilon_4}}^{w_\gamma}-\mathbf m_{w_{\epsilon_1+\epsilon_4}}^{w_{s_2\gamma}}=&-\frac12a_2+\frac14\big(\mathbf1_{a_1\not\equiv0\bmod2}+\mathbf1_{a_1\not\equiv a_2\bmod2}-\mathbf1_{a_2\not\equiv0\mod 2}\big)\\
    &+\frac12(a_1-a_2+2a_3)-\frac14\big(\mathbf1_{a_1\not\equiv0\bmod2}+\mathbf1_{a_1\not\equiv a_1-a_2\bmod2}-\mathbf1_{a_1-a_2\not\equiv0\mod 2}\big)\\
    =&a_3-\frac12a_1+\frac12(\mathbf1_{a_1\not\equiv a_2\bmod2}-\mathbf1_{a_2\not\equiv0\mod 2}).
\end{align*}
When $i=0$, note that
\[
s_1\gamma=(a_1+2)\alpha_1^\vee+(a_2+3)\alpha_3^\vee+(a_3+2)\alpha_3^\vee+(a_3-a_4+1)\alpha_4^\vee.
\]
we have:
\[
\langle\gamma,\gamma\rangle-\langle s_1\gamma,s_1\gamma\rangle=\langle\gamma,\gamma\rangle-\langle\gamma-\theta^\vee,\gamma-\theta^\vee\rangle=2\langle\theta^\vee,\gamma\rangle-\langle\theta^\vee,\theta^\vee\rangle=a_1-1.
\]
Thus
\begin{align*}
    \mathbf m_{w_{\epsilon_1+\epsilon_2}}^{w_\gamma}-\mathbf m_{w_{\epsilon_1+\epsilon_2}}^{w_{s_0\gamma}}=&\ \frac12(2a_1-1)+\frac14\big(\mathbf 1_{a_1\not\equiv0\bmod2}+\mathbf 1_{a_1\not\equiv a_2\bmod2}+\mathbf 1_{a_2\not\equiv0\bmod2}\big)\\
    &-\frac14\big(\mathbf 1_{a_1\not\equiv0\bmod2}+\mathbf 1_{a_1\not\equiv a_2+1\bmod2}+\mathbf 1_{a_2+1\not\equiv0\bmod2}\big)\\
    =&\ \frac12\big(a_1-\mathbf 1_{a_1\equiv a_2\bmod2}-\mathbf 1_{a_2\equiv0\bmod2}\big).\qedhere
\end{align*}
\end{proof}

\subsection{When $G$ is of type G\textsubscript{2}} To calculate the special values of Kazhdan-Lusztig polynomials, we prove some lemmas:

\begin{lemma}\label{lem:g2-alpha-induct}
For each $m,n\in\Z$ there is a $\Cent_e$-equivariant exact sequence
\[
0\to\cO_{\bP_\alpha}\!(m,n-2)\to \cO_{\bP_\alpha}\!(m,n)\to\cF_{n-m}\to 0,
\]
where in $K^{\Cent_e}(\cB_e)$,
\begin{equation}\label{eq:fk-calculation}
[\cF_k]=\begin{cases}
    [(1+\sgn)\otimes\cO_{\bP_\alpha}]-[\std\otimes\cO_{\bP_\alpha}(1,0)]&\text{if }k\equiv0\pmod3\\
    [\std\otimes\cO_{\bP_\alpha}]-[(1+\sgn)\otimes\cO_{\bP_\alpha}(1,0)]&\text{if }k\equiv1\pmod3\\
    [\std\otimes\cO_{\bP_\alpha}]-[\std\otimes\cO_{\bP_\alpha}(1,0)]&\text{if }k\equiv2\pmod3
\end{cases}
\end{equation}
\end{lemma}
\begin{proof}
    Let $\tO_2=\G_m\rtimes\Z/2\subset\GL_2$ act on $\bP^1$ so $a\in\G_m$ acts by $[x:y]\mapsto [ax:a^{-1}y]$. Then by Corollary~\ref{cor:o2-equiv} we know $H^0(\bP^1,\cO_{\bP^1}(0,2))\simeq\C+\xi_2$, so there is a non-zero homomorphism $\cO_{\bP^1}\to\cO_{\bP^1}(0,2)$. Thus for any $m,n\in\Z$ there is an exact sequence of $\tO_2$-equivariant sheaves
    \[
    0\to \cO_{\bP^1}(m,n-2)\to\cO_{\bP^1}(m,n) \to \cF\to0
    \]
    where $\cF$ has length $2$. When $m< n$ the cohomology long exact sequence gives
    \[
    0\to H^0(\cO_{\bP^1}(m,n-2))\to H^0(\cO_{\bP^1}(m,n))\to H^0(\cF)\to 0,
    \]
    so using Corollary~\ref{cor:o2-equiv}, we conclude $H^0(\cF)\simeq \xi_{n-m}$. When $m>n$ the cohomology long exact sequence gives
    \[
    0\to H^0(\cF)\to H^1(\cO_{\bP^1}(m,n-2))\to H^1(\cO_{\bP^1}(m,n))\to 0,
    \]
    which by Corollary~\ref{cor:o2-equiv} combined with Serre duality again gives $H^0(\cF)\simeq \xi_{n-m}$. Finally, when $m=n$ the cohomology longe exact sequence gives
    \[
    0\to H^0(\cO_{\bP^1}(m,m))=\sgn^{\otimes m}\to H^0(\cF)\to H^1(\cO_{\bP^1}(m,m-2))=\sgn^{m-1}\to0,
    \]
    which again gives $H^0(\cF)\simeq\C+\sgn=\xi_0$. Since $\cF$ is $\tO_2$-equivariant and of length $2$, it must be supported on $\{0,\infty\}\subset\bP^1$. Moreover, $\G_m$ acts on the stalk of $\cO_{\bP^1}(m,n)$ at $0$ as $\C\langle n-m\rangle$, so $\cF\simeq\underline\C_0\langle n-m\rangle \oplus\underline\C_\infty\langle m-n\rangle$. For convenience, let $\cF_k\colonequals\underline\C_0\langle k\rangle \oplus\underline\C_\infty\langle -k\rangle$ for $k\in\Z$, considered as a $S_3\subset\tO_2$-equivariant sheaf. From the description, it is clear that $\cF_k$ only depends on $k$ modulo $3$, so it suffices to calculate $\cF_0$, $\cF_1$, and $\cF_2$. 

    Note that there is an exact sequence
    \[
    0\to\cO_{\bP^1}(2,0)\to\cO_{\bP^1}\to \cF_0\to 0,
    \]
    and the Euler sequence gives
    \[
    0\to\cO_{\bP^1}(2,0)\to\std\otimes\cO_{\bP^1}(1,0)\to\cO_{\bP^1}(1,1)=\sgn\otimes\cO_{\bP^1}\to 0,
    \]
    so in $K^{S_3}(\bP^1)$,
    \[
    [\cF_0]=[\cO_{\bP^1}]-[\cO_{\bP^1}(2,0)]=[(1+\sgn)\otimes\cO_{\bP^1}]-[\std\otimes\cO_{\bP^1}(1,0)].
    \]
    Similarly, the exact sequence
    \[
    0\to\cO_{\bP^1}(0,-1)\to\cO_{\bP^1}(0,1)\to\cF_1\to 0,
    \]
    together with the Euler sequence
    \[
    0\to\cO_{\bP^1}(1,0)\to\std\otimes\cO_{\bP^1}\to\cO_{\bP^1}(0,1)\to 0
    \]
    gives
    \[
    [\cF_1]=[\cO_{\bP^1}(0,1)]-[\cO_{\bP^1}(0,-1)]=[\std\otimes\cO_{\bP^1}]-[(1+\sgn)\otimes\cO_{\bP^1}(1,0)].
    \]
    Finally, there are exact sequences
    \[
    0\to\cO_{\bP^1}\to\cO_{\bP^1}(0,2)\to\cF_{2}\to0
    \]
    and twists of the Euler sequence
    \begin{align*}
        0\to\sgn\otimes\cO_{\bP^1}\to &\std\otimes\cO_{\bP^1}(0,1)\to \cO_{\bP^1}(0,2)\to0\\
        0\to\std\otimes\cO_{\bP^1}(1,0)\to &\std^{\otimes 2}\otimes\cO_{\bP^1}\to\std\otimes\cO_{\bP^1}(0,1)\to0,
    \end{align*}
    which together gives
    \begin{align*}
        [\cF_2]&=[\cO_{\bP^1}(0,2)]-[\cO_{\bP^1}]\\
        &=[\std\otimes\cO_{\bP^1}(0,1)]-[(1+\sgn)\otimes\cO_{\bP^1}]\\
        &=[\std^{\otimes2}\otimes\cO_{\bP^1}]-[(1+\sgn)\otimes\cO_{\bP^1}]-[\std\otimes\cO_{\bP^1}(1,0)]\\
        &=[\std\otimes\cO_{\bP^1}]-[\std\otimes\cO_{\bP^1}(1,0)].\qedhere
    \end{align*}    
\end{proof}
\begin{remark}
Note that \eqref{eq:fk-calculation} is useful, since $[\cO_{\bP_\alpha}(1,0)]$ and tensor products of it with representations of $S_3$ are already part of the canonical basis, and $[\cO_{\bP_\alpha}]$ is expressed in terms of the canonical basis in \eqref{g2:Oalpha-exp}.
\end{remark}

\begin{prop}\label{g2:malpha-kl-poly}
Let $G$ be of type G\textsubscript{2}. Then for $\gamma=m\alpha^\vee$,
\begin{align*}
    \mathbf m_1^{w_\gamma}&=1\\
    \mathbf m_{s_0}^{w_\gamma}&=\Big\lceil\frac m2\Big\rceil+N_0\\
    \mathbf m_{s_1s_0}^{w_\gamma}&=\Big\lceil\frac m2\Big\rceil+N_0+N_1\\
    \mathbf m_{s_2s_1s_0}^{w_\gamma}&=\Big\lceil\frac m2\Big\rceil+\frac{m(m-1)}2\\
    \mathbf m_{s_1s_2s_1s_0}^{w_\gamma}&=\frac{m(m-1)}2+N_2\\
    \mathbf m_{s_0s_1s_2s_1s_0}^{w_\gamma}&=N_1+N_2\\
    \mathbf m_{s_2s_1s_2s_1s_0}^{w_\gamma}&=\Big\lfloor\frac m2\Big\rfloor+\frac{m(m-1)}2\\
    \mathbf m_{s_1s_2s_1s_2s_1s_0}^{w_\gamma}&=\Big\lfloor\frac m2\Big\rfloor+N_0+N_1\\
    \mathbf m_{s_0s_1s_2s_1s_2s_1s_0}^{w_\gamma}&=\Big\lfloor\frac m2\Big\rfloor+N_0,
\end{align*}
where for $r\in\{0,1,2\}$,
\begin{equation}\label{eq:nr-defn}
N_r\colonequals\sum_{\substack{1\leqslant j\leqslant m-1\\j\equiv-r\bmod3}}\!(m-j).
\end{equation}
\end{prop}

\begin{proof}
    By using Lemma~\ref{alpha-action} inductively, we obtain:
    \[
        [\cO_{\widetilde U}(m\alpha^\vee)]=[\cO_{\widetilde U}]-\sum_{i=0}^{m-1}[\cO_{\bP_\alpha}(-i,i)].\]
Moreover by Lemma~\ref{lem:g2-alpha-induct}, we have:
\(
[\cO_{\bP_\alpha}(-i,i)]=[\sgn^i\otimes\cO_{\bP_\alpha}]+\sum_{j=1}^i[\cF_{2j}],
\)
so in total
\begin{align}
    [\cO_{\widetilde U}(m\alpha^\vee)]&=[\cO_{\widetilde U}]-\sum_{i=0}^{m-1}[\sgn^i\otimes\cO_{\bP_\alpha}]-\sum_{i=0}^{m-1}\sum_{j=1}^i[\cF_{2j}]\nonumber\\
    &=[\cO_{\widetilde U}]-\Big\lceil\frac m2\Big\rceil[\cO_{\bP_\alpha}]-\Big\lfloor\frac m2\Big\rfloor[\sgn\otimes\cO_{\bP_\alpha}]-\sum_{j=1}^{m-1}(m-j)[\cF_{2j}]\label{eq:g2-malpha-calc1}\\
    &=[\cO_{\widetilde U}]-\lceil\frac m2\rceil[\cO_{\bP_\alpha}]-\Big\lfloor\frac m2\Big\rfloor[\sgn\otimes\cO_{\bP_\alpha}]-N_0[\cF_0]-N_1[\cF_1]-N_2[\cF_2],\nonumber
\end{align}
where $N_r$ is as defined in \eqref{eq:nr-defn}.
Now, substituting the formulae in Lemma~\ref{lem:g2-alpha-induct} for $[\cF_k]$ into \eqref{eq:g2-malpha-calc1} gives
\begin{align*}
    [\cO_{\widetilde U}(m\alpha^\vee)]=&\ [\cO_{\widetilde U}]-\Big\lceil\frac m2\Big\rceil[\cO_{\bP_\alpha}]-\Big\lfloor\frac m2\Big\rfloor[\sgn\otimes\cO_{\bP_\alpha}]-N_0\big([(1+\sgn)\otimes\cO_{\bP_\alpha}]-[\std\otimes\cO_{\bP_\alpha}(1,0)]\big)\\
    &-N_1\big([\std\otimes\cO_{\bP_\alpha}]-[(1+\sgn)\otimes\cO_{\bP_\alpha}(1,0)]\big)-N_2\big([\std\otimes\cO_{\bP_\alpha}]-[\std\otimes\cO_{\bP_\alpha}(0,1)]\big)\\
    =&\ [\cO_{\widetilde U}]-\Big(\Big\lceil \frac m2\Big\rceil+N_0\Big)[\cO_{\bP_\alpha}]-\Big(\Big\lfloor\frac m2\Big\rfloor+N_0\Big)[\sgn\otimes\cO_{\bP_\alpha}]-(N_1+N_2)[\std\otimes\cO_{\bP_\alpha}]\\
    &+N_1[(1+\sgn)\otimes\cO_{\bP_\alpha}(1,0)]+(N_0+N_2)[\std\otimes\cO_{\bP_{\alpha}}(1,0)].
\end{align*}
Finally, substituting \eqref{g2:Oalpha-exp} for $[\cO_{\bP_\alpha}]$ gives
\begin{align*}
    [\cO_{\widetilde U}(m\alpha^\vee)]=&\ [\cO_{\widetilde U}]-\Big(\Big\lceil\frac m2\Big\rceil+N_0\Big)[\cO_{\cB_e}]-\Big(\Big\lfloor\frac m2\Big\rfloor+N_0\Big)[\sgn\otimes\cO_{\cB_e}]-(N_1+N_2)[\std\otimes\cO_{\cB_e}]\\
    &+\Big(\Big\lceil\frac m2\Big\rceil+\frac{m(m-1)}2\Big)[\cO_{\bP_\beta}\!(-x)]+\Big(\Big\lfloor\frac m2\Big\rfloor+\frac{m(m-1)}2\Big)[\cO_{\bP_\beta}(-y)]\\
    &+\Big(\Big\lceil\frac m2\Big\rceil+N_0+N_1\Big)[\cO_{\bP_\alpha}(0,-1)]+\Big(\Big\lfloor\frac m2\Big\rfloor+N_0+N_1\Big)[\cO_{\bP_\alpha}(1,0)]\\
    &+\Big(\frac{m(m-1)}2+N_2\Big)[\std\otimes\cO_{\bP_\alpha}(1,0)],
\end{align*}
where we use the identity $N_0+N_1+N_2=m(m-1)/2$ repeatedly. We conclude by Proposition~\ref{prop:canonical-basis-g}.
\end{proof}

\begin{remark}
The $N_r(m)$ are quasi-polynomials of $m$: for $r\in\{0,1,2\}$,
    \begin{align*}
    N_r&=\sum_{j'=1}^{\lfloor (m+r)/3\rfloor}\!(m+r-3j')=\Big\lfloor\frac{m+r}3\Big\rfloor(m+r)-\frac32\Big\lfloor\frac{m+r}3\Big\rfloor\cdot\Big(\Big\lfloor\frac{m+r}3\Big\rfloor+1\Big).
\end{align*}
\end{remark}
Next, we finally compute the special values of Kazhdan-Lusztig polynomial for arbitrary $\gamma=m\alpha^\vee+n\beta^\vee$ in type $G_2$:

\begin{thm}\label{thm_KL_type_G}
    Let $G$ be of type G\textsubscript{2}. Then for $\gamma=m\alpha^\vee+n\beta^\vee\in Q^\vee$,
\begin{align*}
    \mathbf m_1^{w_\gamma}&=1\\
    \mathbf m_{s_0}^{w_\gamma}&=N_0+\frac12\Big(n(n-m)+\Big\lceil \frac m2\Big\rceil-\Big\lfloor \frac{n-m}2\Big\rfloor+\Big\lceil\frac n2\Big\rceil\Big)\\
    \mathbf m_{s_1s_0}^{w_\gamma}&=\Big\lceil \frac m2\Big\rceil+N_0+N_1+n(n-m)\\
    \mathbf m_{s_2s_1s_0}^{w_\gamma}&=\frac{m(m-1)}2+\frac12\Big(3n(n-m)+\Big\lceil \frac m2\Big\rceil-\Big\lfloor\frac{n-m}2\Big\rfloor-\Big\lceil\frac n2\Big\rceil\Big)\\
    \mathbf m_{s_1s_2s_1s_0}^{w_\gamma}&=\frac{m(m-1)}2+N_2+2n(n-m)\\
    \mathbf m_{s_0s_1s_2s_1s_0}^{w_\gamma}&=N_1+N_2+n(n-m)\\
    \mathbf m_{s_2s_1s_2s_1s_0}^{w_\gamma}&=\frac{m(m-1)}2+\frac12\Big(3n(n-m)+\Big\lfloor \frac m2\Big\rfloor-\Big\lceil\frac{n-m}2\Big\rceil-\Big\lfloor\frac n2\Big\rfloor\Big)\\
    \mathbf m_{s_1s_2s_1s_2s_1s_0}^{w_\gamma}&=\Big\lfloor \frac m2\Big\rfloor+N_0+N_1+n(n-m)\\
    \mathbf m_{s_0s_1s_2s_1s_2s_1s_0}^{w_\gamma}&=N_0+\frac12\Big(n(n-m)+\Big\lfloor \frac m2\Big\rfloor-\Big\lceil \frac{n-m}2\Big\rceil+\Big\lfloor\frac n2\Big\rfloor\Big),
\end{align*}
where for $r\in\{0,1,2\}$,
\[
N_r\colonequals\sum_{\substack{1\leqslant j\leqslant m-1\\j\equiv-r\bmod3}}\!(m-j).
\]
\end{thm}

\begin{proof}
    By a repeated application of Lemma~\ref{alpha-action}, combined with Lemma~\ref{equivariant-structure-g}, we have
    \[
    [\cO_{\widetilde U}(m\alpha^\vee+n\beta^\vee)]=[\cO_{\widetilde U}(m\alpha^\vee)]-\sum_{i=0}^{n-1}[\cO_{\bP_\beta}((i-m)z+iy)].
    \]
    Since the first term is calculated in Proposition~\ref{g2:malpha-kl-poly}, it suffices to calculate the second term. By the standard exact sequences
    \begin{align*}
        0\to \cO_{\bP_\beta}((a-1)z+by)\to &\cO_{\bP_\beta}(az+by)\to\sgn^a\otimes\underline{\C}_z\to 0\\
        0\to \cO_{\bP_\beta}(az+(b-1)y)\to &\cO_{\bP_\beta}(az+by)\to\sgn^{b}\otimes\underline{\C}_y\to 0,
    \end{align*}
    we have
    \[
    [\cO_{\bP_\beta}(az+by)]=[\cO_{\bP_\beta}]+\sum_{i=1}^{a}[\sgn^i\otimes\underline\C_x]+\sum_{i=1}^{b}[\sgn^i\otimes\underline\C_y].
    \]
    Thus,
    \begin{equation}\label{eq:beta-g2-kl-poly-calc}
    \sum_{i=0}^{n-1}[\cO_{\bP_\beta}((i-m)z+iy)]=n[\cO_{\bP_\beta}]+\sum_{i=0}^{n-1}\sum_{j=1}^{i-m}[\sgn^j\otimes\underline\C_x]+\sum_{i=0}^{n-1}\sum_{j=1}^{i}[\sgn^j\otimes\underline\C_y].
    \end{equation}
    Multiplying \eqref{eq:beta-g2-kl-poly-calc} by $1+\sgn$ gives
    \begin{align*}
    (1+\sgn)n[\cO_{\bP_\beta}]+(1+\sgn)\bigg(\sum_{i=0}^{n-1}\sum_{j=1}^{i-m}[\underline\C_x]+\sum_{i=0}^{n-1}\sum_{j=1}^i[\underline\C_y]\bigg)=&\ n(1+\sgn)[\cO_{\bP_\beta}]+(1+\sgn)n(n-m-1)[\underline\C_x]\\
    =&\ n(n-m)(1+\sgn)[\cO_{\bP_\beta}]\\
    &-n(n-m-1)(1+\sgn)[\cO_{\bP_\beta}\!(-x)].
    \end{align*}
    On the other hand, multiplying \eqref{eq:beta-g2-kl-poly-calc} by $1-\sgn$ gives
\begin{equation}\label{eq:beta-g2-kl-poly-calc2}
(1-\sgn)n[\cO_{\bP_\beta}]+\sum_{i=0}^{n-1}(\sgn-\sgn^{i-m+1})[\underline\C_x]+\sum_{i=0}^{n-1}(\sgn-\sgn^{i+1})[\C_y].
\end{equation}
Observing that in the first sum only $i$ with $i\not\equiv m\pmod 2$ contributes, and in the second sum only odd $i$ contributes, we see \eqref{eq:beta-g2-kl-poly-calc2} equals
 \[
        (1-\sgn)n[\cO_{\bP_\beta}]+\Big(\Big\lceil\frac m2\Big\rceil+\Big\lfloor\frac{n-m}2\Big\rfloor\Big)\big([\sgn\otimes\underline\C_x]-[\C_x]\big)+\Big\lfloor\frac n2\Big\rfloor\big([\sgn\otimes\underline\C_y]-[\underline\C_y]\big).
\]
    Further observing that:
    \begin{align*}
        [\sgn\otimes\C_x]-[\C_x]&=(1-\sgn)[\cO_{\bP_\beta}\!(-x)]-(1-\sgn)[\cO_{\bP_\beta}]\\
        [\sgn\otimes\C_y]-[\C_y]&=-(1-\sgn)[\cO_{\bP_\beta}\!(-x)]-(1-\sgn)[\cO_{\bP_\beta}]
    \end{align*}
    shows the above equals
\[
\Big(n-\Big\lceil \frac m2\Big\rceil-\Big\lfloor \frac{n-m}2\Big\rfloor-\Big\lfloor\frac n2\Big\rfloor\Big)(1-\sgn)[\cO_{\bP_\beta}]+\Big(\Big\lceil \frac m2\Big\rceil+\Big\lfloor\frac{n-m}2\Big\rfloor-\Big\lfloor\frac n2\Big\rfloor\Big)(1-\sgn)[\cO_{\bP_\beta}\!(-x)].\]
Thus \eqref{eq:beta-g2-kl-poly-calc}, which is $[\cO_{\widetilde U}(m\alpha^\vee)]-[\cO_{\widetilde U}(m\alpha^\vee+n\beta^\vee)]$, equals:
\begin{align}
&n(n-m)\frac{1+\sgn}2[\cO_{\bP_\beta}]+\Big(\Big\lceil\frac n2\Big\rceil-\Big\lceil \frac m2\Big\rceil-\Big\lfloor \frac{n-m}2\Big\rfloor\Big)\frac{1-\sgn}2[\cO_{\bP_\beta}]\nonumber\\
&-\frac12\Big(n(n-m-1)-\Big\lceil \frac m2\Big\rceil-\Big\lfloor\frac{n-m}2\Big\rfloor+\Big\lfloor\frac n2\Big\rfloor\Big)[\cO_{\bP_\beta}\!(-x)]\label{eq:beta-g2-kl-poly-calc3}\\&-\frac12\Big(n(n-m-1)+\Big\lceil \frac m2\Big\rceil+\Big\lfloor\frac{n-m}2\Big\rfloor-\Big\lfloor\frac n2\Big\rfloor\Big)[\cO_{\bP_\beta}(-y)].\nonumber
\end{align}
Thus, to finish we have to compute $[\cO_{\bP_\beta}]=[\cO_{\widetilde U}]-[\cO_{\widetilde U}(\beta^\vee)]$ in terms of the canonical basis: using Lemma~\ref{alpha-action} and \eqref{eq:alpha+beta} and Lemma~\ref{equivariant-structure-g},
\begin{align*}
    [\cO_{\widetilde U}]-[\cO_{\widetilde U}(\beta^\vee)]&=[\cO_{\widetilde U}]-[\cO_{\widetilde U}(\alpha^\vee+\beta^\vee)]-[\cO_{\bP_\alpha}(\beta^\vee)]\\
    &=[\cO_{\cB_e}]-[\cO_{\bP_\alpha}(3,0)]-[\cO_{\bP_\alpha}(0,-1)]\\
    &=[\std\otimes\cO_{\bP_\alpha}(2,0)],
\end{align*}
where the last equality is by Euler's sequence. Now again by Euler's sequence
\begin{align*}
    0\to\cO_{\bP_\alpha}(2,0)\to&\std\otimes\cO_{\bP_\alpha}(1,0)\to\sgn\otimes\cO_{\bP_\alpha}\to 0,
\end{align*}
we have
\begin{align*}
    [\std\otimes\cO_{\bP_\alpha}(2,0)]
    =&\ [\std^{\otimes2}\otimes\cO_{\bP_\alpha}(1,0)]-[\std\otimes\cO_{\bP_\alpha}]\\
    =&\ [\cO_{\bP_\alpha}(1,0)]+[\cO_{\bP_\alpha}(0,-1)]+[\std\otimes\cO_{\bP_\alpha}(1,0)]\\
    &-[\std\otimes\cO_{\cB_e}]+[\cO_{\bP_\beta}\!(-x)]+[\cO_{\bP_\beta}(-y)]+[\std\otimes\cO_{\bP_\alpha}(1,0)]\\
    =&\ [\cO_{\bP_\alpha}(1,0)]+[\cO_{\bP_\alpha}(0,-1)]+2[\std\otimes\cO_{\bP_\alpha}(1,0)]+[\cO_{\bP_\beta}\!(-x)]+[\cO_{\bP_\beta}(-y)]-[\std\otimes\cO_{\cB_e}],
\end{align*}
where the last step uses \eqref{g2:Oalpha-exp}. So in total,
\[
[\cO_{\bP_\beta}]=[\cO_{\cB_e}]+[\std\otimes\cO_{\cB_e}]-[\cO_{\bP_\alpha}(1,0)]-[\cO_{\bP_\alpha}(0,-1)]-2[\std\otimes\cO_{\bP_\alpha}(1,0)]-[\cO_{\bP_\beta}\!(-x)]-[\cO_{\bP_\beta}(-y)].\]
Note that in particular:
\begin{align*}
    \frac12(1-\sgn)[\cO_{\bP_\beta}]=&\ \frac12(1-\sgn)[\cO_{\cB_e}]\\
    \frac12(1+\sgn)[\cO_{\bP_\beta}]=&\ \frac12(1+\sgn)[\cO_{\cB_e}]+[\std\otimes\cO_{\cB_e}]-[\cO_{\bP_\alpha}(1,0)]-[\cO_{\bP_\alpha}(0,-1)]\\
    &-2[\std\otimes\cO_{\bP_\alpha}(1,0)]-[\cO_{\bP_\beta}\!(-x)]-[\cO_{\bP_\beta}(-y)].
\end{align*}
Thus \eqref{eq:beta-g2-kl-poly-calc3}, which is $[\cO_{\widetilde U}(m\alpha^\vee)]-[\cO_{\widetilde U}(m\alpha^\vee+n\beta^\vee)]$, equals
\begin{align*}
&\frac12\Big(n(n-m)-\Big\lceil \frac m2\Big\rceil-\Big\lfloor \frac{n-m}2\Big\rfloor+\Big\lceil\frac n2\Big\rceil\Big)[\cO_{\cB_e}]+\frac12\Big(n(n-m)+\Big\lceil \frac m2\Big\rceil+\Big\lfloor \frac{n-m}2\Big\rfloor-\Big\lceil\frac n2\Big\rceil\Big)[\sgn\otimes\cO_{\cB_e}]\\
&+n(m-m)[\std\otimes\cO_{\cB_e}]-n(n-m)[\cO_{\bP_\alpha}(1,0)]-n(n-m)[\cO_{\bP_\alpha}(0,-1)]-2n(n-m)[\std\otimes\cO_{\bP_\alpha}(1,0)]\\
&-\frac12\Big(3n(n-m)-\Big\lceil \frac m2\Big\rceil-\Big\lfloor\frac{n-m}2\Big\rfloor-\Big\lceil\frac n2\Big\rceil\Big)[\cO_{\bP_\beta}\!(-x)]-\frac12\Big(3n(n-m)-\Big\lfloor \frac m2\Big\rfloor-\Big\lceil\frac{n-m}2\Big\rceil-\Big\lfloor\frac n2\Big\rfloor\Big)[\cO_{\bP_\beta}(-y)].
\end{align*}
Combining with Proposition~\ref{g2:malpha-kl-poly} gives the desired special values of Kazhdan-Lusztig polynomials.
\end{proof}

As usual, we compute $\mathbf m_w^{w_\gamma}-\mathbf m_w^{w_{s_i\gamma}}$ for $w\in\hatW$ and $\mu(w)=i\in\widehat S$:

\begin{cor}\label{cor:diff_m_expl_G} Let $G$ be of type G\textsubscript{2}. Then for $\gamma=m\alpha^\vee+n\beta^\vee\in Q^\vee$,
    \begin{align*}
        \mathbf m^{w_\gamma}_{s_0}-\mathbf m_{s_0}^{w_{s_0\gamma}}&=2\Big\lfloor\frac{m-1}3\Big\rfloor+\frac12(m-\mathbf 1_{m\equiv n\bmod2}-\mathbf 1_{n\equiv0\bmod2})\\
        \mathbf m_{s_1s_0}^{w_\gamma}-\mathbf m_{s_1s_0}^{w_{s_1\gamma}}&=\frac12(n+\mathbf 1_{m\not\equiv0\bmod2}-\mathbf 1_{m\not\equiv n\bmod2})-\Big\lfloor\frac{m-2}3\Big\rfloor\\
        \mathbf m^{w_\gamma}_{s_2s_1s_0}-\mathbf m^{w_{s_2\gamma}}_{s_2s_1s_0}&=-\Big\lfloor\frac{n-m}2\Big\rfloor-\Big\lceil\frac{n}2\Big\rceil\\
        \mathbf m^{w_\gamma}_{s_1s_2s_1s_0}-\mathbf m^{w_{s_1\gamma}}_{s_1s_2s_1s_0}&=n-m+\Big\lfloor\frac{m-2}3\Big\rfloor\\
        \mathbf m^{w_\gamma}_{s_0s_1s_2s_1s_0}-\mathbf m^{w_{s_0\gamma}}_{s_0s_1s_2s_1s_0}&=-2\Big\lfloor\frac{m-1}3\Big\rfloor\\
        \mathbf m^{w_\gamma}_{s_2s_1s_2s_1s_0}-\mathbf m^{w_{s_2\gamma}}_{s_2s_1s_2s_1s_0}&=-\Big\lceil\frac{n-m}2\Big\rceil-\Big\lfloor\frac{n}2\Big\rfloor\\
        \mathbf m^{w_\gamma}_{s_1s_2s_1s_2s_1s_0}-\mathbf m^{w_{s_1\gamma}}_{s_1s_2s_1s_2s_1s_0}&=\frac12(n+\mathbf 1_{m\equiv0\bmod2}-\mathbf 1_{m\equiv n\bmod2})-\Big\lfloor\frac{m-2}3\Big\rfloor\\
        \mathbf m^{w_\gamma}_{s_0s_1s_2s_1s_2s_1s_0}-\mathbf m^{w_{s_0\gamma}}_{s_0s_1s_2s_1s_2s_1s_0}&=2\Big\lfloor\frac{m-1}3\Big\rfloor+\frac12(m-\mathbf 1_{m\not\equiv n\bmod2}-\mathbf 1_{n\not\equiv0\bmod2})\\
    \end{align*}
\end{cor}
\begin{proof}
    Note: (see Remark~\ref{type-g-highest})
    \begin{align*}
    s_0\gamma&=(2-m)\alpha^\vee+(n-m+1)\beta^\vee\\
        s_1\gamma&=(3n-m)\alpha^\vee+n\beta^\vee\\
        s_2\gamma&=m\alpha^\vee+(m-n)\beta^\vee.
    \end{align*}
    We will first calculate how the $N_r$ change under the transformations, for $r\in\{0,1,2\}$:
    \begin{align*}
        N_0(m)-N_0(2-m)&=\sum_{\substack{1\leqslant j\leqslant m-1\\j\equiv0\bmod3}}\!(m-j)-\sum_{\substack{1\leqslant j\leqslant 1-m\\j\equiv0\bmod3}}(2-m-j)\\
        &=\sum_{\substack{1\leqslant j\leqslant m-1\\j\equiv0\bmod3}}\big((m-j)+(2-m+j)\big)\\
        &=2\Big\lfloor\frac{m-1}3\Big\rfloor.
    \end{align*}
    Similarly,
    \begin{align*}
        N_2(m)-N_2(3n-m)&=\sum_{\substack{1\leqslant j\leqslant m-1\\j\equiv1\bmod3}}\!(m-j)-\sum_{\substack{1\leqslant j\leqslant 3n-m-1\\j\equiv1\bmod3}}(3n-m-j)\\
        &=\sum_{\substack{1\leqslant j\leqslant m-1\\j\equiv1\bmod3}}\!(m-j)-\sum_{\substack{m\leqslant j\leqslant 3n-2\\j\equiv1\bmod3}}(j-m+1)\\
        &=\Big\lfloor\frac {m-2}3\Big\rfloor-\sum_{\substack{1\leqslant j\leqslant 3n-2\\j\equiv1\bmod3}}(j-m+1)\\
        &=\Big\lfloor\frac{m-2}3\Big\rfloor-\sum_{j'=1}^{n}(3j'-m-1)\\
        &=\Big\lfloor\frac{m-2}3\Big\rfloor-\frac12n(3n-2m+1),
    \end{align*}
    where for the second-to-last equality we made the change of variables $j=3j'-2$. To compute $\mathbf m^{w_\gamma}_{s_0}-\mathbf m^{w_{s_0\gamma}}_{s_0}$, it is first convenient to write:
    \[
    \mathbf m^{w_\gamma}_{s_0}=N_0(m)+\frac12n(n-m+1)+\frac12\Big(\Big\lceil \frac m2\Big\rceil-\Big\lfloor \frac{n-m}2\Big\rfloor-\Big\lfloor\frac n2\Big\rfloor\Big),
    \]
    since $n(n-m+1)$ is invariant under $s_0$. Thus
    \begin{align*}
        \mathbf m^{w_\gamma}_{s_0}-\mathbf m^{w_{s_0\gamma}}_{s_0}=&\ N_0(m)-N_0(2-m)+\frac12\Big(\Big\lceil \frac m2\Big\rceil-\Big\lfloor \frac{n-m}2\Big\rfloor-\Big\lfloor\frac n2\Big\rfloor-\Big\lceil \frac{2-m}2\Big\rceil+\Big\lfloor \frac{n-1}2\Big\rfloor+\Big\lfloor\frac{n-m+1}2\Big\rfloor\Big)\\
        =&\ 2\Big\lfloor\frac{m-1}3\Big\rfloor+\frac12(m-1+\mathbf1_{m\not\equiv n\bmod2}-\mathbf1_{n\equiv0\bmod2})\\
        =&\ 2\Big\lfloor\frac{m-1}3\Big\rfloor+\frac12(m-\mathbf1_{m\equiv n\bmod2}-\mathbf1_{n\equiv0\bmod2}).
    \end{align*}
The computations for $s_0s_1s_2s_1s_0$ and $s_0s_1s_2s_1s_2s_1s_0$ are completely analogous. Similarly, since $N_0(m)+N_1(m)+N_2(m)=m(m-1)/2$ we can write
\[
\mathbf m^{w_\gamma}_{s_1s_0}=\frac{m^2}2+n(n-m)+\frac12\mathbf 1_{m\not\equiv0\bmod2}-N_2(m).
\]
Thus,
\begin{align*}
    \mathbf m^{w_\gamma}_{s_1s_0}-\mathbf m^{w_{s_1\gamma}}_{s_1s_0}=&\ \frac{m^2-(3n-m)^2}2+n(3n-2m)+\frac12(\mathbf 1_{m\not\equiv0\bmod2}-\mathbf 1_{m\not\equiv n\bmod2})\\
    &-\Big\lfloor\frac{m-2}3\Big\rfloor+\frac12n(3n-2m+1)\\
    =&\ \frac12(n+\mathbf 1_{m\not\equiv0\bmod2}-\mathbf 1_{m\not\equiv n\bmod2})-\Big\lfloor\frac{m-2}3\Big\rfloor.
\end{align*}
The computation for $s_1s_2s_1s_0$ and $s_1s_2s_1s_2s_1s_0$ is completely analogous. Finally, since $\frac{m(m-1)}2$, $n(n-m)$, and $\lceil \frac m2\rceil$ are all invariant under $s_2$, we have:
\[
    \mathbf m^{w_\gamma}_{w_{s_2s_1s_0}}-\mathbf m^{w_{s_2\gamma}}_{w_{s_2s_1s_0}}=\frac12\Big(-\Big\lfloor\frac{n-m}2\Big\rfloor-\Big\lceil\frac{n}2\Big\rceil+\Big\lfloor\frac{-n}2\Big\rfloor+\Big\lceil\frac{m-n}2\Big\rceil\Big)=-\Big\lfloor\frac{n-m}2\Big\rfloor-\Big\lceil\frac{n}2\Big\rceil.
\]
The computation for $s_2s_1s_2s_1s_0$ is completely analogous.
\end{proof}

\section{Characters of $L(\Lambda)$ for certain singular $\Lambda$}\label{sec:character-computation}

We are now ready to compute characters of certain irreducible representations of $\widehat{\mathfrak g}$:

\begin{thm}\label{thm:character-formula}
Let $w \in c^0_{\subreg}$ and $i\colonequals\mu(w) \in \widehat{S}$. Pick $\lambda=-\Lambda_i+\sum_{k \neq i}m_k\Lambda_k$, and set $\Lambda:=w^{-1} \circ \lambda$. We have 
\begin{equation*}
\widehat{R}\operatorname{ch}L(\Lambda) = \frac{1}{2}\sum_{\gamma \in Q^\vee}\sum_{u \in W}\epsilon(uw)({\mathbf{m}}^{w_\gamma}_{w}-{\mathbf{m}}^{w_{s_i\gamma}}_{w})e^{ut_{-\gamma}(\lambda+\widehat{\rho})},
\end{equation*}
where ${\mathbf{m}}^{w_\gamma}_{w}-{\mathbf{m}}^{w_{s_i\gamma}}_{w}$ are given in Corollaries \ref{cor:diff_m_expl_B}, \ref{cor:diff_m_expl_C}, \ref{cor:diff_m_expl_F}, \ref{cor:diff_m_expl_G}.
\end{thm}
\begin{proof}
In general, we have 
\begin{align*}
\widehat{R}\operatorname{ch}L(\Lambda) &= \sum_{\gamma \in Q^\vee}\sum_{u \in W} \epsilon(uw){\mathbf{m}}^{w_\gamma}_{w}e^{ut_{-\gamma}(\lambda+\widehat{\rho})}\\
&=\frac{1}{2}\sum_{\gamma \in Q^\vee}\sum_{u \in W}\epsilon(uw){\mathbf{m}}^{w_\gamma}_{w}e^{ut_{-\gamma}(\lambda+\widehat{\rho})}-
\frac{1}{2}\sum_{\gamma \in Q^\vee}\sum_{u\in W}\epsilon(uw){\mathbf{m}}^{w_\gamma}_{w}e^{us_it_{-\gamma}(\lambda+\widehat{\rho})}\\
&=\frac{1}{2}\sum_{\gamma \in Q^\vee}\sum_{u \in W}\epsilon(uw){\mathbf{m}}^{w_\gamma}_{w}e^{ut_{-\gamma}(\lambda+\widehat{\rho})}-\frac{1}{2}\sum_{\gamma \in Q^\vee}\sum_{u\in W}\epsilon(uw){\mathbf{m}}^{w_{s_i\gamma}}_{w}e^{ut_{-\gamma}(\lambda+\widehat{\rho})}\\
&=\frac{1}{2}\sum_{\gamma \in Q^\vee}\sum_{u \in W}\epsilon(uw)({\mathbf{m}}^{w_\gamma}_{w}-{\mathbf{m}}^{w_{s_i\gamma}}_{w})e^{ut_{-\gamma}(\lambda+\widehat{\rho})}.\qedhere
\end{align*}
\end{proof}

\begin{example}
Assume that $\mathfrak{g}=\mathfrak{sp}_{2n}$, and $\Lambda=-\Lambda_0$. Note that Corollary~\ref{cor:diff_m_expl_C} can be re-written as:
\[
\mathbf m^{w_\gamma}_{w_{\epsilon_1}}-\mathbf m^{w_{s_0\gamma}}_{w_{\epsilon_1}}=\begin{cases}
1&\text{if $\langle\gamma,\overline\Lambda_1\rangle\ge1$ and $\langle\gamma,\overline\Lambda_n\rangle$ is odd}\\
-1&\text{if $\langle\gamma,\overline\Lambda_1\rangle\le0$ and $\langle\gamma,\overline\Lambda_n\rangle$ is even}\\
0&\text{otherwise}.
\end{cases}
\]
Thus by Theorem~\ref{thm:character-formula} for $w=s_0=w_{\epsilon_1}$ (substituting $-\gamma$ for $\gamma$),
\begin{align*}
\widehat{R}\operatorname{ch}L(\Lambda) &= -\frac{1}{2}\sum_{\gamma \in Q^\vee}\sum_{u \in W}\epsilon(u)({\mathbf{m}}^{w_{-\gamma}}_{w_{\epsilon_1}}-{\mathbf{m}}^{w_{-s_0\gamma}}_{w_{\epsilon_1}})e^{ut_{\gamma}(-\Lambda_0+\widehat{\rho})}\\
&=\frac{1}{2}\sum_{\substack{\langle \gamma, \bar{\Lambda}_1\rangle \geqslant 0\\ \langle \gamma, \bar{\Lambda}_n\rangle\in 2\mathbb{Z}}}\sum_{u \in W}\epsilon(u)e^{ut_\gamma (-\Lambda_0+\widehat{\rho})}-\frac{1}{2}\sum_{\substack{\langle \gamma, \bar{\Lambda}_1\rangle \leqslant -1\\ \langle \gamma, \bar{\Lambda}_n\rangle\in 2\mathbb{Z}+1}}\sum_{u \in W}\epsilon(u)e^{ut_{\gamma} (-\Lambda_0+\widehat{\rho})}\\
&=\frac{1}{2}\sum_{\substack{\langle \gamma, \bar{\Lambda}_1\rangle \geqslant 0\\ \langle \gamma, \bar{\Lambda}_n\rangle\in 2\mathbb{Z}}}\sum_{u \in W}\epsilon(u)e^{ut_\gamma (-\Lambda_0+\widehat{\rho})}+\frac{1}{2}\sum_{\substack{\langle \gamma, \bar{\Lambda}_1\rangle \geqslant 0\\ \langle \gamma, \bar{\Lambda}_n\rangle\in 2\mathbb{Z}}}\sum_{u \in W}\epsilon(u)e^{ut_{s_{\theta^\vee}(\gamma)-\theta^\vee} (-\Lambda_0+\widehat{\rho})}\\
&=\sum_{\langle \gamma, \bar{\Lambda}_1\rangle\ge0,\, \langle\gamma,\bar{\Lambda}_n\rangle \in 2\mathbb{Z}}\epsilon(u)e^{ut_\gamma(-\Lambda_0+\widehat{\rho})}.
\end{align*}
The same formula was obtained in \cite[Theorem 2.2 (a)]{kac-wakimoto} using different methods.
\end{example}

\begin{example}\label{char_G_2_lambda_0}
Assume $\mathfrak{g}=G_2$, and $\Lambda=-\Lambda_0$. We then have 
\begin{equation*}
\widehat{R}\operatorname{ch}L(-\Lambda_0)=-\sum_{\gamma=m\alpha^\vee+n\beta^\vee}\sum_{u \in W}\epsilon(u)\Big(\Big\lfloor\frac{m-1}3\Big\rfloor+\frac{1}{4}\big(m-\mathbf 1_{m\equiv n\bmod2}-\mathbf 1_{n\equiv0\bmod2}\big)\Big)e^{ut_{-\gamma}(-\Lambda_0+\widehat{\rho})}.
\end{equation*}
\end{example}

\begin{example}\label{char_B_3_2lambda_0}
    Assume $\fg=B_3$ and $\Lambda=-2\Lambda_0$. We then have
    \[
    \widehat R\ch L(-2\Lambda_0)=-\sum_{\gamma=a_1\alpha_1^\vee+a_2\alpha_2^\vee+b\beta^\vee}\sum_{u\in W}\epsilon(u)\bigg(\frac14(a_1-1)-\frac14(-1)^{a_1}\mathbf 1_{a_2\not\equiv0\bmod2}\bigg)e^{ut_{-\gamma}(-2\Lambda_0+\widehat\rho)}.
    \]
\end{example}

\appendix

\section{The Equivariant Derived McKay correspondence}\label{appendix}

Let $\Gamma\subset\SL_2(\C)$ be a finite subgroup of order $m$, and let $\tau$ be the tautological $2$-dimensional representation of $\Gamma$. Let $(\bA^2)^{[m]}$ be the Hilbert scheme of $0$-dimensional subschemes $\xi\subset\bA^2$ of length $m$, i.e.,
\[
(\bA^2)^{[m]}\colonequals\{I\subset \C[x,y]:\C[x,y]/I\text{ is $m$-dimensional}\}.
\]
Let $\bA^2/\!/\Gamma$ be the closed subscheme
\[
\bA^2/\!/\Gamma\colonequals\{I\subset\C[x,y]:I\text{ is $\Gamma$-invariant, and $\C[x,y]/I\simeq\C[\Gamma]$ as $\Gamma$-representations}\}.
\]
Then $\bA^2/\!/\Gamma\to\bA^2/\Gamma$ is a resolution of singularities.
Let $\mathcal{K}$ be the tautological $\Gamma$-equivariant vector bundle over $\bA^2/\!/\Gamma$, with fiber over $\xi\subset X$ being $\Gamma(\xi,\cO_\xi)$. 

\begin{prop}(\cite{derived-mckay})\label{prop_E}
Vector bundle $\mathcal{K}$ is a tilting generator, the structure sheaf $\cO_{\bA^2/\!/\Gamma}$ is a direct summand in $\mathcal{K}$, and $\mathcal{K}$ is globally generated. We have $\operatorname{End}(\mathcal{K})=\mathbb{C}[\bA^2] \# \Gamma$.
\end{prop}
\begin{proof}
Authors of \cite{derived-mckay} construct the equivalence $\Psi\colon \cD^\bounded(\Coh(\bA^2/\!/\Gamma))\iso \cD^\bounded(\Coh_\Gamma(\bA^2))$. It is easy to see (using \cite[Proposition 1.5]{derived-mckay}) that $\Psi(\mathcal{K})=\mathbb{C}[\bA^2] \# \Gamma$. It then follows from \cite[Proposition 1.5]{derived-mckay} that $\mathcal{K}$ is tilting and $\operatorname{End}(\mathcal{K})=\mathbb{C}[\bA^2] \# \Gamma$. To see that the structure sheaf $\cO_{\bA^2/\!/\Gamma}$ is a direct summand in $\mathcal{K}$ one can, for example, use the ``quiver'' realization of $\mathcal{K}$ described in Section \ref{quiver_descr_equiv}. Finally, $\mathcal{K}$ is globally generated as it is the quotient of $\mathbb{C}[x,y] \otimes \cO_{\bA^2/\!/\Gamma}$ (here we consider $\mathbb{C}[x,y]$ as an infinite-dimensional vector space over $\mathbb{C}$, compare with the proof of \cite[Lemma 2.1]{derived-mckay}).
\end{proof}

It follows from Proposition \ref{prop_E} that there is a derived equivalence (cf. \cite{derived-mckay})
\begin{equation}\label{derived-mckay}
\cD^\bounded\big(\!\Coh(\bA^2/\!/\Gamma)\big)\iso \cD^\bounded\big(\!\Coh_\Gamma(\bA^2)\big):\cF\mapsto \operatorname{RHom}(\cE,\cF).
\end{equation}
This equivalence is called {\emph{derived McKay equivalence}}.

\subsection{In the language of quivers}\label{quiver_descr_equiv}
We re-cast the equivalence in the language of quivers, where we now assume $\Gamma\ne1$ to avoid technicalities. Let $M(\Gamma)$ be the McKay quiver, i.e., with vertex set $\widehat{I}\colonequals\Irr(\Gamma)$ (where for $i\in \widehat{I}$ we let $\pi_i$ denote the corresponding representation of $\Gamma$), and an edge between $i,j\in \widehat{I}$, denoted $i\sim j$, when $\pi_i$ appears in $\pi_j\otimes\tau$. For each $i\in \widehat{I}$, let $d_i\colonequals\dim(\pi_i)$.

Recall \cite{w-wang} that $\bA^2/\!/\Gamma$ admits a quiver description: there is an isomorphism
\begin{equation}\label{quiver-desc}
\bA^2/\!/\Gamma\simeq\big\{f\colon\tau\to_\Gamma\End(\C[\Gamma]),v\in\C[\Gamma]:[f(\tau),f(\tau)]=0,\{f(\tau)^nv\}_{n\geqslant0}\text{ generates }\C[\Gamma]\big\}/\C[\Gamma]^\times,
\end{equation}
where $f\colon \tau\to_\Gamma\End(\C[\Gamma])$ means $f$ is a $\Gamma$-equivariant homomorphism $\tau\to\End(\C[\Gamma])$. It can equivalently be considered a homomorphism $f\colon \C[\Gamma]\otimes \tau\to\C[\Gamma]$, and the condition that $[f(\tau),f(\tau)]=0$ becomes equivalent to the commutativity of
\begin{equation}\label{commutative-diagram} \begin{tikzcd}
\C[\Gamma]\otimes\tau\otimes\tau\arrow{r}{\text{swap}} \arrow[swap]{d}{f\otimes\tau} &\C[\Gamma]\otimes\tau\otimes\tau\arrow{r}{f\otimes\tau}& \C[\Gamma]\otimes\tau \arrow{d}{f} \\
\C[\Gamma]\otimes\tau\arrow{rr}{f}&& \C[\Gamma].
\end{tikzcd}
\end{equation}
Now, there is an isomorphism:
\[
\hom_\Gamma(\C[\Gamma]\otimes\tau,\C[\Gamma])\simeq\bigoplus_{\substack{(i,j)\in I\times I\\i\sim j}}\hom(\C^{d_i},\C^{d_j}),
\]
which allows us to re-write \eqref{quiver-desc} is isomorphic to:
\begin{align*}
\mathcal Q(M(\Gamma))\colonequals\Big\{&(x_{ij})\in\prod_{i\sim j}\hom(\C^{d_i},\C^{d_j}):\text{ for all $i\in \widehat{I}$, }\sum_{j\sim i}x_{ij}x_{ji}=0,\\
&\text{there are no subspaces $0\ne V_i\subsetneq\C^{d_i}$ such that }x_{ij}V_i\subset V_j\Big\}/\prod_{i\in \widehat{I}}\GL_{d_i}(\C),
\end{align*}
which notably only depends on the quiver $M(\Gamma)$. Here $\sum_{j\sim i}x_{ij}x_{ji}=0$ is equivalent to the commutativity of diagram~\eqref{commutative-diagram}. Now the quiver variety carries a tautological bundle of quiver representations $(\cV_i)_{i \in \widehat{I}}$ of dimension $(d_i)_{i\in \widehat{I}}$, with homomorphisms $x_{ij}\colon\cV_i\to\cV_j$. The tautological vector bundle $\mathcal{K}$ on $\bA^2/\!/\Gamma$ becomes $\bigoplus_{i\in \widehat{I}}\End(\cV_i)$.

Moreover, as in \cite[\S3.4]{derived-mckay} there is an equivalence
\[
\Coh_\Gamma(\bA^2)\simeq\Rep\big(M(\Gamma)\big),
\]
where $\Rep(M(\Gamma))$ is the category of double quiver representations, i.e., a family of vector spaces $(V_i)_{i\in \widehat{I}}$ and for any edge $(i,j)$ a homomorphism $x_{ij}\colon V_i\to V_j$ such that for any vertex $i\in \widehat{I}$, $\sum_{j\sim i}x_{ji}x_{ij}=0$. Now, \eqref{derived-mckay} can be re-formulated as the equivalence:
\[
\cD^\bounded\big(\!\Coh(\mathcal Q(M))\big)\simeq \cD^\bounded\big(\!\Rep(M)\big):\cF\to(\hom(\cV_i,\cF))_i,(-\circ x_{ij})_{i\sim j}
\]
where $M$ is a quiver which is an affine Dynkin diagram of type ADE.

Let $G$ be a simple Lie group and let $e\in\fg$ be a subregular nilpotent. The corresponding Slodowy variety $\widetilde{S}$ is isomorphic to $\bA^2/\!/\Gamma$. It follows from the above that $\mathcal{E}$ is $\Cent_{G}(\varphi)$-equivariant so the McKay equivalence~\eqref{derived-mckay} can be upgraded to the \emph{equivariant derived McKay equivalence}.
\begin{thm}
 There is an equivalence
\begin{align*}
\cD^\bounded(\Coh^{\Cent(e)}(\widetilde S))&\simeq \cD^\bounded(\Rep^{\Cent(e)}(M)),
\end{align*}
where $M$ is the affine Dynkin quiver of the unfolding of $G^\vee$. 
\end{thm}
\begin{proof}
The same argument as in the proof of \cite[Proposition 5.2.1]{bezrukavnikov-mirkovic} works.
\end{proof}

\begin{remark}
It follows from Proposition \ref{prop:descr_irred_t_str_B_e} together with Proposition \ref{prop_E} that the standard $t$-structure on the right-hand side transfers to the exotic $t$-structure on the left-hand side. 
\end{remark}

\section{Very distinguished nilpotents and quasi-polynomial multiplicities}\label{appendix_C}

\begin{center}
    Roman Bezrukavnikov, Vasily Krylov, and Kenta Suzuki
\end{center}
\vspace{0.2cm}

We first recall the following definition. 

\begin{defn}\label{very_dist_def}
An element $e\in \mathcal{N}$ is called {\em{distinguished}} if 
it is not contained in a proper Levi subalgebra.
An element $e\in \mathcal{N}$ is called {\em very distinguished} if every element $e'\in \mathcal{N}$ such that $e \in \overline{\mathbb{O}}_{e'}$
is distinguished. 
\end{defn}

The group $\Cent_e$
is finite if $e$ is distinguished. One example of a very distinguished nilpotent $e$ is a subregular nilpotent in types $D$, $E$. It was shown in \cite[Equation (31)]{BKK} that for such $e$ and $w_i \in c$ (notations of loc. cit.), we have 
\begin{equation*}
{\bf{m}}^{w_\gamma}_{w_i}=\Big\langle \Lambda_i,-\gamma+\frac{\|\gamma\|^2}{2}K \Big\rangle.  
\end{equation*}
It follows that $\gamma \mapsto {\bf{m}}^{w_\gamma}_{w_i}$ is the {\emph{quadratic}} function on $\gamma$ in these cases.

Another example of a very distinguished nilpotent $e$ is a subregular nilpotent in types $B,F_4,G_2$. It follows from Theorems \ref{KL-polynomial-computation-B}, \ref{thm_KL_type_F}, \ref{thm_KL_type_G} that  the function $\gamma \mapsto {\bf{m}}_{w_\nu}^{w_\gamma}$ will be a {\em{quasi-polynomial}} in these cases.
The goal of this section is to prove that in general, for $w_\nu \in c$, corresponding to a very distinguished nilpotent, the function $\gamma \mapsto {\bf{m}}_{w_\nu}^{w_\gamma}$ is a quasi-polynomial and we can estimate its period and degree.

\begin{remark}
Note that if $c$ does not correspond to a very distinguished nilpotent, then the function $\gamma \mapsto {\bf{m}}_w^{w_\gamma}$ may not be a quasi-polynomial. This, for example, follows from the computation of the functions ${\bf{m}}_{w_i}^{w_\gamma}, i \in \Z$, for subregular nilpotent in type $A$ case (see \cite[Equation (38)]{BKK}), see also Theorem \ref{thm_KL_type_C}. 
\end{remark}

Let $e'\in {\mathbb{O}}_c$ be a distinguished nilpotent. Our goal is to describe the eigenvalues of $\cO_{\widetilde{\cN}}(\gamma)$, acting on the finite dimensional 
vector space $K^{\Cent_{e'}}(\cB_{e'})$. We will see that they are roots of unity of order equal to the order of an element
in $\Cent_{e'}$.  We will also estimate the sizes of the corresponding Jordan blocks. Our main technical tool is the localization theorem in $K$-theory.

We start with the following general lemma. Let $H$ be an algebraic group acting on a variety $X$. Abusing notations, from now on we will denote by $K_H(X)$ the {\emph{complexified}} $H$-equivariant $K$-group of $X$.

Pick a semisimple element $\zeta \in H$ and set $Z:=Z_H(\zeta)$. Let $\gamma \subset H$ be the conjugacy class of the element $\zeta$. Let $i\colon X^\zeta \hookrightarrow X$ be the inclusion of the fixed point locus of $\zeta$. In \cite[Theorem 4.3]{edidin-graham} the authors construct the isomorphism of specializations of (complexified) $K$-groups:
\begin{equation*}
i_!\colon K_Z(X^\zeta)_\zeta \iso K_H(X)_{\gamma}.
\end{equation*}

\begin{lemma}\label{lem:loc_comm_line}
If $\mathcal{L}$ is a line bundle on $X$ then $[\mathcal{L}] \otimes i_!x = i_!([i^*\mathcal{L}] \otimes x)$ for every $x \in K_Z(X^\zeta)_\zeta$.
\end{lemma}
\begin{proof}
Step $1$. Assume first that $X$ is smooth. Then the claim follows from \cite[Theorem (b)]{edidin-graham} together with the fact that $i^!([\mathcal{L}] \otimes y)=[i^*\mathcal{L}] \otimes i^!y$, where $i^!$ (as in \cite[Section 4.2]{edidin-graham}) is the composition of the restriction map and $i^*$.

Step $2$. In general, as in \cite[Section 6.1]{edidin-graham} consider an equivariant envelope $p\colon \widetilde{X} \rightarrow X$ and let $q\colon \widetilde{X}^\zeta \rightarrow X^\zeta$ be the restriction of $p$ to the fixed points of $\zeta$.
It follows from \cite[Theorem 4.3]{edidin-graham} that $i_!q_*=p_*\widetilde{i}_!$. We then have for every $\widetilde{x} \in K_Z(\widetilde{X}^{\zeta})_{\zeta}$ (using the projection formula together with Step $1$):
\begin{align*}
[\mathcal{L}] \otimes i_!(q_*\widetilde{x})&=[\mathcal{L}] \otimes p_*\widetilde{i}_!\widetilde{x}=p_*(\widetilde{i}_!([\widetilde{i}^*p^*\mathcal{L}] \otimes \widetilde{x}))\\ &= p_*(\widetilde{i}([q^*i^*\mathcal{L}] \otimes \widetilde{x}))=i_!q_*([q^*i^*\mathcal{L}] \otimes \widetilde{x})=i_!([i^*\mathcal{L}] \otimes q_*\widetilde{x}).\qedhere
\end{align*}
Surjectivity of $q_*$ (see \cite[Proposition 6.7]{edidin-graham}) finishes the proof.
\end{proof}

We are now ready to describe generalized eigenvalues and eigenspaces of $\cO_{\cB_{e'}}(\gamma)$ acting on $K^{\Cent_{e'}}(\cB_{e'})$.
To each conjugacy class $[\zeta]\in \Cent_{e'}/_{\operatorname{ad}} \Cent_{e'}$ and a connected component $D \subset \cB^{\zeta}_{e'}$, we associate the number $\chi_{e',\zeta,D}$ by which $\zeta$ acts on $\cO_{\cB_{e'}}(\gamma)|_{D}$. Note that $\chi_{e',\zeta,D}^l=1$ if $\zeta^l=1$. Let $\Cent_{e'}(\zeta)\subset \Cent_{e'}$ be the centralizer of $\zeta$ in $\Cent_{e'}$.
\begin{lemma}\label{dist_lem}  
We have the decomposition 
$K^{\Cent_{e'}}(\cB_{e'})=\bigoplus_{[\zeta] \in \Cent_{e'}/_{\operatorname{ad}}\Cent_{e'}, D \subset \cB^{\zeta}_{e'}} K(D)^{\Cent_{e'}(\zeta)}$. 
After this identification, the action of $\cO_{\widetilde{\cN}}(\gamma)$ on $K(D)^{\Cent_{e'}(\zeta)}$  is given by $\chi_{e',\gamma,D} \cdot (\cO_{D}(\gamma) \otimes - )$. In particular, $K(D)^{\Cent_{e'}(\zeta)}$ lies in the generalized eigenspace of $\cO_{\widetilde{\cN}}(\gamma) \otimes -$ with eigenvalue $\chi_{e',\gamma,D}$, and sizes of the corresponding Jordan blocks do not exceed $\dim D + 1 \leqslant \dim \cB_{e'}^{\zeta}+1$. 
\end{lemma}
\proof
The localization theorem for finite groups  (see, for example, \cite[Theorem 4.3]{edidin-graham}) implies
\begin{equation}\label{loc_gr}
K^{\Cent_{e'}}(\cB_{e'})=\bigoplus_{[\zeta] \in \Cent_{e'}/_{\operatorname{ad}}\Cent_{e'}} K^{\Cent_{e'}}(\cB_{e'})_{[\zeta]}=\bigoplus_{[\zeta] \in \Cent_{e'}/_{\operatorname{ad}}\Cent_{e'}} K^{\Cent_{e'}(\zeta)}(\cB_{e'}^{\zeta})_{\zeta}=\bigoplus_{[\zeta] \in \Cent_{e'}/_{\operatorname{ad}}\Cent_{e'}} K(\cB_{e'}^{\zeta})^{\Cent_{e'}(\zeta)}.
\end{equation}
The first identification is clear (use that $\Cent_{e'}$ is finite), the second identification is given by $i_!$, where $i\colon \cB_{e'}^{\zeta} \hookrightarrow \cB_{e'}$ is the embedding. 
In the last identification we use the isomorphism $K^{\Cent_{e'}(\zeta)}(\cB_{e'}^\zeta)_{\zeta} \simeq K^{\Cent_{e'}(\zeta)}(\cB_{e'}^\zeta)_{1}$ that can be described as follows (see \cite[Section 5.2]{edidin-graham}): starting from a $\Cent_{e'}(\zeta)$-equivariant sheaf $\cF$, the action of $\zeta$ induces the decomposition $\cF=\bigoplus_{\chi \in \mathbb{C}^\times}\cF_\chi$ and the identification above sends $[\cF]$ to $\sum_{\chi \in \mathbb{C}^\times}\chi [\cF_\chi]$. We also use the identification $K^{\Cent_{e'}(\zeta)}(\cB_{e'}^\zeta)_{1} \simeq K(\cB_{e'}^\zeta)^{\Cent_{e'}(\zeta)}$ (which holds as $\Cent_{e'}(\zeta)$ is finite).
It now follows from Lemma \ref{lem:loc_comm_line} that after the identification (\ref{loc_gr}), the action of $\cO_{\widetilde{\cN}}(\gamma)$ on $K(D)^{\Cent_{e'}(\zeta)}$  is given by $\chi_{e',\gamma,D} \cdot (\cO_{D}(\gamma) \otimes - )$ (twist by $\chi_{e',\gamma,D}$ appears because of the identification $K^{\Cent_{e'}(\zeta)}(\cB_{e'}^\zeta)_{\zeta} \simeq K^{\Cent_{e'}(\zeta)}(\cB_{e'}^\zeta)_{1}$ that we use). 
Note now that the operator $\cO_{D}(\gamma) \otimes - $ is unipotent so $\chi_{e',\gamma,D}$ is the unique  eigenvalue of $\cO_{\widetilde{\cN}}(\gamma)$ restricted to $K(D)^{\Cent_{e'}(\zeta)}$. Moreover, $[\cO_D(\gamma)-\cO_D] \otimes -$ is equal to zero being raised to the $(\dim D+1)$-th power (induction on $\dim D$), so all of the Jordan blocks of $\cO_{\widetilde{\mathcal{N}}}(\gamma)$ acting on $K(D)^{\Cent_{e'}(\zeta)}$ do not exceed $\dim D+1$.
\qed

For a distinguished nilpotent $e'$ let $d_{e',\chi}$ be the maximum of $\dim D+1$, where $D$ is a connected component of $\cB_{e'}^{\zeta}$ such that $\chi_{e',\zeta,D}=\chi$.

Let $e \in \mathcal{N}$ be a very distinguished nilpotent. Set $U\colonequals\bigsqcup_{e \in \overline{\mathbb{O}}_{e'}} \mathbb{O}_{e'} \subset \mathcal{N}$, this is an open $G^\vee$-invariant subset of $\mathcal{N}$.

\begin{cor}\label{cor_eigen_on_k_theory}
Decomposition into generalized eigenspaces of $Q^\vee$ acting on $K^{G^\vee}\!({\widetilde{U}})$ can be described as follows: possible eigenvalues $\chi$ are $\chi_{e',\zeta,D}$ for some $e' \in \overline{\mathbb{O}}_e$, $\zeta \in \Cent_{e'}$, $D \subset \cB_{e'}^{\zeta}$. Dimension of the corresponding eigenspace is equal to $\sum_{e',\zeta,D}\dim K(D)^{\Cent_{e'}(\zeta)}$, where the sum runs over all $e' \in \overline{\mathbb{O}}_e$, $[\zeta] \in \Cent_{e'}/_{\operatorname{ad}}\Cent_{e'}$, $D \subset \cB_{e'}^{\zeta}$ such that $\chi_{e',\zeta,D}=\chi$.  The sizes of the Jordan normal blocks corresponding to $\chi$ do not exceed $d_{U, \chi} := \sum_{e \in \overline{\mathbb{O}}_{e'}}d_{e',\chi}$. 
\end{cor}
\begin{proof}
Note that there exists a filtration on the $Q^\vee$-module $K^{G^\vee}\!({\widetilde{U}})$ labeled by $\mathbb{O}_{e'}$ such that $e \in \overline{\mathbb{O}}_{e'}$ and with associated graded being isomorphic to $\bigoplus_{e' \in \overline{\mathbb{O}}_e}K^{\Cent_{e'}}(\cB_{e'})$ (see \cite[Section 4.2]{BKK}). Claim follows from Lemma \ref{dist_lem}.    
\end{proof}

We identify $\Z^{\oplus r} \iso Q^\vee$ via $(a_1,\ldots,a_r) \mapsto \sum_{i=1}^r a_i \alpha_i^{\vee}$.
\begin{defn}\label{def_quasi_poly}
A function $f\colon Q^\vee=\Z^{\oplus r} \rightarrow \Z$ is a \emph{quasi-polynomial with period $l$} if  
\begin{equation*}
f(a_1,\ldots,a_r)=\sum_{d_1,\ldots,d_r \in \Z_{\geqslant 0}}c_{d_1,\ldots,d_r}\!(a_1,\ldots,a_r)a_1^{d_1}\ldots a_r^{d_r},
\end{equation*}
where $c_{d_1,\ldots,d_r}$ are almost all equal to zero and descend to the functions on $(\Z/l\Z)^{\oplus r}$.
A function ${\mathbf{m}}\colon Q^\vee \rightarrow \Z$ is a quasi-polynomial if it is a quasi-polynomial with period $l$ for some $l \in \mathbb{Z}_{\geqslant 0}$.
\end{defn}

Assume that the group $Q^\vee$ acts on a finite dimensional vector space $V$. Assume also that there exists a character $\chi\colon Q^\vee \rightarrow \mu_l$ such that the operator $(\gamma-\chi(\gamma))^p=0$ for some $p \in \mathbb{Z}_{\geqslant 0}$.
\begin{lemma}\label{lemma_lin_alg}
For every $v \in V$, $u \in V^*$, the function $Q^\vee \ni \gamma \mapsto \langle u, \gamma v\rangle$ is a quasi-polynomial with period $l$ and degree $<p$.  
\end{lemma}
\begin{proof}
Pick a basis $v_1,\ldots,v_n$ of $V$ in which $Q^\vee$ acts via upper triangular operators. Recall the identification $Q^\vee=\Z^{\oplus r}$, for $k \in \{1,\ldots,r\}$ let $A_k$ be the operator by which $\alpha_k^{\vee}$ acts on $V$. The matrix of $A_k$ in the basis $v_i$ has the following form:
\begin{equation*}
A_k = \begin{pmatrix}
 \chi(\alpha_k^{\vee}) & * & \ldots & * \\
 0 & \chi(\alpha_k^{\vee}) & \ldots & * \\
 \vdots & \vdots & \ddots & * \\
 0 & 0 & \ldots & \chi(\alpha_k^{\vee})
\end{pmatrix}=\chi(\alpha_k^{\vee})\begin{pmatrix}
 1 & * & \ldots & * \\
 0 & 1 & \ldots & * \\
 \vdots & \vdots & \ddots & * \\
 0 & 0 & \ldots & 1
\end{pmatrix}.
\end{equation*}
It remains to check that the entries of the matrix $A_1^{a_1}\ldots A_r^{a_r}$ are quasi-polynomials with period $l$ and degree $<p$. 
Set $B_k\colonequals\frac{1}{\chi(\alpha_k^{\vee})}A_k$. Since  $\chi(\alpha_k^{\vee})^l=1$ we see that it is enough to show that the entries of the matrix $B_1^{a_1}\ldots B_r^{a_r}$ are polynomials of degree $<p$. This is standard.
\end{proof}

\begin{thm}\label{quasi_pol_dist}
Suppose that $w_\nu\in c$ and ${\mathbb{O}}_c$ consists of very distinguished elements. 
Then the function $\gamma\mapsto {\bf{m}}_{w_\nu}^{w_\gamma}$ is a quasi-polynomial. It has the form 
\begin{equation*}
{\bf{m}}_{w_\nu}^{w_\gamma}=\sum_{\chi}f_\chi(\gamma),
\end{equation*}
where $f_{\chi}(\gamma)$ is a quasi-polynomial with period $l=\mathrm{order}(\chi)$, degree of $f$ is at most $d_{U,\chi}-1$.
\end{thm}

\proof 

Note that ${\bf{m}}_{w_\nu}^{w_\gamma}$ is a matrix coefficient of the operator $\cO_{\widetilde{\cN}}(\gamma)\otimes-$, acting on $K^{G^\vee}\!(\widetilde{U})$.
The statement now follows from Corollary \ref{cor_eigen_on_k_theory} and Lemma \ref{lemma_lin_alg}. \qed

\begin{remark} Recall that the set of $\chi$ appearing in Theorem \ref{quasi_pol_dist} runs through all $\chi_{e',\zeta,D}$, where $e' \in \overline{\mathbb{O}}_e$, $\zeta \in \Cent_{e'}$, $D \subset \cB_{e'}^{\zeta}$, and $\chi_{e',\zeta,D}$ is the number by which $\zeta$ acts on $\cO_{\cB_{e'}}(\gamma)|_D$. Recall also that $d_{U,\chi}=\sum_{e \in \overline{\mathbb{O}}_{e'}}d_{e',\chi}$ with $d_{e',\chi}$ being the maximum of $\dim D+1$ over all connected components $D \subset \cB_{e'}^{\zeta}$ such that $\chi=\chi_{e',\zeta,D}$.   Lemma \ref{dist_lem} gives an effective bound 
on degrees and periods of the quasi-polynomials $f_\chi$ (let us recall that the order of $\chi=\chi_{e',\zeta,D}$ divides the order of $\zeta$ that in turn divides $|\Cent_{e'}|$). 
These may not be optimal. 
\end{remark}

\begin{example}
To apply Theorem \ref{quasi_pol_dist} one needs an effective way of finding ${\mathbb{O}}_c$
given $w\in c$. While this is not easy in general, recall that the cell $c$ such that ${\mathbb{O}}_c$
is subregular consists of elements $w\ne 1$ with a unique minimal decomposition.
A similar (but more complicated) description of the next case, which includes most
very distinguished examples, should follow from \cite{green-xu}.
\end{example}

\bibliographystyle{amsalpha}
\bibliography{bibfile}

\end{document}